\documentclass[11pt]{article}

\usepackage[T1]{fontenc}
\usepackage[utf8]{inputenc}
\usepackage[english]{babel}
\usepackage{amsmath,amsfonts,amssymb,amsthm,mathrsfs}
\usepackage[colorlinks=true]{hyperref}
\hypersetup{pdftex, colorlinks=true, linkcolor=vrr, citecolor=black, filecolor=vrr,urlcolor=vrr}
\usepackage[font=small, labelfont={bf,small}]{caption}
\usepackage{textcase}
\usepackage[longnamesfirst, authoryear, comma]{natbib}
\usepackage{titlesec} 
\usepackage{mathtools}
\usepackage{etoolbox}
\usepackage{cancel}
\usepackage[normalem]{ulem}
\usepackage{booktabs, multirow}
\usepackage{tikz}
\usetikzlibrary {positioning}

\usepackage[
  textwidth=33pc,
  textheight=48pc,
]{geometry}

\usepackage{color}

\definecolor{vry}{RGB}{253, 231, 37}

\definecolor{vrg}{RGB}{94,201,98}
 
\definecolor{vrdg}{RGB}{33, 145, 140}

\definecolor{vrb}{RGB}{59,82,139}

\definecolor{vrp}{RGB}{68,1,84}

\definecolor{vro}{RGB}{249,142,9}

\definecolor{vrr}{RGB}{188,55,84}

\definecolor{vrnb}{RGB}{13,8,135}

\makeatletter

\numberwithin{equation}{section}
\allowdisplaybreaks

\def\@noindentfalse{\global\let\if@noindent\iffalse}
\def\@noindenttrue {\global\let\if@noindent\iftrue}
\def\@aftertheorem{%
\@noindenttrue
\everypar{%
\if@noindent%
\@noindentfalse\clubpenalty\@M\setbox\z@\lastbox%
\else%
\clubpenalty \@clubpenalty\everypar{}%
\fi}}

\theoremstyle{plain}
\newtheorem{theorem}{Theorem}[section]
\AfterEndEnvironment{theorem}{\@aftertheorem}

\AfterEndEnvironment{definition}{\@aftertheorem}
\newtheorem{lemma}[theorem]{Lemma}
\AfterEndEnvironment{lemma}{\@aftertheorem}

\AfterEndEnvironment{claim}{\@aftertheorem}
\newtheorem{proposition}[theorem]{Proposition}
\AfterEndEnvironment{proposition}{\@aftertheorem}

\AfterEndEnvironment{corollary}{\@aftertheorem}

\theoremstyle{definition}
\newtheorem{remark}[theorem]{Remark}
\AfterEndEnvironment{remark}{\@aftertheorem}

\AfterEndEnvironment{example}{\@aftertheorem}
\AfterEndEnvironment{proof}{\@aftertheorem}

\titleformat{name=\section}
{\center\bf}{\thesection\kern1em}{0pt}{}
\titleformat{name=\section, numberless}
{\center\small\sc}{}{0pt}{\MakeTextUppercase}

\titleformat{name=\subsection}
{\bf\mathversion{bold}}{\thesubsection\kern1em}{0pt}{}
\titleformat{name=\subsection, numberless}
{\bf\mathversion{bold}}{}{0pt}{}

\titleformat{name=\subsubsection}
{\normalsize\itshape}{\thesubsubsection\kern1em}{0pt}{}
\titleformat{name=\subsubsection, numberless}
{\normalsize\itshape}{}{0pt}{}

\def\beqn#1{\begin{equation}#1\end{equation}}

\def\be#1{\begin{equation*}#1\end{equation*}}
\def\ben#1{\begin{equation}#1\end{equation}}
\def\bes#1{\begin{equation*}\begin{split}#1\end{split}\end{equation*}}
\def\besn#1{\begin{equation}\begin{split}#1\end{split}\end{equation}}
\def\bs#1{\begin{split}#1\end{split}}
\def\ba#1{\begin{align*}#1\end{align*}}
\def\ban#1{\begin{align}#1\end{align}}
\def\bg#1{\begin{gather*}#1\end{gather*}}
\def\bgn#1{\begin{gather}#1\end{gather}}
\def\bmn#1{\begin{multline}#1\end{multline}}

\usepackage[linewidth=1pt]{mdframed}
\usepackage[textwidth=1.65in]{todonotes}
\let\@@todo\getsdo
\def\getsdo#1{\@@todo[color=red,backgroundcolor=red!10,size=\tiny]{#1}}

\usepackage{delimset}
\def\given{\mskip 0.5mu plus 0.25mu\vert\mskip 0.5mu plus 0.15mu}
\newcounter{bracketlevel}%
\def\@bracketfactory#1#2#3#4#5#6{%
\expandafter\def\csname#1\endcsname##1{%
\global\advance\c@bracketlevel 1\relax%
\global\expandafter\let\csname @middummy\alph{bracketlevel}\endcsname\given%
\global\def\given{\mskip#5\csname#4\endcsname\vert\mskip#6}\csname#4l\endcsname#2##1\csname#4r\endcsname#3%
\global\expandafter\let\expandafter\given\csname @middummy\alph{bracketlevel}\endcsname%
\global\advance\c@bracketlevel -1\relax%
}%
}
\def\bracketfactory#1#2#3{%
\@bracketfactory{#1}{#2}{#3}{relax}{0.5mu plus 0.25mu}{0.5mu plus 0.15mu}
\@bracketfactory{b#1}{#2}{#3}{big}{1mu plus 0.25mu minus 0.25mu}{0.6mu plus 0.15mu minus 0.15mu}
\@bracketfactory{bb#1}{#2}{#3}{Big}{2.4mu plus 0.8mu minus 0.8mu}{1.8mu plus 0.6mu minus 0.6mu}
\@bracketfactory{bbb#1}{#2}{#3}{bigg}{3.2mu plus 1mu minus 1mu}{2.4mu plus 0.75mu minus 0.75mu}
\@bracketfactory{bbbb#1}{#2}{#3}{Bigg}{4mu plus 1mu minus 1mu}{3mu plus 0.75mu minus 0.75mu}
}
\bracketfactory{clc}{\lbrace}{\rbrace}
\bracketfactory{clr}{(}{)}
\bracketfactory{cls}{[}{]}
\bracketfactory{abs}{\lvert}{\rvert}
\bracketfactory{norm}{\Vert}{\Vert}
\bracketfactory{floor}{\lfloor}{\rfloor}
\bracketfactory{ceil}{\lceil}{\rceil}
\bracketfactory{angle}{\langle}{\rangle}

\let\original@left\left
\let\original@right\right
\renewcommand{\left}{\mathopen{}\mathclose\bgroup\original@left}
\renewcommand{\right}{\aftergroup\egroup\original@right}

\newcounter{ctr}\loop\stepcounter{ctr}\edef\X{\@Alph\c@ctr}%
\expandafter\edef\csname s\X\endcsname{\noexpand\mathscr{\X}}
\expandafter\edef\csname c\X\endcsname{\noexpand\mathcal{\X}}
\expandafter\edef\csname b\X\endcsname{\noexpand\boldsymbol{\X}}
\expandafter\edef\csname I\X\endcsname{\noexpand\mathbb{\X}}
\ifnum\thectr<26\repeat

\def\^#1{\relax\ifmmode {\mathaccent"705E #1} \else {\accent94 #1}\fi}
\def\~#1{\relax\ifmmode {\mathaccent"707E #1} \else {\accent"7E #1}\fi}
\edef\-#1{\relax\noexpand\ifmmode {\noexpand\bar{#1}} \noexpand\else \-#1\noexpand\fi}

\def\sump{\sideset{}{'}\sum}
\def\tsump{\sum'}
\renewcommand{\leq}{\leqslant}
\renewcommand{\geq}{\geqslant}
\renewcommand{\phi}{\varphi}
\newcommand{\eps}{\varepsilon}
\newcommand{\D}{\Delta}

\newcommand{\I}{\mathop{{}\mathrm{I}}\mathopen{}}
\newcommand{\dtv}{\mathop{d_{\mathrm{TV}}}\mathopen{}}
\newcommand{\hence}{\mathrel{\Rightarrow}}
\newcommand{\longto}{\longrightarrow}

\newcommand{\R}{\mathbb{R}}
\newcommand{\N}{\mathbb{N}}

\def\bt{\boldsymbol{t}}
\def\bu{\boldsymbol{u}}

\let\@IE\IE\let\IE\undefined
\newcommand{\IE}{\mathop{{}\@IE}\mathopen{}}
\let\@IP\IP\let\IP\undefined
\newcommand{\IP}{\mathop{{}\@IP}\mathopen{}}
\newcommand{\PP}{\IP}
\newcommand{\EE}{\IE}
\newcommand{\dd}[1]{\mathop{{}\mathrm{d}#1}}
\newcommand{\ee}{\mathop{{}\mathrm{e}}\nolimits}
\newcommand{\Var}{\mathop{\mathrm{Var}}}
\newcommand{\Cov}{\mathop{\mathrm{Cov}}}
\newcommand{\Be}{\mathop{\mathrm{Bernoulli}}}
\newcommand{\Bi}{\mathop{\mathrm{Binomial}}}
\newcommand{\Po}{\mathop{\mathrm{Poisson}}}
\newcommand{\bigo}{\mathop{{}\mathrm{O}}\mathopen{}}

\newcommand{\law}{\mathop{{}\sL}\mathopen{}}

\def\mel{\MoveEqLeft}

\DeclareMathOperator{\inj}{inj}

\newcommand{\noverlap}{\nu}
\newcommand{\sub}{\mathrm{sub}}
\newcommand{\bbs}{\boldsymbol{s}}
\newcommand{\dm}{\mathop{d_{\widetilde{\mathrm{m}}}}}
\newcommand{\dMZ}{\mathop{d_{\mathrm{m}}}}
\newcommand{\dsub}{\mathop{d_{\mathrm{sub}}}}
\newcommand{\merge}[1]{\mathop{\otimes}_{#1}}
\renewcommand{\cite}[2][]{\citet[#1]{#2}}

\newsavebox{\bvDisp}
\newsavebox{\bvText}
\newsavebox{\bvScript}
\newsavebox{\bvScriptScript}

\savebox{\bvDisp}{%
  \begin{tikzpicture}[scale=.20]
    \node at (0,0) [draw,circle,fill=black,inner sep=1.2pt] {};
  \end{tikzpicture}%
}

\savebox{\bvText}{%
  \begin{tikzpicture}[scale=.20]
    \node at (0,0) [draw,circle,fill=black,inner sep=1.2pt] {};
  \end{tikzpicture}%
}

\savebox{\bvScript}{%
  \begin{tikzpicture}[scale=.16]
    \node at (0,0) [draw,circle,fill=black,inner sep=1.0pt] {};
  \end{tikzpicture}%
}

\savebox{\bvScriptScript}{%
  \begin{tikzpicture}[scale=.13]
    \node at (0,0) [draw,circle,fill=black,inner sep=0.8pt] {};
  \end{tikzpicture}%
}

\newcommand{\blackvertex}{%
  \mathchoice
    {\raisebox{0.1ex}{\usebox{\bvDisp}}}
    {\raisebox{0.1ex}{\usebox{\bvText}}}
    {\raisebox{0.08ex}{\usebox{\bvScript}}}
    {\raisebox{0.07ex}{\usebox{\bvScriptScript}}}
}
\newcommand{\bv}{\blackvertex}

\newsavebox{\wvDisp}
\newsavebox{\wvText}
\newsavebox{\wvScript}
\newsavebox{\wvScriptScript}

\savebox{\wvDisp}{%
  \begin{tikzpicture}[scale=.20]
    \node at (0,0) [draw,circle,fill=white,inner sep=1.2pt] {};
  \end{tikzpicture}%
}
\savebox{\wvText}{%
  \begin{tikzpicture}[scale=.20]
    \node at (0,0) [draw,circle,fill=white,inner sep=1.2pt] {};
  \end{tikzpicture}%
}
\savebox{\wvScript}{%
  \begin{tikzpicture}[scale=.13]
    \node at (0,0) [draw,circle,fill=white,inner sep=1.0pt] {};
  \end{tikzpicture}%
}
\savebox{\wvScriptScript}{%
  \begin{tikzpicture}[scale=.10]
    \node at (0,0) [draw,circle,fill=white,inner sep=0.8pt] {};
  \end{tikzpicture}%
}

\newcommand{\whitevertex}{%
  \mathchoice
    {\raisebox{0.1ex}{\usebox{\wvDisp}}}
    {\raisebox{0.1ex}{\usebox{\wvText}}}
    {\raisebox{0.08ex}{\usebox{\wvScript}}}
    {\raisebox{0.07ex}{\usebox{\wvScriptScript}}}
}
\newcommand{\wv}{\whitevertex}

\newsavebox{\bbeDisp}
\newsavebox{\bbeText}
\newsavebox{\bbeScript}
\newsavebox{\bbeScriptScript}

\savebox{\bbeDisp}{%
  \begin{tikzpicture}[scale=0.20]
    \node at (0,0) [draw,circle,fill=black,inner sep=1.2pt] (A) {};
    \node at (0.8,1) [draw,circle,fill=black,inner sep=1.2pt] (B) {};
    \draw (A) -- (B);
  \end{tikzpicture}%
}
\savebox{\bbeText}{%
  \begin{tikzpicture}[scale=0.20]
    \node at (0,0) [draw,circle,fill=black,inner sep=1.2pt] (A) {};
    \node at (0.8,1) [draw,circle,fill=black,inner sep=1.2pt] (B) {};
    \draw (A) -- (B);
  \end{tikzpicture}%
}
\savebox{\bbeScript}{%
  \begin{tikzpicture}[scale=0.13]
    \node at (0,0) [draw,circle,fill=black,inner sep=0.8pt] (A) {};
    \node at (0.8,1) [draw,circle,fill=black,inner sep=0.8pt] (B) {};
    \draw (A) -- (B);
  \end{tikzpicture}%
}
\savebox{\bbeScriptScript}{%
  \begin{tikzpicture}[scale=0.10]
    \node at (0,0) [draw,circle,fill=black,inner sep=0.6pt] (A) {};
    \node at (0.8,1) [draw,circle,fill=black,inner sep=0.6pt] (B) {};
    \draw (A) -- (B);
  \end{tikzpicture}%
}

\newsavebox{\bweDisp}
\newsavebox{\bweText}
\newsavebox{\bweScript}
\newsavebox{\bweScriptScript}

\savebox{\bweDisp}{%
  \begin{tikzpicture}[scale=0.20]
    \node at (0,0) [draw,circle,fill=black,inner sep=1.1pt] (A) {};
    \node at (0.8,1) [draw,circle,fill=white,inner sep=1.1pt] (B) {};
    \draw (A) -- (B);
  \end{tikzpicture}%
}
\savebox{\bweText}{%
  \begin{tikzpicture}[scale=0.20]
    \node at (0,0) [draw,circle,fill=black,inner sep=1.1pt] (A) {};
    \node at (0.8,1) [draw,circle,fill=white,inner sep=1.1pt] (B) {};
    \draw (A) -- (B);
  \end{tikzpicture}%
}
\savebox{\bweScript}{%
  \begin{tikzpicture}[scale=0.13]
    \node at (0,0) [draw,circle,fill=black,inner sep=0.7pt] (A) {};
    \node at (0.8,1) [draw,circle,fill=white,inner sep=0.7pt] (B) {};
    \draw (A) -- (B);
  \end{tikzpicture}%
}
\savebox{\bweScriptScript}{%
  \begin{tikzpicture}[scale=0.10]
    \node at (0,0) [draw,circle,fill=black,inner sep=0.6pt] (A) {};
    \node at (0.8,1) [draw,circle,fill=white,inner sep=0.6pt] (B) {};
    \draw (A) -- (B);
  \end{tikzpicture}%
}

\newcommand{\blackwhiteedge}{%
  \mathchoice
    {\raisebox{-0.4ex}{\usebox{\bweDisp}}}
    {\raisebox{-0.4ex}{\usebox{\bweText}}}
    {\raisebox{-0.3ex}{\usebox{\bweScript}}}
    {\raisebox{-0.2ex}{\usebox{\bweScriptScript}}}
}

\newsavebox{\wbeDisp}
\newsavebox{\wbeText}
\newsavebox{\wbeScript}
\newsavebox{\wbeScriptScript}

\savebox{\wbeDisp}{%
  \begin{tikzpicture}[scale=0.20]
    \node at (0,0) [draw,circle,fill=white,inner sep=1.1pt] (A) {};
    \node at (0.8,1) [draw,circle,fill=black,inner sep=1.1pt] (B) {};
    \draw (A) -- (B);
  \end{tikzpicture}%
}
\savebox{\wbeText}{%
  \begin{tikzpicture}[scale=0.20]
    \node at (0,0) [draw,circle,fill=white,inner sep=1.1pt] (A) {};
    \node at (0.8,1) [draw,circle,fill=black,inner sep=1.1pt] (B) {};
    \draw (A) -- (B);
  \end{tikzpicture}%
}
\savebox{\wbeScript}{%
  \begin{tikzpicture}[scale=0.13]
    \node at (0,0) [draw,circle,fill=white,inner sep=0.7pt] (A) {};
    \node at (0.8,1) [draw,circle,fill=black,inner sep=0.7pt] (B) {};
    \draw (A) -- (B);
  \end{tikzpicture}%
}
\savebox{\wbeScriptScript}{%
  \begin{tikzpicture}[scale=0.10]
    \node at (0,0) [draw,circle,fill=white,inner sep=0.6pt] (A) {};
    \node at (0.8,1) [draw,circle,fill=black,inner sep=0.6pt] (B) {};
    \draw (A) -- (B);
  \end{tikzpicture}%
}

\newcommand{\whiteblackedge}{%
  \mathchoice
    {\raisebox{-0.4ex}{\usebox{\wbeDisp}}}
    {\raisebox{-0.4ex}{\usebox{\wbeText}}}
    {\raisebox{-0.3ex}{\usebox{\wbeScript}}}
    {\raisebox{-0.2ex}{\usebox{\wbeScriptScript}}}
}

\newsavebox{\wweDisp}
\newsavebox{\wweText}
\newsavebox{\wweScript}
\newsavebox{\wweScriptScript}

\savebox{\wweDisp}{%
  \begin{tikzpicture}[scale=0.20]
    \node at (0,0) [draw,circle,fill=white,inner sep=1.1pt] (A) {};
    \node at (0.8,1) [draw,circle,fill=white,inner sep=1.1pt] (B) {};
    \draw (A) -- (B);
  \end{tikzpicture}%
}
\savebox{\wweText}{%
  \begin{tikzpicture}[scale=0.20]
    \node at (0,0) [draw,circle,fill=white,inner sep=1.1pt] (A) {};
    \node at (0.8,1) [draw,circle,fill=white,inner sep=1.1pt] (B) {};
    \draw (A) -- (B);
  \end{tikzpicture}%
}
\savebox{\wweScript}{%
  \begin{tikzpicture}[scale=0.13]
    \node at (0,0) [draw,circle,fill=white,inner sep=0.7pt] (A) {};
    \node at (0.8,1) [draw,circle,fill=white,inner sep=0.7pt] (B) {};
    \draw (A) -- (B);
  \end{tikzpicture}%
}
\savebox{\wweScriptScript}{%
  \begin{tikzpicture}[scale=0.10]
    \node at (0,0) [draw,circle,fill=white,inner sep=0.6pt] (A) {};
    \node at (0.8,1) [draw,circle,fill=white,inner sep=0.6pt] (B) {};
    \draw (A) -- (B);
  \end{tikzpicture}%
}

\newsavebox{\edgeDisp}
\newsavebox{\edgeText}
\newsavebox{\edgeScript}
\newsavebox{\edgeScriptScript}

\savebox{\edgeDisp}{%
  \begin{tikzpicture}[scale=0.20]
    \node at (0,0) [draw,circle,fill=black,inner sep=.1pt] (A) {};
    \node at (0.8,1) [draw,circle,fill=black,inner sep=.1pt] (B) {};
    \draw (A) -- (B);
  \end{tikzpicture}%
}
\savebox{\edgeText}{%
  \begin{tikzpicture}[scale=0.20]
    \node at (0,0) [draw,circle,fill=black,inner sep=.1pt] (A) {};
    \node at (0.8,1) [draw,circle,fill=black,inner sep=.1pt] (B) {};
    \draw (A) -- (B);
  \end{tikzpicture}%
}
\savebox{\edgeScript}{%
  \begin{tikzpicture}[scale=0.13]
    \begin{scope}[rotate=-30]
      \draw (-0.2,-0.2) -- (0.2,0.2);
      \draw (-0.2,0.2)  -- (0.2,-0.2);
    \end{scope}
    \begin{scope}[shift={(0.8,1)},rotate=-30]
      \draw (-0.2,-0.2) -- (0.2,0.2);
      \draw (-0.2,0.2)  -- (0.2,-0.2);
    \end{scope}
    \node at (0,0)   [inner sep=.1pt] (A) {};
    \node at (0.8,1) [inner sep=.1pt] (B) {};
    \draw (A) -- (B); 
  \end{tikzpicture}%
}
\savebox{\edgeScriptScript}{%
  \begin{tikzpicture}[scale=0.10]
    \node at (0,0) [draw,circle,fill=black,inner sep=.1pt] (A) {};
    \node at (0.8,1) [draw,circle,fill=black,inner sep=.1pt] (B) {};
    \draw (A) -- (B);
  \end{tikzpicture}%
}

\makeatother

\begin{document}


\title{\Large\bf Co-evolving Vertex and Edge Dynamics in~Dense~Graphs}

\author{Siva Athreya, \quad  Frank den Hollander, \quad Adrian R\"ollin}
 
\newcommand{\Addresses}{{
\vfill
\footnotesize
\bigskip
\noindent Siva Athreya, International Centre for Theoretical Sciences, Survey No.\ 151, Shivakote,  
He\-sa\-ra\-ghat\-ta Hobli, Bengaluru 560089, India \& Indian Statistical Institute,  
8th Mile Mysore Road, Bengaluru 560059, India. 
\texttt{athreya@icts.res.in}

\medskip
\noindent Frank den Hollander, Mathematical Institute, Leiden University, Einsteinweg~55, 2333 CC Leiden, The Netherlands.
\texttt{denholla@math.leidenuniv.nl}

\medskip
\noindent Adrian R\"ollin, Department of Statistics and Data Science, National University of Singapore, 6~Science Drive~2, Singapore 117546. 
\texttt{adrian.roellin@nus.edu.sg}

\medskip
}}

\date{\today}

\maketitle


\begin{abstract} \noindent
We consider a random graph in which vertices can have one of two possible colours. Each vertex switches its colour at a rate that is proportional to the number of vertices of the other colour to which it is connected by an edge. Each edge turns on or off according to a rate that depends on whether the vertices at its two endpoints have the same colour or not. We prove that, in the limit as the graph size tends to infinity and the graph becomes dense, the graph process converges, in a suitable path topology, to a limiting Markov process that lives on a certain subset of the space of coloured graphons. In the limit, the density of each vertex colour evolves according to a Fisher-Wright diffusion driven by the density of the edges, while the underlying edge connectivity structure evolves according to a stochastic flow whose drift depends on the densities of the two vertex colours. 
\end{abstract}



\normalsize

  
\section{Introduction and main results}  
\label{sec1}

Networks form the backbone of our interconnected society, yet gaining a deep understanding of networks is scientifically challenging. There is a vast literature on static networks and processes evolving on them, such as the transmission of information, the spread of disease, the flow of traffic, or the distribution of energy. In contrast, the mathematical analysis of processes on \emph{dynamic} networks, which themselves evolve over time, remains in its infancy. While there has been notable progress in the physics and computer science literature, based on heuristic and approximative approaches, the current mathematical literature contains only a handful of scattered examples.

A pressing challenge lies in understanding systems with two simultaneous levels of dynamics: randomly evolving processes on randomly evolving networks locked in a feedback loop --- a setting termed \emph{co-evolution} that creates complex bi-directional interactions. Fundamental questions are: What constitutes dynamic equilibrium in such systems and do network dynamics accelerate or hinder equilibration? What are the classes of spatio-temporal behaviours that emerge? How can we design robust networks that leverage this interplay? Are there universal frameworks for estimating, optimising and controlling dynamics, both of and on networks? Addressing these challenges requires mathematical models that capture the co-evolution of network structure, yet remain tractable for rigorous analysis.

In this paper we propose a \emph{voter model} on a dynamically evolving dense coloured graph in which vertices can change colours (often referred to as \emph{opinions}) and edges can switch between being present or absent in a coupled way (Section~\ref{sec3}).  We provide and prove a complete mathematical description of a class of mutually coupled colour dynamics and graph dynamics in the large-size limit based on coloured graphons. We show emergence of a form of \emph{homogenisation} that manifests itself through a \emph{collapse} of the joint dynamics onto a subset of the space of coloured graphons, where the vertex colours homogenise while edge states evolve via stochastic flows governed by the overall colour densities (Theorems~\ref{thm1}--\ref{thm2}). We believe that our work represents an important advance towards the goals mentioned above, and is the first to rigorously characterise the large-size limit for a canonical model of co-evolving vertex and edge dynamics in dense graphs. The framework we put forward offers a mathematical foundation to analyse emergent phenomena arising from the interplay between local dynamics and global structure.

The remainder of this section is organised as follows: In Section~\ref{sec2} we provide a brief review of the literature, followed by the description of the model in Section~\ref{sec3} and the statement of the main results in Section~\ref{sec4}. We conclude this section  with a discussion of the significance of our main results, the novelty in our proof technique, and the organisation of the paper in  Section~\ref{sec5}.


\subsection{Background and earlier work}
\label{sec2}

The voter model, which dates back to the work of \cite{HL75}, is an interacting particle system describing opinion evolution in a simplified social network context, specifically, describing how individuals interact and form collective opinions. In static networks, \emph{consensus} is inevitable, with all individuals eventually adopting the same opinion. However, empirical studies of social media have uncovered evolution patterns where individuals split into communities holding distinct opinions that remain stable over long periods, a phenomenon known as \emph{polarisation}. Strong evidence indicates that the interplay between opinion dynamics and network dynamics is essential in facilitating polarisation. Recent years have seen an increasing number of papers attempting to describe co-evolutionary random networks via simulations and heuristic arguments. See, for example, \cite{HPZ11, LHAJS20, BVP24}, who study various aspects of polarisation, \cite{GZ06, KB08a} for results on how consensus is formed, and \cite{PTN06, RM20} and references therein on emergence of complex structural changes. 

Since the mathematical literature on co-evolution is \emph{scarce}, we provide some background and describe earlier work, so as to allow the reader to place our paper in the \emph{proper context}.
 

\paragraph{Earlier work for sparse graphs.}

\cite{DGLMSSSV12} analysed two discrete-time voter models on a dynamically evolving sparse random graph on $n$ vertices. At each time step, a uniformly chosen discordant edge (an edge connecting vertices with different opinions) is updated. With probability $\alpha \in (0,1)$, a randomly selected vertex of the chosen discordant edge adopts the colour of its neighbour. Otherwise, that is, with probability $1-\alpha$, the edge is rewired uniformly at random either to any vertex (variant 1) or to a vertex with the same colour (variant 2). The dynamics terminates once no discordant edges remain. The initial graph is an Erd\H{o}s-R\'enyi random graph with average degree $\lambda \in (1,\infty)$, and the initial densities of the colours are $u$ and $1-u$ with $u \in (0,1)$. It was argued, with the help of simulations and approximate computations, that a \emph{sharp threshold} $\alpha_c$ exists between consensus and polarisation. An approximate formula was derived for $\alpha_c$ in terms of $\lambda$ and $u$ based on a mean-field type approximation.  

\cite{HN06} considered a similar dynamics and conjectured a phase transition similar to the one found in \cite{DGLMSSSV12}, based on simulation. At each time unit choose a vertex uniformly at random. With probability $\alpha$ choose an edge that is connected to the selected vertex uniformly at random, and reconnect that edge to a vertex that shares the same colour as the selected vertex chosen uniformly at random. With probability $1-\alpha$ the selected vertex adopts the colour of a vertex to which it is connected by an edge chosen uniformly at random. Note that in this model the dynamics also ends when there are no edges between vertices with different colours. 

Sparse random graphs with colour dynamics were analysed with rigorous arguments by \cite{CCC16,ABHdHQ24,ACHQ23,ABHdHQ24pr,C24}. \cite{CCC16} showed that on a configuration random graph with prescribed degrees the fraction of vertices with one colour evolves according to a \emph{Fisher-Wright diffusion} on time scale~$n$ with a \emph{diffusion coefficient} that depends on the limiting degree distribution. \cite{ABHdHQ24} studied the same dynamics on the random regular graph, for which the diffusion coefficient can be computed, and it was shown that the fraction of discordant edges also follows a Fisher-Wright diffusion. \cite{ABHdHQ24pr} studied the same model when the edges are subjected to a \emph{rewiring dynamics}, and it was identified how the diffusion coefficient depends on the rate of rewiring. \cite{ACHQ23} and \cite{C24} analysed the same dynamics on the configuration random graph with directed edges, and the diffusion coefficient was computed under certain mild assumptions on the in-degrees and the out-degrees. 


\paragraph{Earlier work for dense graphs.}

Among the few rigorous results for dense graphs are those of \cite{BS17}, who analysed a version of the two discrete-time voter models on a dynamically evolving dense graph on $n$ vertices proposed by \cite{DGLMSSSV12}. The starting point is a dense Erd\H{o}s-R\'enyi random graph with positive densities of the two colours. At each time step, a uniformly chosen discordant edge (an edge connecting vertices with different opinions) is updated. With probability $\beta/n$, the opinion of a randomly selected vertex of the chosen discordant edge adopts the opinion of its neighbour. Otherwise, that is, with probability $1-\beta/n$, the edge is rewired uniformly at random either to any vertex (variant 1) or to a vertex with the same opinion (variant 2). The dynamics terminates once no discordant edges remain. For sufficiently small $\beta$, the dynamics ends in a time of order $n^2$ and polarisation occurs  (that is, the colours get fixated at non-trivial proportions in the infinite-time limit) at \emph{equal} fractions of the two opinions. For sufficiently large $\beta$, the dynamics ends in a time of order $n^3$ and polarisation occurs at \emph{different} fractions of the two opinions. As $\beta\to\infty$ the fraction of the minority opinion tends to zero, that is, there is asymptotic consensus (all the vertices hold the same opinion). 

\cite{EHS25} studied a voter model variant where discordant interactions lead to agreement (with probability $q$) or permanent edge removal. Both consensus (within components) and polarisation occur with positive probability, and \emph{bounds} were derived for these probabilities that depend on~$q$. \cite{BdHM22} considered graph-valued processes where vertex opinions fluctuate randomly and influence edge probabilities, which in turn influence the vertex opinions. A \emph{sample-path large deviation principle} and process convergence were derived for the empirical graphon limit, especially when vertex opinion fluctuations dominate edge state fluctuations. \cite{BdHM24pr} studied two-opinion voter models on dense dynamic random graphs favouring concordant edges, focusing on consensus versus polarisation. \emph{Functional laws of large numbers} for the opinion densities were derived, and \emph{equilibria} were characterised. \cite{BK25} considered two-opinion voter models on dense dynamic graphs where edges evolve based on endpoint opinions, while the opinion dynamics is unaffected by the graph state. A \emph{functional central limit theorem} was derived for the vector of subgraph counts.

\cite{bcw23} analysed heterogeneous diffusive interacting particle systems with mean-field interactions weighted by a dense graph. A \emph{law of large numbers} was shown as the graph sequence converges to a graphon, with the limit described by coupled stochastic differential equations. While applicable to the voter model, the graph structure was assumed to be \emph{static}.


\subsection{Model: Coupled dynamics of vertices and edges}
\label{sec3}

In this paper we consider an example of co-evolving dynamics in which vertices can switch between colour $0$ or $1$ (the opinions) and edges can switch between being present or absent, in a coupled way. We now define the state space and the generator of this Markov process.


\paragraph{Generator.}
Let $\sG_n$ be the set of all simple, vertex-coloured graphs on the vertex set $[n]=\{1,\dots,n\}$, each vertex having one of two possible colours (0 or 1). For any $G \in \sG_n$ and any vertex $u$, let $G^u$ be the graph obtained by \emph{flipping} the colour of $u$. For any pair of distinct vertices $u$ and $v$, let $G^{uv}$ be the graph obtained by switching the status of the edge between $u$ and $v$. We write $c_u(G)$ for the colour of vertex $u$, and $e_{uv}(G)$ for the status (0 or 1) of the edge between vertices $u$ and $v$, where $u \neq v$. For notational convenience, to shorten formulas we write
\be{
  \^c_u = 1-c_u, 
  \qquad 
  \^e_{uv}=1-e_{uv},
  \qquad
  c_{uv} = \I[c_u=c_v], 
  \qquad
  \^c_{uv} = \I[c_u\neq c_v]. 
}
Whenever the underlying graph $G$ is clear from the context, we abbreviate the latter as $c_u$, $e_{uv}$, etc. We will occasionally write $c^G_u$, etc., to make the dependence on the underlying graph $G$ explicit.

For each $n\in\IN$, we consider the Markov chain $G^n = (G^n_t)_{t \ge 0}$ on $\sG_n$ with generator
\ben{\label{1}
(\cA_n f)(G) 
= \eta \sum_{1 \le u \le n} r^{\mathrm{v}}_u(G) \bclr{f(G^u)-f(G)} 
+ \rho \sum_{1 \le u < v \le n} r^{\mathrm{e}}_{u,v}(G) \bclr{f(G^{uv})-f(G)},
}
acting on any bounded function $f \colon \sG_n \to \IR$, where
\ba{
r^{\mathrm{v}}_u(G) 
& = \sum_{v:v\neq u}e_{uv}\^c_{uv},\\
r^{\mathrm{e}}_{u,v}(G) 
&= c_{uv}\bclr{s_{\mathrm{c},0}\^e_{uv}+s_{\mathrm{c},1}e_{uv}} + \^c_{uv} \bclr{s_{\mathrm{d},0}\^e_{uv}+s_{\mathrm{d},1}e_{uv}}.
}
This dynamics has two parts, one for vertex-colour updates and one for edge-state updates:
\begin{itemize}
\item \textbf{Vertex updates.} Every vertex flips its colour at rate $\eta m$, where $m$ is the current number of neighbours of opposite colour. 
\item \textbf{Edge updates.} Every pair of vertices switches its edge at rates governed by their colour match (concordant or discordant) and by whether or not they are currently connected:
\begin{itemize}
\item Concordant vertices connect at rate $\rho s_{\mathrm{c},0}$ and disconnect at rate~$\rho s_{\mathrm{c},1}$.
\item Discordant vertices connect at rate $\rho s_{\mathrm{d},0}$ and disconnect at rate~$\rho s_{\mathrm{d},1}$.
\end{itemize}
\end{itemize}
We assume throughout that $\eta$ and $\rho$ are positive constants and that the switching rates $s_{\mathrm{c},0}$, $s_{\mathrm{c},1}$, $s_{\mathrm{d},0}$, and $s_{\mathrm{d},1}$ are non-negative constants, all independent of $n$. In our notation of the switching rates, `$\mathrm{c}$' stands for concordant, `$\mathrm{d}$' stands for discordant, `$0$' for an absent edge and `$1$' for a present edge. Thus, our model has a total of six parameters.  We allow some or all switching rates to be zero.


\subsection{Main results}
\label{sec4}

Our goal is to show that, in the dense graph limit, the coupled dynamics can be \emph{fully quantified} at the graph level. To do so, we develop tools to describe processes of dense coloured graphs and their limits as processes of coloured graphons, and we show that \emph{homogenisation} of the colours can be exploited to derive limiting evolution equations. The latter allows us to control the mutual feedback that is inherent in the \emph{co-evolution} of graph and colours, and to identify scenarios leading to consensus versus polarisation.

To state our main results precisely, we need some definitions and notation (more details will be given in later sections). A \emph{coloured graphon} $(\kappa,c)$ is a pair consisting of a symmetric measurable function $\kappa\colon [0,1]^2 \to [0,1]$, referred to as \emph{graphon}, and a measurable function $c\colon [0,1] \to [0,1]$, referred to as \emph{colouring}. We may interpret $\kappa(\dd x,\dd y)$ as the fraction of edges present between vertices in $\dd x$ and $\dd y$, and $c(\dd x)$ as the fraction of vertices of colour $1$ among vertices in $\dd x$. The space of all coloured graphons is denoted by $\cW$ and is endowed with a pseudo-metric $\dsub$. The set of measure-preserving maps on $[0,1]$ induces a set of equivalence classes on $\cW$, denoted by $\~\cW$, on which $\dsub$ turns out to be a metric. The subspace of $\cW$ for which the colouring is constant is denoted by $\cW_0$ and the corresponding set of equivalence classes by $\~\cW_0$. Every coloured graph~$G\in\cG_n$ induces a coloured graphon~$(\kappa^G,c^G)\in\cW$, where $\kappa^G$ is the \emph{empirical graphon} and~$c^G$ is the \emph{empirical colouring} (see Section~\ref{sec7} for details). 

Our first result asserts the existence of a certain Markov process on $\cW_0$ --- and, by means of the equivalence relation, an associated Markov process on $\~\cW_0$ --- that turns out to be the limiting process of~$G^n$. Two key quantities drive the process: the limiting density $q_t$ of vertices of colour $1$ and the limiting edge density $p_t$. These two quantities, along with additional ordinary differential equations that determine the evolution of the graphon, yield a Markov process on $\cW_0$. More precisely, consider the system of stochastic differential equations
\ben{\label{2}
\left\{
\begin{aligned}
\dd{q_t} &= \sqrt{2\eta\, p_t q_t (1-q_t)}\dd{W_t},\\
\dd{\kappa_t(x,y)} &= \rho\, V\bigl(\kappa_t(x,y), q_t\bigr)\dd t,
\quad (x,y) \in [0,1]^2,
\end{aligned}
\right.
}
with initial condition $(\kappa_0,q_0)$, where $\kappa_0$ is a graphon, $q_0\in(0,1)$, 
\be{
p_t = \int_{[0,1]^2} \dd x\dd y\kappa_t(x,y),
}
$W = (W_t)_{t \ge 0}$ is standard Brownian motion on $\R$, and $V\colon\,[0,1]\times [0,1] \rightarrow [0,1]$ is the function given by
\besn{\label{3}
V(p,q)
= (1-p)&\bcls{\bclr{q^2+(1-q)^2}s_{\mathrm{c},0} + 2q(1-q)s_{\mathrm{d},0}}\\
{}- p&\bcls{\bclr{q^2+(1-q)^2}s_{\mathrm{c},1} + 2q(1-q)s_{\mathrm{d},1}}.
}

\begin{theorem}[Existence and uniqueness of limit] 
\label{thm1}
The following statements hold:
\begin{itemize}
\item[(a)]
The system of stochastic differential equations \eqref{2} has a unique strong solution $(\kappa_t,q_t)_{t\geq 0}$.
\item[(b)] 
If $(\kappa_0,q_0)$ and $(\kappa'_0,q'_0)$ are such that the uncoloured graphons $\kappa_0$ and $\kappa'_0$ lie in the same equivalence class and $q_0=q'_0$, then the corresponding solutions $(\kappa_t,c_t)_{t\ge0}$ and $(\kappa'_t,c'_t)_{t\ge0}$ can be constructed on a common probability space so that they remain in the same coloured graphon equivalence class for all $t\ge0$.
\item[(c)] 
The solution to \eqref{2} induces a $\cW_0$-valued Markov process $(\kappa_t,c_t)_{t\geq 0}$, where $c_t(x) = q_t$ for all $x \in [0,1]$ and all $t\geq 0$, which in turn induces a $\~\cW_0$-valued Markov process $(\~\kappa_t,\~c_t)_{t\geq 0}$.
\end{itemize}
\end{theorem}

Our second result asserts that the processes $(G^n)_{n\in\N}$ converge weakly to the process from Theorem~\ref{thm1} in the path topology of convergence in measure. We will need a quantity that measures the connectedness of the initial graph to ensure denseness: For a graphon $\kappa$, let
\ben{\label{4}
\noverlap(\kappa) = \inf_{x,y {\in} [0,1]} \int_0^1 \kappa(x,z)\kappa(z,y)\dd z,
} 
and, for $G\in \cG_n$, let $\noverlap(G) = \noverlap(\kappa^G)$, where $\kappa^G$ is the empirical graphon of $G$.

\begin{theorem}[Convergence of the finite coloured graph process]
\label{thm2}
Let the Markov process $G^n = (G^n_t)_{t \geq 0}$ on $\cG_n$ be defined as in Section~\ref{sec3}. Let $(\kappa',c')\in \cW$ and suppose that
\ben{\label{5}
\lim_{n\to\infty} \dsub\bigl(G^n_0,(\kappa',c')\bigr) = 0 \qquad\text{and}\qquad \liminf_{n\to\infty} \noverlap(G^n_0) > 0.
}
Then, as $n \to \infty$, $(G^n_t)_{t\geq 0}$ converges weakly to the Markov process $(\kappa_t,c_t)_{t\geq 0}$ from Theorem~\ref{thm1} in the path topology of convergence in measure with respect to $(\~\cW,\dsub)$, with initial condition $\kappa_0=\kappa'$ and $c_0(x)=\int_0^1 \dd y c'(y)$ for all $x\in[0,1]$. Furthermore, the finite-dimensional distributions of $(G^n_t)_{t> 0}$ converge to those of $(\kappa_t,c_t)_{t> 0}$.
\end{theorem}


\subsection{Discussion and outline} 
\label{sec5} 

Our main results provide a complete description of the limiting coloured graphon process in terms of joint evolution equations \eqref{2}  and offers an example of \emph{co-evolution} of individuals and their friendships in a social network.


\paragraph{Solution to the stochastic differential equations \eqref{2}.} 
In proving Theorem~\ref{thm2}, which is our main result, our first order of work  will be to show existence and uniqueness of the solution $(p_t,q_t)_{t \geq 0}$ of the autonomous pair of coupled equations 
\ben{\label{6}
\left\{
\begin{aligned}
\dd{p_t} &= \rho\, V\bclr{p_t,q_t}\dd t,\\
\dd{q_t} &= \sqrt{2\eta\, p_t\,q_t(1-q_t)}\dd{W_t},
\end{aligned}
\right.
}
with initial values 
\be{
p_0 = \int_{[0,1]^2}\dd x\dd y \kappa_0(x,y),
\qquad
q_0 = \int_{[0,1]}\dd x c_0(x).
}
Then, we will insert the solution $q_t$ into \eqref{2} and solve for $\kappa_t(x,y)$ for each $(x,y) \in [0,1]^2$ individually via integration (see Lemma~\ref{lem8}).


\paragraph{Limiting process.} 
Theorem~\ref{thm1} identifies the limiting process $(\kappa_t,c_t)_{t\geq 0}$ as a Markov process that takes values on a subspace of the space of coloured graphons. 
\begin{itemize}
\item  
We see that the path of the coloured graph process converges to a limit by contracting to a subspace $\~\cW_0$ where the colours are \emph{homogenised}. The convergence is in the path topology of convergence in measure (also referred to as the Meyer-Zheng path topology), which amounts to weak convergence in `time~$\times$ space' (for details, see \cite[Section~2.2]{AdHR24} and Section~\ref{sec8} below). In the limit as $n\to\infty$, the rate of contraction diverges, so that the path of the coloured graph spends most of its time close to the subspace in which the limting process lives. 
\item
The equations in \eqref{2} show that the fractions of connected edges and vertices with colour $1$ evolve in a coupled way: $p_t$ moves according to a \emph{stochastic flow} with a drift that depends on $p_t$ and $q_t$, while $q_t$ moves according to a \emph{Fisher-Wright diffusion} with a diffusion coefficient that is proportional to $p_t$. After consensus is reached, corresponding to $q_t \in \{0,1\}$, the diffusion is trapped and the stochastic flow becomes a \emph{deterministic flow}.
\item
The equations in \eqref{2} also show that the graphon $\kappa_t$ moves according to a \emph{stochastic flow} that is driven by the pair $(p_t,q_t)$. Note that, because $p \mapsto V(p,q)$ is linear for every $q \in [0,1]$, we have
\ben{
\int_{[0,1]^2} \dd x\dd y V\bclr{\kappa(x,y),q} = V\left(\int_{[0,1]^2} \dd x\dd y \kappa(x,y),q\right) = V(p,q),
}
so that the second line of \eqref{2} is consistent with the first line of \eqref{6}. 
\end{itemize}


\paragraph{Collapse.}
What Theorem~\ref{thm2} shows is that the system rapidly establishes a \emph{quasi-equilibrium}, in which the colours are randomly redistributed at a fixed colour density, and afterwards \emph{slowly changes the colour density} while continuing to rapidly redistribute the colours at the current colour density. The separation between the fast time scale and the slow time scale allows for an explicit derivation of the limiting evolution equations. The same type of separation occurs in the description of interacting particle systems when going from a \emph{microscopic} description to a \emph{mesoscopic} description, for instance, in the derivation of hydrodynamic limit equations in non-equilibrium statistical physics.     


\paragraph{Consensus versus polarisation of the limiting process.}
Suppose that $s_{\mathrm{c},0}+s_{\mathrm{c},1}>0$ and $s_{\mathrm{d},0}+s_{\mathrm{d},1}>0$. A straightforward computation shows that $V(p,q)$ $=0$ when $p = \alpha_q\,p_c + (1-\alpha_q)\,p_d$ with
\bes{
p_c = \frac{s_{\mathrm{c},0}}{s_{\mathrm{c},0}+s_{\mathrm{c},1}}, \quad 
p_d = \frac{s_{\mathrm{d},0}}{s_{\mathrm{d},0}+s_{\mathrm{d},1}}, \quad
\frac{1-\alpha_q}{\alpha_q} = \frac{s_{\mathrm{d},0}+s_{\mathrm{d},1}}{s_{\mathrm{c},0}
+s_{\mathrm{c},1}}\,\frac{2q(1-q)}{q^2+(1-q)^2}.
}
Note that, because $(q_t)_{t \geq 0}$ is a bounded martingale, $\lim_{t\to\infty} q_t$ exists almost surely.
\begin{itemize}
\item
\emph{Consensus} occurs when $q \in \{0,1\}$, which are absorbing states, in which case $\alpha_0 = \alpha_1 = 1$. After consensus is reached, the fraction of connected edges $p_t$ converges to $p_c$ exponentially fast as $t\to\infty$. In the equilibrium state, all vertices have the same colour, while edges connect and disconnect randomly with the fraction of connected edges fixed at $p_c$ (see Figure~\ref{fig1} below for a simulation).

If $s_{\mathrm{c},0}=s_{\mathrm{d},0}=s_0$ and $s_{\mathrm{c},1}=s_{\mathrm{d},1}=s_1$, then $V(p,q) = s_0(1-p)-s_1 p$, in which case $p_t$ evolves independently of $q_t$ and converges to $\frac{s_0}{s_0+s_1}$ exponentially fast as $t\to\infty$, while $q_t$ eventually gets absorbed in $\{0,1\}$.
\item
\emph{Polarisation} occurs when $p=0$ and $q \notin \{0,1\}$. This is possible only when $p_c = p_d =0$, that is, $s_{\mathrm{c},0} = s_{\mathrm{d},0}=0$, for which disconnected edges cannot connect. The fraction of connected edges $p_t$ converges to $0$ exponentially fast as $t\to\infty$ (at rate at least $s_{\mathrm{c},1} \wedge s_{\mathrm{d},1}$), and $\lim_{t\to\infty} q_t\in (0,1)$ with positive probability (see Figure~\ref{fig2} below for a simulation).

If $s_{c,0} = s_{d,0}=0$ and $s_{c,1} = s_{d,1}=s_1$, then $V(p,q) = - s_1 p$, in which case $p_t$ evolves independently of $q_t$ and is given by $p_t = p_0 \ee^{-\rho s_1t}$. Putting $T(t) = \eta \int_0^t p_u\, \dd u$ as a new time variable, we see that $q_t$ is equal in distribution to $q^*_{T(t)}$ with $(q^*_t)_{t \geq 0}$ the solution to the standard Fisher-Wright diffusion $\dd q^*_t = \sqrt{2q^*_t(1-q^*_t)}\,\dd W_t$, where $(W_t)_{t \geq 0}$ is Brownian motion.  Since $T(\infty) = \eta p_0/\rho s_1$, it follows that polarisation occurs with probability
\bes{ 
  \IP^*(\tau> \eta p_0/\rho s_1),
}
where $\tau$ is the consensus time of the standard Fisher-Wright diffusion, which satisfies $\mathbb{P}^*(\tau<\infty)=1$.
\end{itemize}


\paragraph{Convergence of consensus and polarisation probabilities.} 
An important question is which finite-system quantities converge to the corresponding infinite-system quantities under the topology of convergence in measure. Ultimately, this is a question about continuity with respect to that topology. While a full discussion of this question in the context of coloured graphon-valued processes is beyond the scope of this paper, we consider a few special cases. To this end, let $(\cM_{\~\cW},\dMZ)$ be the metric space of coloured graphon-valued paths (see Sections~\ref{sec8}--\ref{sec9} for more details), let $f\colon \cM_{\~\cW}\to\IR$ be a measurable and bounded function, and let $C_f\subset \cM_{\~\cW}$ be the set of continuity points of $f$ with respect to $\dMZ$. If
\ben{\label{7}
  \IP\bclr{(\kappa_t,c_t)_{t\geq 0}\in C_f}=1,
}
then, by the portmanteau theorem, 
\be{
  \lim_{n\to\infty} \IE\bcls{f\bclr{(G^n_t)_{t\geq 0}}} = \IE\bcls{f\bclr{ (\kappa_t,c_t)_{t\geq 0}}}.
}
\begin{itemize}
\item The most straightforward quantities are \emph{occupation times} over finite time intervals. Let $0<a<b<\infty$ and $0<u<1$. For $(\phi,\omega)\in \cM_{\~\cW}$, define the function
\be{
  f\bclr{(\phi,\omega)} = \int_a^b \I[\-\omega_t\leq u]\dd t,
} 
where $\-\omega_t = \int_0^1 \omega_t(y) \dd y$. This is the amount of time the fraction of colour-$1$ vertices equals or is below $u$ in the time interval $[a,b]$. It is not difficult to see that points of discontinuity with respect to $\dMZ$ are paths for which $\omega$ equals $u$ for a positive amount of time between time points $a$ and $b$. Since the colour function of the limiting process $(c_t)_{t\geq 0} \equiv (q_t)_{t\geq 0}$ is diffusive on $(0,1)$, we see that \eqref{7} holds for this~$f$. Hence
\ben{\label{8}
 \lim_{n\to\infty} \IE\left[\int_a^b \I[q^n_t\leq u]\,\dd t\right] 
 = \int_a^b \IP(q_t \leq u)\dd t,
}
where $q^n_t$ is the fraction of vertices of colour $1$ in $G^n_t$. Now, Theorem~\ref{thm2} states in addition that the finite-dimensional distributions of $(G^n_t)_{t>0}$ converge to those of $(\kappa_t,c_t)_{t>0}$. Hence, pointwise convergence also holds: 
\ben{\label{9}
   \lim_{n\to\infty} \IP(q^n_t\leq u) = \IP(q_t \leq u) \qquad \forall\,t>0.
}
Note that neither \eqref{8} nor \eqref{9} implies the other.
\item 
The result \eqref{9} in particular implies that the absorption probabilities for each fixed time~$t>0$ converge:
\ben{\label{10}
  \lim_{n\to\infty} \IP(q^n_t=0) = \IP(q_t=0), \qquad
  \lim_{n\to\infty} \IP(q^n_t=1) = \IP(q_t=1).
}
It would be interesting to know whether the distribution of $q^n_\infty = \lim_{t\to\infty} q^n_t$ converges to that of $q_\infty=\lim_{t\to\infty} q_t$ as $n\to\infty$. This would require substantial additional work, outside the focus of this paper. Figures \ref{fig1}--\ref{fig2} suggest that the answer is yes.
\end{itemize}


\paragraph{Concordant and discordant edges.} 
The limiting fractions of concordant and discordant connected edges are, respectively,
\bes{
\cC_t = p_t\,\big[(q_t)^2 + (1-q_t)^2\big], \qquad 
\cD_t = p_t\,\big[2q_t(1-q_t)\big]. 
}
Note that 
\bes{
\cC_t - \cD_t = p_t\,\big[q_t - (1-q_t)\big]^2 \geq 0,
}
with equality if and only if $p_t = 0$ and/or $q_t = \tfrac12$. The same expressions hold for the limiting fractions of concordant and discordant disconnected edges, after $p_t$ is replaced by $1-p_t$. This explains why the concordant and discordant edge densities in Figures~\ref{fig1} and~\ref{fig2} are separated, with the concordant edge density dominating the discordant edge density (of course, for finite $n$, the separation is not strict).

\begin{figure}[t!]
\includegraphics[width=0.5\linewidth]{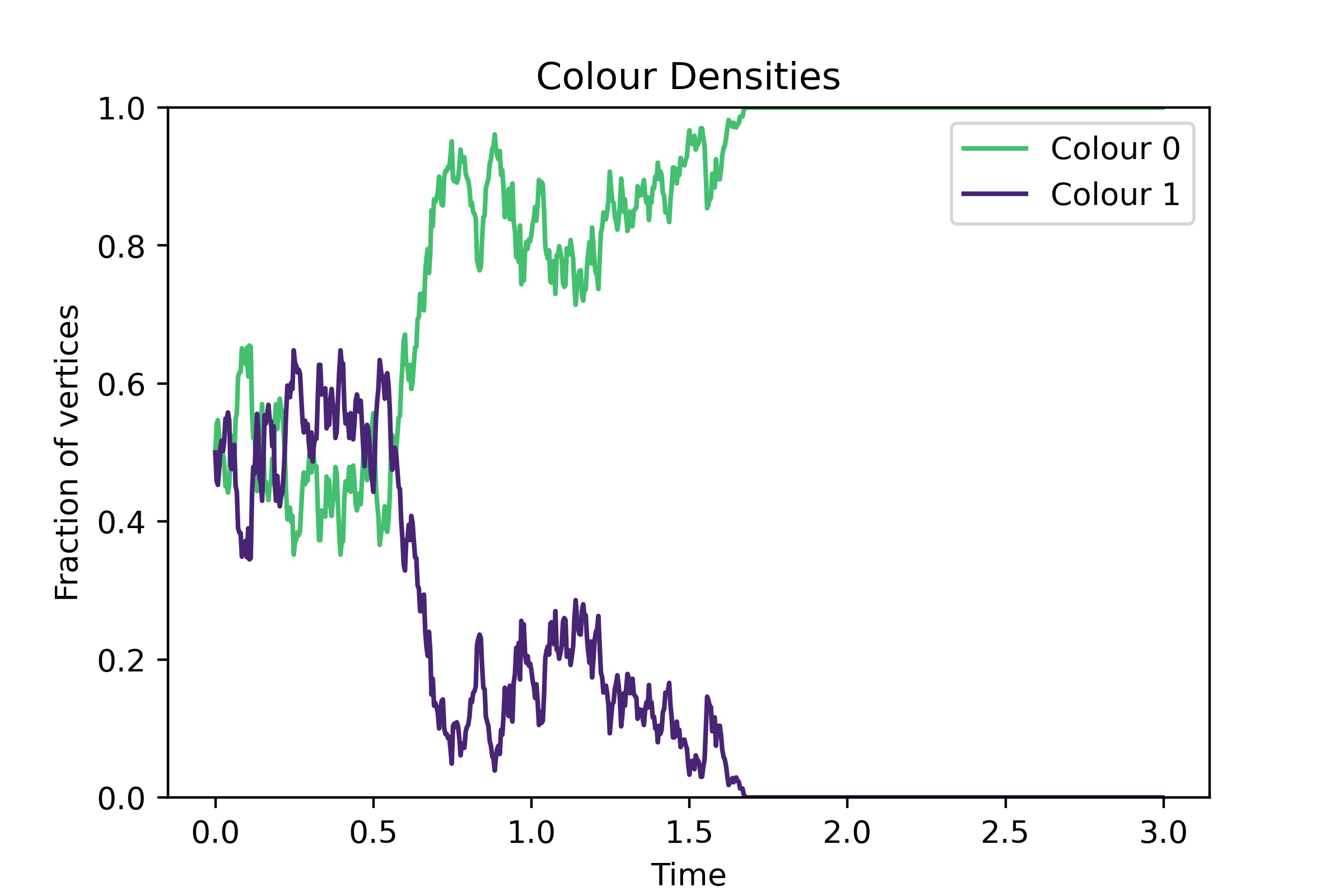}
\includegraphics[width=0.475\linewidth]{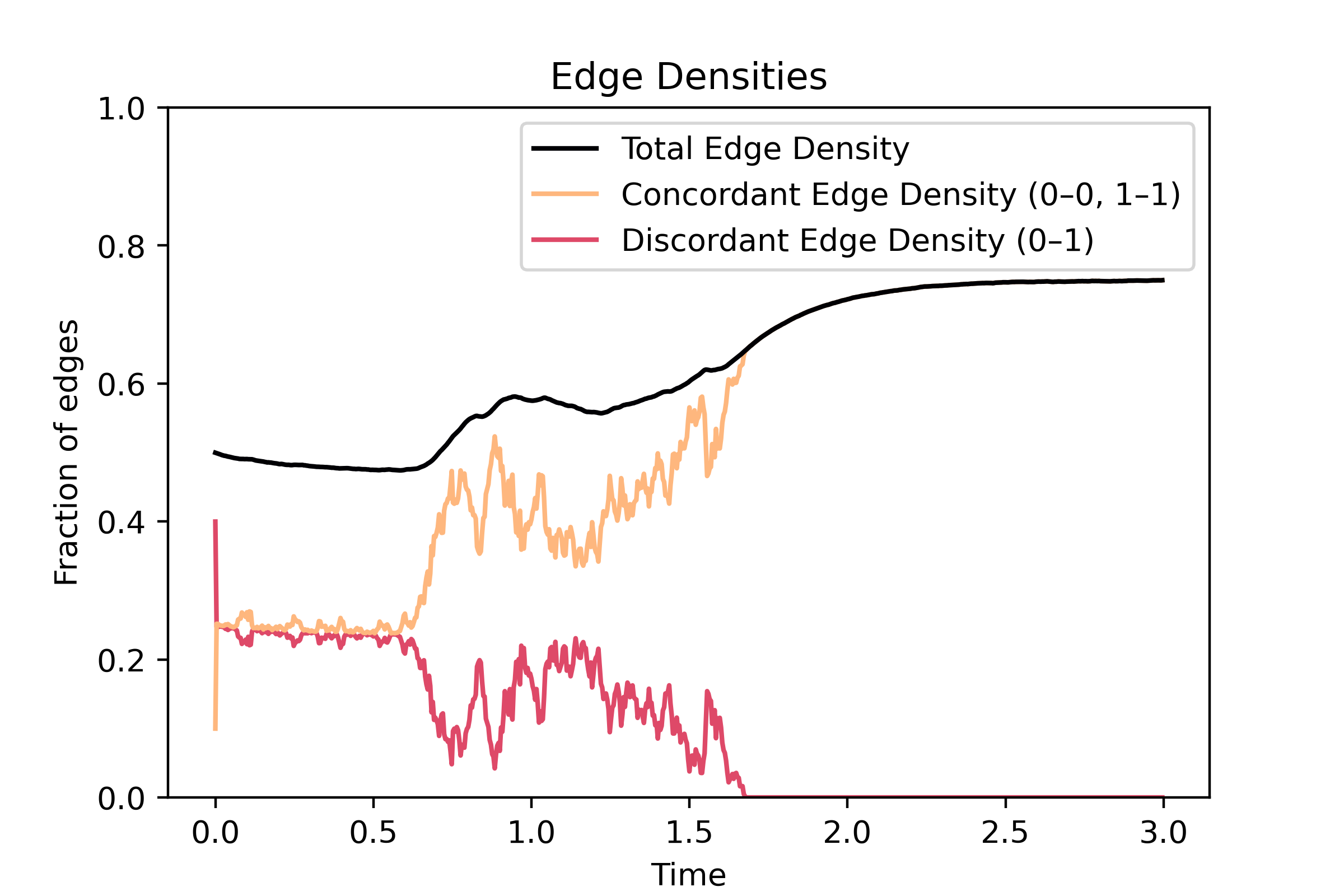}
\caption{These plots visualise a simulation for $n=1000$ with $\eta=1.0$, $\rho=1.1$, $s_{\mathrm{c},0}=1.5$, $s_{\mathrm{d},0}=0.7$, $s_{\mathrm{c},1}=0.5$, and $s_{\mathrm{d},1}=2.0$. Since the diffusion coefficient does not decrease to zero, the process eventually gets absorbed in a state of consensus.}
\label{fig1}
\end{figure}

\paragraph{Simulation.} 
Figure~\ref{fig1} shows a simulation for $n=1000$, $\eta=1.0$, $\rho=1.1$, $s_{\mathrm{c},0} = 1.5$, $s_{\mathrm{c},1} = 0.5$, $s_{\mathrm{d},0} = 0.7$, $s_{\mathrm{d},1} = 2.0$. The graph is initialised as follows: 
\begin{itemize}
\item Vertices are placed at positions $I_n = \{0,\tfrac{1}{n},\ldots,\tfrac{n-1}{n}\} \subset [0,1]$; the first half receives colour $0$ and the second half receives colour $1$. 
\item
Concordant vertices $x,y \in I_n$ are connected with probability $0.1\,(1-|x-y|)$, discordant vertices $x,y \in I_n$ are connected with probability $0.9\,(1-|x-y|)$.
\end{itemize}
The choice of parameters is such that discordant edges are more likely to disconnect than to connect, while concordant edges are more likely to connect than to disconnect. The initial graph is therefore \emph{in conflict} with the dynamics: initially discordant edges are more likely to be present than concordant edges. Consequently, there is a \emph{rapid collapse} at the beginning of the evolution. The plots confirm that the fraction of connected edges evolves as a \emph{stochastic flow}, while the fractions of concordant and discordant connected edges and the fractions of colour $0$ and colour $1$ evolve as \emph{diffusions}. After consensus is reached, the fraction of connected edges evolves as a \emph{deterministic flow}.

\begin{figure}[t!]
\includegraphics[width=0.475\linewidth]{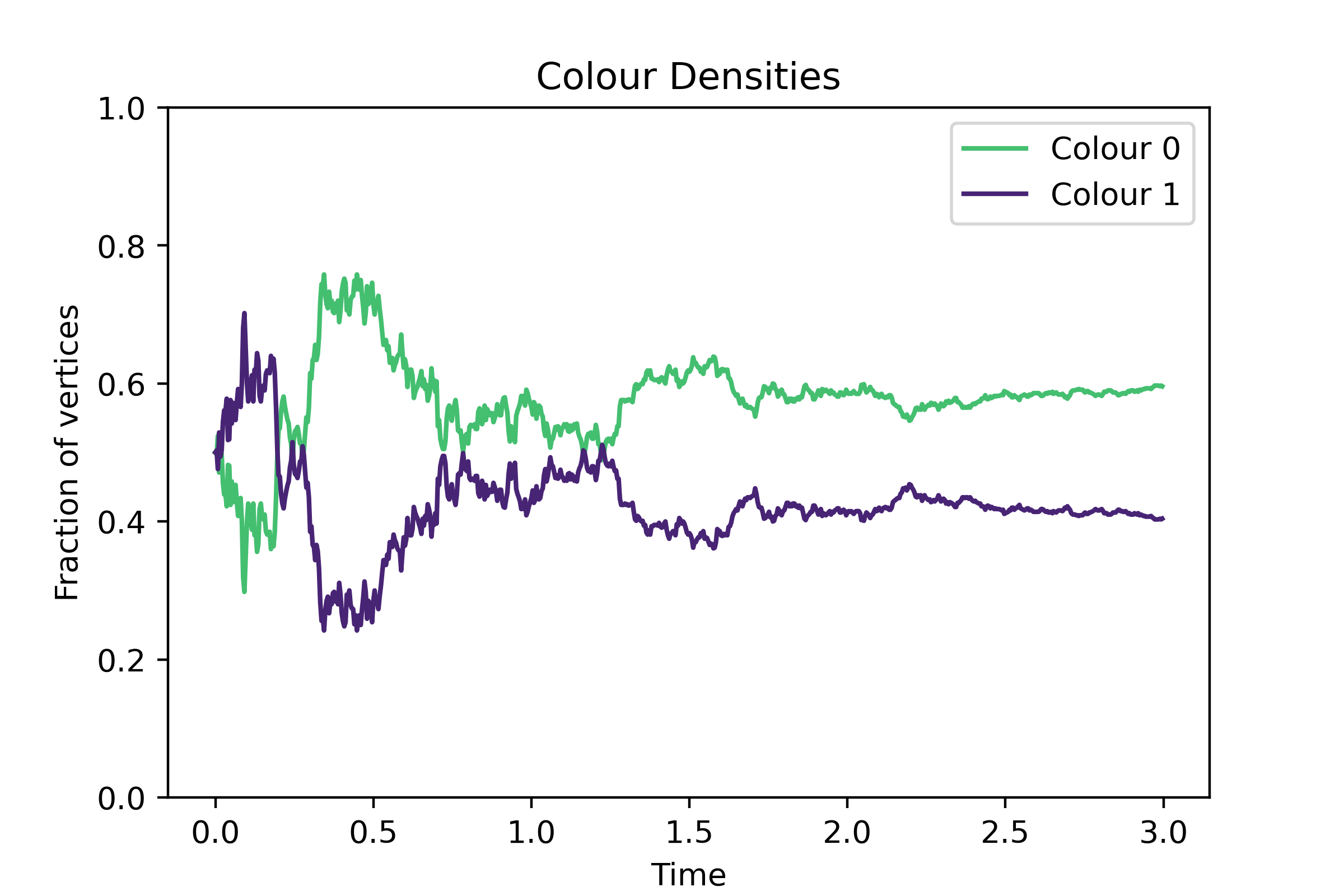}
\includegraphics[width=0.472\linewidth]{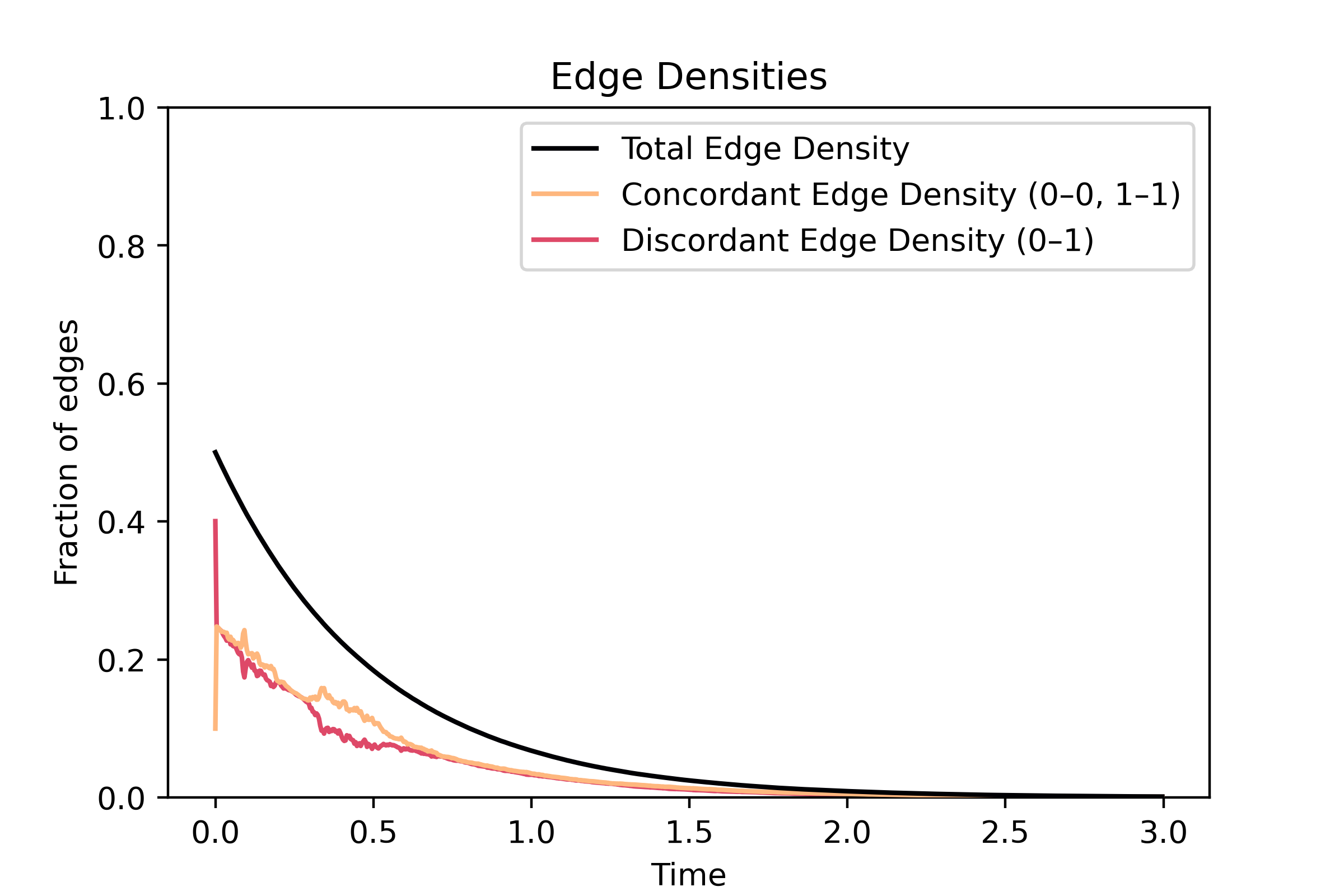}
\caption{These plots visualise a simulation for $n=1000$ with $\eta=1.0$, $\rho=2.0$, $s_{\mathrm{c},0}=s_{\mathrm{d},0}=0.0$ and $s_{\mathrm{c},1}=s_{\mathrm{d},1}=1.0$. The left plot shows the colour densities for each respective colour. The right plot shows the overall edge density as well as the fraction of concordant and discordant edges. The rapid homogenisation is clearly visible close to time 0. Since the diffusion coefficient decreases to zero, the process eventually gets trapped in a state of polarisation.}
\label{fig2}
\end{figure}

Figure~\ref{fig2} shows a simulation for $n=1000$, $\eta=1.0$, $\rho=2.0$, $s_{\mathrm{c},0}=s_{\mathrm{d},0}=0$ and $s_{\mathrm{c},1}=s_{\mathrm{d},1}=1$, with the same initialisation. The plots confirm that the fractions of concordant and discordant connected edges tends to zero rapidly, and the fractions of colour $0$ and colour $1$ tend to non-trivial limiting values. After polarisation is reached, the system is frozen.


\paragraph{Summary of proofs.}
Our proof of Theorem~\ref{thm1} will use \eqref{6}, as explained in Lemma~\ref{lem8} below. Our proof of Theorem~\ref{thm2} is based on an analysis of the evolution of the coloured subgraph densities via associated generators. To enable this, we develop a  theory for dense coloured graph sequences converging to coloured graphons in the cut-metric and subgraph-density-metric, as laid out in Lemmas \ref{lem1}--\ref{lem4} below. Coloured subgraph densities are a powerful tool, yet are difficult to work with. 

In our setting, standard approaches to convergence fail, and our limiting evolution cannot be properly defined on all of the state space. Therefore a new approach has to be developed, involving a \emph{projection} onto the subset of the space of coloured graphons where vertex states \emph{homogenise}, while edge states evolve via stochastic flows governed by the overall colour densities. The projection leads to a system of \emph{coupled equations} for the evolution of the coloured subgraph densities on the subset, which can be solved in closed form despite the fact that there is no moment closure in the original dynamics. A similar approach was used by \cite{AdHR24} for the analysis of a Moran model with random resampling rates.

We use the path topology of convergence in measure (also referred to as the Meyer-Zheng path topology) to show that there is a collapse of the joint dynamics onto the above subset. In Proposition \ref{prop1} below we show that convergence of finitely many coloured subgraph densities in the path topology of convergence in measure is equivalent to that of coloured graphons.

Using the form of the generator along with Taylor expansion, we provide a technical computation that yields a workable expression for the evolution of the subgraph densities in Lemma~\ref{lem6}. The evolution of the coloured subgraph densities under the projection to the subset is subsequently described in Lemma~\ref{lem7}. We show in Lemma~\ref{lem10} that the distance (induced by the path topology of convergence in measure) between the finite coloured graph process and its projection converges to zero as $n \to \infty$. A crucial ingredient is the classical \emph{duality} of the voter model with coalescing random walks, in combination with the fact that the continuous-time random walk on a dense graph sequence \emph{mixes rapidly}, as shown in Lemma~\ref{lem11}. In Lemma~\ref{lem16} we show that the projection of the finite coloured graph process converges weakly to the limiting process as $n\to\infty$ in the path topology of convergence in measure. The proof of this lemma uses Proposition \ref{prop1}(c), convergence of the generators of the projection of the finite coloured graph processes, along with \cite[Theorem~1.6, Chapter~8]{EK86}. Afterwards the proof of Theorem~\ref{thm2} readily follows from straightforward observations on the metric that induces path topology of convergence in measure (see Remark \ref{rem1} below).


\paragraph{Topology and scaling.} 
The path topology of convergence in measure (also referred to as the Meyer-Zheng path topology) is robust against deviations from the limit over small time intervals. This fact ensures that any collapse at the start of the process does not affect the convergence, and simplifies the analysis. The mathematical tools available to handle such convergence in a stronger path topology are limited. While it may be possible to establish a stronger form of convergence on the subspace itself, it seems unlikely that the full process can be shown to converge in a stronger path topology, since the fluctuations around the subspace are unlikely to be uniformly bounded. To achieve stronger results, it may be necessary to incorporate these fluctuations into the limiting process.

With our scaling, our continuous-time model exhibits order $\eta n^2$ vertex updates and order $\rho n^2$ edge updates per unit time (provided at least one of the $s$-parameters is non-zero). This contrasts with the scaling of the discrete-time model in \cite{BS17}, where after $n^2$ time steps there are order $\beta n$ vertex updates and order $n^2-\beta n$ edge updates. The scaling of \cite{BS17} can be replicated with our model by setting $\eta = \frac{\beta}{n}$, but this regime is not the focus of  our paper.


\paragraph{Outline.}
The rest of the paper is organised as follows. In Section~\ref{sec6} we discuss limit theory of  dense coloured graph sequences. More specifically, in Section~\ref{sec7} we define coloured graph processes and their dense limits, which are coloured graphon processes and in  Section~\ref{sec8} we recall some basic facts about the path topology of convergence in measure, also referred to as the Meyer-Zheng path topology. In Section~\ref{sec10}, using Taylor approximation, we write down the generators of coloured subgraph densities and their projections. The calculations required to do so are technical and long. In Section~\ref{sec14} we use these results to prove Theorems~\ref{thm1}--\ref{thm2}. In Appendix~\ref{sec17} we present details of a key step in the proof  of Lemma~\ref{lem7} which  describes the action of the generator on injective subgraph densities.


\section{Limits of dense coloured graph processes}
\label{sec6}

In this section we present the tools that are needed throughout the paper. In Section~\ref{sec7} we introduce dense coloured graphs, coloured graphons and metric space of equivalence classes that govern them. In Section~\ref{sec8} we collect some essential facts about the path topology that we are working with. In Section~\ref{sec9} we present a result that mirrors a result on weak convergence for graphon-valued stochastic processes from \cite[Section~3.1]{AdHR21}. We  extend the observations therein to the dense coloured graph setting and to the path topology of convergence in measure (see Proposition~\ref{prop1} below).


\subsection{Dense coloured graphs and coloured graphons}
\label{sec7}

In this section we prove four lemmas (Lemmas \ref{lem1}--\ref{lem4}) that together provide the tools that are necessary for the description via coloured dense graphs and coloured graphons. These are straightforward extensions of dense graphs and graphons introduced and analysed by \cite{LS06}, \cite{BCLSV08} and \cite{L12}.

We assume that all coloured graphs are simple, that is, without loops or multiple edges. We consider a coloured graph $G$ to be the triple 
\be{
  G = (V(G), E(G), c\colon V(G) \to \{0,1\}),
} 
where $V(G)$ is the finite set of vertices on the graph, usually taken to be $[n] = \{1,\ldots,n\}$ for some $n \in \N$, $E(G) \subset \{ \{i,j\}\colon i,j \in [n]\}$ is the set of edges, and $c_i$ is the colour of vertex $i \in [n]$. As is common practice in graph theory, we denote the number of vertices in $G$ by $v(G)=\abs{V(G)}$ and the number of edges in $G$ by $e(G)=\abs{E(G)}$. Throughout, we use $0$ for a black vertex (denoted by $\bv$) and $1$ for a white vertex (denoted by $\wv$), while $w(G)$ and $b(G)$ denote the number of white, respectively, black vertices in $G$. Thus, $w(G)+b(G)=v(G)$.

As before, let $\cG_n$ be the set of all coloured graphs on $n$ vertices with two colours ($0$ = black or $1$ = white). A sequence $(G^n)_{n\in\N}$ of coloured graphs is called a \emph{dense coloured graph sequence} if $\liminf_{n\to\infty} e(G^n)/n^2>0$. We let $\cW$ be the space consisting of all pairs $(\kappa,c)$, where $\kappa\colon[0,1]^2 \to [0,1]$  is a measurable function {(the usual graphon)} satisfying $\kappa(x,y) = \kappa(y,x)$ for all $(x,y) \in [0,1]^2$, and where $c\colon [0,1] \to [0,1]$ is a measurable function, {called a \emph{vertex colouring} or simply a \emph{colouring}}. {We call $c$ a \emph{pure} colouring if $c(x) \in \{0,1\}$ for all $x \in [0,1]$}. We call the pair $(\kappa,c)\in \cW$ a \emph{coloured graphon}. Any coloured graph $G\in\cG_n$ can be represented canonically as a coloured graphon via the pair
\ben{\label{11}
  \bs{\kappa^n(x,y) 
  &= \begin{cases}
  1 &\text{if there is an edge between vertex $\lceil{nx}\rceil$ and vertex $\lceil{ny}\rceil,$}\\
  0 &\text{otherwise,}
  \end{cases}\\
  c^n(x) 
  &= \begin{cases}
  1 &\text{if vertex $\lceil{nx}\rceil$ is white,}\\
  0 &\text{otherwise.}
  \end{cases}
  }
}
We define a pseudo-metric on $\cW$ by using a coloured subgraph density metric. Let $\cF_k$ denote the set of isomorphism classes of finite coloured graphs given by~$\cF = \{F_i\}_{i\in\N}$, with each $F_i$ being a representative of an isomorphism class. We can define the \emph{coloured subgraph distance} of two coloured graphons $(\kappa_1,c_1)$ and $(\kappa_2,c_2)$ by 
\ben{ \label{12}
  \dsub\bclr{(\kappa_1,c_1),(\kappa_2,c_2)} =  \sum_{i \in\N} 2^{-i} |t_{F_i}(\kappa_1,c_1) - t_{F_i}(\kappa_2,c_2)|, 
}
where, for any coloured graph $F$ on $m$ vertices and any coloured graphon~$(\kappa,c)$,
\ben{ \label{13}
  t_{F}(\kappa,c) = \int_{[0,1]^m}  \dd{x_1} \cdots \dd{x_m}\prod_{\{i,j\} \in E(F)} \kappa(x_i,x_j) 
  \prod_{i=1}^m c(x_i)^{c^F_i}(1-c(x_i))^{1-c^F_i}
}
represents the coloured subgraph density of $F$ in $(\kappa,c)$. If $G\in\cG_n$ is a coloured graph, then $t_{F}(G)$ is defined through its canonical embedding into $\cW$. The coloured graphon representation of a coloured graph is not one-to-one: indeed, the representation in \eqref{11} depends on the ordering of the vertices. Since typically we are interested only in graph properties that are independent of vertex labels (like coloured subgraph densities), it is natural to consider \emph{equivalence classes} of coloured graphons obtained by letting two coloured graphons be equivalent when they are identical up to vertex labels in some sense. {Some care is needed here. The following lemma shows that working with general colourings does not give us more generality than working with pure colourings as far as $d_{\sub}$ is concerned. However, general colourings give us a much wider range of representatives of the equivalence classes to choose from, which will be useful to formulate our main results.}

{\begin{lemma}\label{lem17} Let $(\kappa,c)\in \cW$ be a coloured graphon. There exists a coloured graphon $(\kappa',c')\in \cW$ with pure colouring $c'$such that, for every coloured graph $F$,
\be{
  t_F(\kappa,c) = t_F(\kappa',c'),
}
and hence, $\dsub\bclr{(\kappa,c),(\kappa',c')} = 0$.
\end{lemma}
\begin{proof}
Set $p = \int_0^1 c(y)\dd y$ and define $C,D\colon[0,1]\to[0,1]$ by
\be{
  C(x) = \int_0^x c(y)\dd y,\qquad D(x) = \int_0^x (1-c(y))\dd y,
}
so $C(1)=p$ and $D(1)=1-p$. Both functions are non-decreasing and $1$-Lipschitz, hence continuous. Let
\be{
  C^{-1}(u) = \inf\bclc{x\in[0,1]\colon C(x)>u},\qquad D^{-1}(v) = \inf\bclc{x\in[0,1]\colon D(x)>v}
}
be their right-continuous generalised inverses, and define $\phi\colon[0,1]\to[0,1]$ by
\be{
  \phi(u) = \begin{cases} C^{-1}(u) & \text{if } u\in[0,p),\\ D^{-1}(u-p) & \text{if } u\in[p,1].\end{cases}
}
Now set $\kappa'(u,v) = \kappa(\phi(u),\phi(v))$ and $c'(u) = \I[u\in[0,p]]$, so that $c'$ takes values in $\{0,1\}$ as required.
We claim that, for every bounded measurable $f\colon[0,1]\to\R$,
\ben{ \label{lem17pf}
  \int_0^p f(\phi(u))\dd u = \int_0^1 f(x)\, c(x)\dd x,\qquad \int_p^1 f(\phi(u))\dd u = \int_0^1 f(x)(1-c(x))\dd x.
}
Indeed, by continuity of $C$, for $u\in[0,p)$ and $t\in[0,1]$, $C^{-1}(u)\leq t$ if and only if $u\leq C(t)$. Hence
\be{
  \int_0^p \I[C^{-1}(u)\leq t]\dd u = C(t) = \int_0^t c(x)\dd x,
}
which is the first identity of \eqref{lem17pf} for $f=\I[\cdot\leq t]$. By a standard monotone class argument this extends to all bounded measurable $f$. The second identity follows by the same argument applied to $D$.
From \eqref{lem17pf} and $c'(u)=\I[u\in[0,p]]$, we have for every bounded measurable $g\colon[0,1]\to\R$ and every $a\in\{0,1\}$
\ben{ \label{lem17pf2}
  \int_0^1 g(\phi(u))\,c'(u)^a(1-c'(u))^{1-a}\dd u = \int_0^1 g(x)\,c(x)^a(1-c(x))^{1-a}\dd x.
}
For a coloured graph $F$ on $k$ vertices, the definition of $t_F$ together with $\kappa'(u,v)=\kappa(\phi(u),\phi(v))$ gives
\be{
  t_F(\kappa',c') = \int_{[0,1]^k}\prod_{\{i,j\}\in E(F)}\kappa(\phi(u_i),\phi(u_j))\prod_{i=1}^k c'(u_i)^{c^F_i}(1-c'(u_i))^{1-c^F_i}\dd{u_1}\cdots\dd{u_k}.
}
The integrand is bounded and depends on each $u_i$ only through $\phi(u_i)$ and the colour factor, so applying \eqref{lem17pf2} coordinate by coordinate via Fubini gives
\be{
  t_F(\kappa',c') = \int_{[0,1]^k}\prod_{\{i,j\}\in E(F)}\kappa(x_i,x_j)\prod_{i=1}^k c(x_i)^{c^F_i}(1-c(x_i))^{1-c^F_i}\dd{x_1}\cdots\dd{x_k} = t_F(\kappa,c),
}
where the last equality is \eqref{13}. The vanishing of $\dsub$ then follows from~\eqref{12}.
\end{proof}
}

{%
\begin{figure}[t!]
\centering
\includegraphics[width=\linewidth]{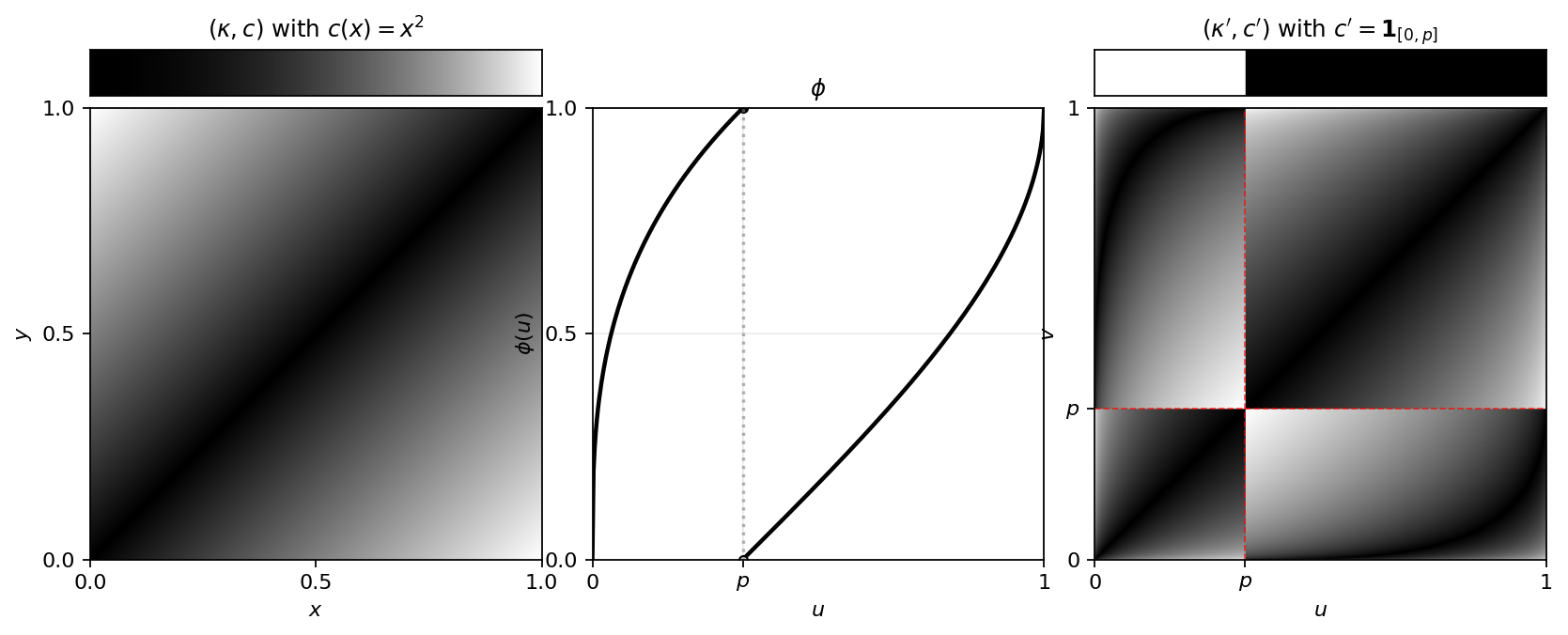}
\caption{{Illustration of the construction in Lemma~\ref{lem17} for $\kappa(x,y) = 1 - |x-y|$ and the fractional colouring $c(x) = x^2$, for which $p = \int_0^1 c(y)\dd y = 1/3$. The middle panel shows the map $\phi$, which on $[0,p)$ uses the inverse of $C(x) = \int_0^x c(y)\dd y = x^3/3$, giving $\phi(u) = (3u)^{1/3}$, and on $[p,1]$ uses the inverse of $D(x) = \int_0^x (1-c(y))\dd y = x - x^3/3$ shifted by $p$. The transformed graphon $\kappa'(u,v) = \kappa(\phi(u),\phi(v))$ consists of four reparametrised copies of $\kappa$, one on each of the blocks $[0,p]^2$, $[0,p]\times[p,1]$, $[p,1]\times[0,p]$, $[p,1]^2$ (dashed red lines), while the colouring $c' = \I[\cdot\in[0,p]]$ is pure.}}
\label{fig:lem17}
\end{figure}
}

Now, let $\Sigma$ denote the set of all measure-preserving maps on $[0,1]$. For any $\lambda \in \Sigma$ and $(\kappa,c) \in \cW$, define
\ben{ \label{14}
  \kappa^\lambda(x,y) =  \kappa(\lambda(x), \lambda(y)), 
  \qquad 
  c^\lambda(x) =  c(\lambda(x)).
}
{Now, if $c_1$ and $c_2$ are pure colourings,} then $(\kappa_1,c_1) \sim (\kappa_2,c_2)$ if and only if there exist $\lambda_1,\lambda_2\in\Sigma$ such that $(\kappa^{\lambda_1}_1,c^{\lambda_1}_1) =(\kappa^{\lambda_2}_2, c^{\lambda_2}_2)$; see \cite[Corollary~10.35]{L12}. This defines a quotient space $\~\cW$ on which $\dsub$ is in fact a metric. 

{For coloured graphons with pure colourings,} we can define an alternative distance $d_{\Box}$ between $(\kappa_1,c_1), (\kappa_2,c_2)\in \cW$ by
\bgn{ 
\label{15}
d_{\Box}((\kappa_1,c_1), (\kappa_2,c_2)) =  \inf_{\lambda \in \Sigma} 
\bbclr{\norm{\kappa_1 -\kappa^\lambda_2}_{\Box} + \norm{c_1 -c^\lambda_2}_{\Box}}, \\
\norm{\kappa}_\Box = \sup_{A,B \subset [0,1]} \bbbabs{ \int_{A \times B} \kappa(x,y)\dd x\dd y }, 
\qquad
\norm{c}_\Box = \sup_{A\subset [0,1]} \bbbabs{\int_{A} c(x)\dd x}. \nonumber
}
The relation between the two distances is captured by the following lemma. The first part is a so-called \emph{counting lemma}, the second part is a so-called \emph{inverse counting lemma}.

\begin{lemma} 
\label{lem1} Let $(\kappa_1,c_1),(\kappa_2,c_2) \in \cW$ { and assume $c_1$ and $c_2$ are pure colourings.}
\begin{enumerate}
\item[(a)]
For any coloured graph $F$,
\be{
  |t_F(\kappa_1,c_1) - t_F(\kappa_2,c_2)| \leq e(F) \, \norm{\kappa_1-\kappa_2}_{\Box} + \norm{c_1-c_2}_{\Box}.
}
\item[(b)] 
If, for some $k\in\N$,
\be{
  |t_F(\kappa_1,c_1) - t_F(\kappa_2,c_2)| \leq  2^{-k^2} \qquad\text{for all $F$ on $k$ vertices,}
}
then { there exists $C>0$ such that}
\be{ 
  d_{\Box}\bclr{(\kappa_1,c_1), (\kappa_1,c_1)} \leq \frac{C}{\sqrt{\log k}}.
}
\end{enumerate}
\end{lemma}

\begin{proof}
The proof of part (a) and (b) is analogous to the proof of \cite[Lemma~10.25]{L12} and \cite[Lemma~10.32]{L12}, respectively.
\end{proof}

The following result characterises the space of coloured graphons endowed with $d_{\Box}$ as a compact metric space.

\begin{lemma} 
\label{lem2}
$(\tilde{\cW},d_{\Box})$ is a compact metric space, and therefore is complete and separable. Furthermore, a sequence $(\kappa_n, c_n)$
converges in the metric $\dsub$ if and only if {any of its pure colouring representations} is a Cauchy sequence with respect to $d_{\Box}$.  
\end{lemma}

\begin{proof}
The  proof of the first statement about compactness is analogous to the proof of \cite[Theorem~9.23]{L12} for the uncoloured version. The proof uses the Szemer\'edi partition lemma, as in \cite[Lemma~9.15]{L12}, which can also be extended analogously. The second statement readily follows from Lemma~\ref{lem1}.
\end{proof}

\begin{lemma}
\label{lem3} 
The linear span of coloured subgraph densities is dense in $C(\tilde{\cW},d_\Box)$.
\end{lemma}
\begin{proof} 
The proof is analogous to the proof of \cite[Theorem~2.2]{DGKR15} for the uncoloured version.
\end{proof}

The key result from dense graph theory also holds in the setting of dense coloured graphs.

\begin{lemma} \label{lem4}
Let $(G_n)_{n \in \N}$ be a dense coloured graph sequence that is a Cauchy sequence with respect to $\dsub$. Then there exists a coloured graphon $(\kappa,c) \in \cW$ such that
\be{ 
\lim_{n\to\infty} \dsub\clr{(\kappa^n,c^n),(\kappa,c)} = 0. 
}
\end{lemma}

\begin{proof}
The result follows from Lemma~\ref{lem2}.
\end{proof}

A convenient way of building a random coloured graph $G^n$ on $n$ vertices from a coloured graphon $(\kappa,c)$ is the following. Let $U_1,\dots,U_n$ be i.i.d.\ uniform $[0,1]$-valued random variables. For each pair of vertices $i$ and $j$, connect $i$ and $j$ with probability $\kappa(U_i,U_j)$, independently of all the other edges. Assign to each vertex colour $1$ with probability $c(U_i)$ and colour $0$ with probability $1-c(U_i)$. It is not difficult to prove that
\be{
\lim_{n\to\infty} \dsub\bclr{(\kappa^n,c^n), (\kappa,c)} = 0 \quad \text{ almost surely},
}
which is a law of large numbers for the coloured subgraph density metric. The proof again follows by analogy with the uncoloured version (see, for example, \cite[Theorem~2.1]{AR16}). 

For finite coloured graphs it is typically easier to work with the set of \emph{injective} homomorphisms, which we denote by $\inj(F,G)$. For $F$ a coloured graph on $k$ vertices and $G\in\cG_n$, we define the injective coloured subgraph density of $F$ in $G$ to be
\be{
  t^{\inj}_{F}(G) =
  \begin{cases}
  \displaystyle\frac{\abs{\inj(F,G)}}{(n)_k} \in[0,1], &\text{if } k\leq n,\\
  0, &\text{otherwise.}
  \end{cases} 
}
It is easy to see that
\ben{ \label{16}
  \babs{t^{\inj}_{F}(G^n)-t_{F}((\kappa^n,c^n))} \leq \frac{C_F}{n}
}
for some constant $C_F$ that depends on $F$. Hence, as far as the metric $\dsub$ is concerned, it makes little difference whether we are using $t^{\inj}_{F}$ or $t_{F}$ in the large-$n$ limit.

Finally, for any uncoloured graph $\-F$ on $m$ vertices, we define the usual uncoloured subgraph density of $\-F$ in $(\kappa,c)$ to be 
\be{    
  t_{\bar{F}}(\kappa) = \int_{[0,1]^m}  \dd{x_1} \cdots \dd{x_m} \prod_{\{i,j\} \in E(F)} \kappa(x_i,x_j).
}
It is clear that if $\cC(\-F)$ is the set of all coloured graphs that can be obtained by colouring the vertices of $\-F$ in all possible ways, then for any $(\kappa,c)\in\cW$ we have
\be{
  t_{\bar{F}}(\kappa) = \sum_{H\in\cC(\-F)} t_{H}(\kappa,c).
}


\subsection{Path topology of convergence in measure}
\label{sec8}

We now collect and discuss just the essential facts about the path topology of convergence in measure that are required for our purposes. This topology was introduced by \cite{MZ84} (building on earlier work) through so-called pseudopaths: viewing paths as point-measures on the product of time and state space, and then defining the topology via weak convergence on that product space. However, a more illuminating approach for our setting was developed by \cite{K91}, who works directly with paths (or rather, their usual equivalence classes under almost-everywhere equality) and defines the topology by means of a suitable metric.

In order not to obfuscate these abstract concepts with the specifics of our framework, we work here with a general metric space $(E,r)$. Let $\cM_{E,r}$ be the space of measurable functions $x\colon \R_+ \to E$, identified under equality almost everywhere, so that all such functions that differ only on a null set are considered the same element. For $x,y \in \cM_{E,r}$ with $x = (x_t)_{t \ge 0}$ and $y = (y_t)_{t \ge 0}$, define
\ben{ \label{17}
\dMZ(x,y) = \int_0^\infty \bclr{1\wedge r(x_t,y_t)} \ee^{-t}\dd t.
}
As shown by \cite[Section~4]{K91}, this metric renders $\cM_{E,r}$ a complete, separable metric space. Hence, tightness and relative compactness are equivalent on $\cM_{E,r}$ by Prohorov's Theorem. An alternative metric that may make the term ``convergence in measure'' more transparent, is given by
\be{
 \dm(x,y) = \inf\bbclc{\eps>0 : \int_0^\infty \I[r(x_t,y_t)>\eps]\ee^{-t}\dd t \leq \eps};
}
see \cite[Section~3]{Aldous89}. The two metrics are related by the following lemma. We omit the proof, as it follows from standard techniques.

\begin{lemma} 
For any $x,y\in\cM_{E,r}$,
\ben{\label{18}
\tfrac12 \dMZ(x,y) \leq \dm(x,y) \leq \sqrt{\dMZ(x,y)}.
}
\end{lemma}

Our first result is a characterisation of tightness in $(\cM_{E,r},\dMZ)$. 

\begin{theorem}[{\cite[Theorem~4.6]{K91}}]
\label{thm3} 
Let $(X^n)_{n\in\N}$ be a sequence of $\cM_{E,r}$-valued processes. The family is tight in $(\cM_{E,r},\dMZ)$ if and only if the following hold:
\begin{itemize}
\item[$(i)$] For every $\eps>0$, there is a compact $K\subset E$ such that\/ $\sup_{n\in\N} \int _0^\infty \IP(X^n_t\not\in K)\ee^{-t}\dd t \leq \eps$.
\item[$(ii)$]
$\lim_{h \downarrow 0}\sup_{n\in\N} \int_0^\infty\IE\cls{1\wedge r(X^n_{t},X^n_{t+h})}\ee^{-t}\dd t = 0$.
\end{itemize}
\end{theorem}

It is clear that, if $(E,r)$ is compact, then $(i)$ in Theorem~\ref{thm3} is automatically satisfied. This will be the case in our setting.

Our second result is a characterisation of weak convergence in $(\cM_{E,r},\dMZ)$.

\begin{theorem}[{\cite[Theorem~4.8]{K91}}]
\label{thm4} 
Let $(X^n)_{n\in\N}$ be a sequence of $\cM_{E,r}$-valued processes that is tight. Then $X^n$ converges weakly in $(\cM_{E,r},\dMZ)$ if and only if 
\ben{\label{19}
\int_{\IR_+^k} \IE[f(X^n_{t_1},\dots,X^n_{t_k})] \ee^{-(t_1+\dots+t_k)}\dd{t_1}\cdots \dd{t_k}
}
converges for every $k\in\N$ and every $f\in C_b(E^k)$. 
\end{theorem}

While convergence of finite-dimensional distributions implies \eqref{19}, the converse is not true.

The following remark, while straightforward to prove, is useful to conceptualise the proof of Theorem~\ref{thm2}.

\begin{remark} \label{rem1}
Let $(X^n)_{n\in\N}$ and $(Y^n)_{n\in\N}$ be two sequences of $\cM_{E,r}$-valued processes, and let $Z$ be a $\cM_{E,r}$-valued processes.\\ 
$(a)$ If
\ben{\label{20}
\IE\cls{1\wedge r(X^n_t,Y^n_t)} \to 0 \qquad \text{uniformly in $t\in(0,T)$ for every $T>0$}, 
}
then 
\ben{\label{21}
  \lim_{n\to\infty} \IE[\dMZ(X^n,Y^n)] = 0.
}
$(b)$ If $Y^n$ converges weakly to $Z$ as $n\to\infty$ in $(\cM_{E,r},\dMZ)$ and \eqref{21} holds, then $X^n$ also converges weakly to $Z$ as $n\to\infty$ in $(\cM_{E,r},\dMZ)$.
\end{remark}


\subsection{Weak convergence of coloured graphon processes}
\label{sec9}

In this section, we use the setup of Section~\ref{sec8} to state and prove a proposition (Proposition~\ref{prop1}) that characterises weak convergence of coloured graphon-valued processes in the path topology of convergence in measure.

Since $(\tilde\cW,\dsub)$ is a metric space, we can define the path topology of convergence in measure on $\tilde\cW$-valued paths in the usual way (see Section~\ref{sec8} and \cite{K91}). We denote by $\cM_{\~\cW}$ the corresponding space of all measurable paths from $[0,\infty)$ to $\~\cW$, endowed with a metric $\dMZ$ that induces the path topology of convergence in measure (see Section~\ref{sec8}), which turns $\cM_{\~\cW}$ into a complete and separable metric space. We use ``$\Longrightarrow$'' to denote weak convergence on $\cM_{\~\cW}$ with respect to the Borel-sigma-algebra induced by $\dMZ$.

Let $(\tilde{\kappa},\tilde{c})$ be a $\tilde\cW$-valued stochastic process, that is, a random element of $\cM_{\~\cW}$. We will denote the value of $(\tilde{\kappa},\tilde{c})$ at time~$s \geq 0$ by  $(\tilde{\kappa}_s,\tilde{c}_s) \in \tilde\cW$, which is an equivalence class of coloured graphons. If $(\kappa_s,c_s)$ is a representative graphon of the equivalence class $(\tilde{\kappa}_s,\tilde{c}_s)$, then we write $\kappa_s(x,y)$ to denote the value of that graphon evaluated at the coordinates $(x,y)\in[0,1]^2$ and $c_s(x)$ to denote the value of the colour evaluated at $x \in [0,1]$. Note that, for a given $(\tilde{\kappa},\tilde{c}) \in \cM_{\~\cW}$ and a given finite coloured graph $F$, we can consider the induced real-valued process $t_F(\tilde\kappa, \tilde{c})$ as an element of $\cM_{[0,1]}$. Then $t_F(\tilde{\kappa}_s, \tilde{c}_s)$ equals the value of the process $t_F(\tilde{\kappa}, \tilde{c})$ at time~$s\geq 0$.

\begin{proposition}{\rm (Characterisation of convergence)} 
\label{prop1} 
Let $(\tilde{\kappa},\tilde{c})$ and $(\tilde{\kappa}^n, \tilde{c}^n)$, $n\in\N$, be $\cM_{\~\cW}$-valued processes. Then the following are equivalent:
\begin{enumerate}
  \item[(a)] $(\tilde{\kappa}^n, \tilde{c}^n)\Longrightarrow (\tilde{\kappa}, \tilde{c})$ as $n\to\infty$ in $\cM_{\~\cW}$.
  \item[(b)] For all $d \in \N$ and all coloured graphs $F_1,\dots,F_d$,
  \bmn{ \label{22}
    \bclr{t_{F_1}(\tilde{\kappa}^n, \tilde{c}^n),\dots,t_{F_d}(\tilde{\kappa}^n, \tilde{c}^n)}
    \Longrightarrow \bclr{t_{F_1}(\tilde{\kappa}, \tilde{c}),\dots,t_{F_d}(\tilde{\kappa}, \tilde{c})} \\
    \text{ in $\cM_{[0,1]^d}$  as $n\to\infty$}.
  }
\end{enumerate}
Moreover, if $(\tilde{\kappa},\tilde{c})$ and $(\tilde{\kappa}^n, \tilde{c}^n)$ are c\`adl\`ag, then the following condition implies both $(a)$ and~$(b)$:
\begin{enumerate}
  \item[(c)] For all $d \in \N$ and all coloured graphs $F_1,\dots,F_d$, the finite-dimensional distributions in the left-hand side of \eqref{22} converge to those in the right-hand side almost everywhere.
\end{enumerate}
\end{proposition}

\begin{proof}
\textit{$(a)\Rightarrow (b)$}: Let $f\colon \~\cW \to \IR^d$ be Lipschitz continuous. Then, for any sequence $x_1,x_2,\ldots\in\cM_{\~\cW}$ that converges to $x\in\cM_{\~\cW}$, we have
\be{
  \int_0^\infty \bclr{1\wedge \norm{f(x_n(t))-f(x(t))}}\ee^{-t}\dd t
  \leq C_f \int_0^\infty \dsub\bclr{x_n(t),x(t)}\ee^{-t}\dd t
}
for some $C_f>0$. This implies that the function that maps a path $y\in\cM_{\~\cW}$ to the path $t\to f(y(t))$, which is an element of $\cM_{\IR^d}$, is Lipschitz continuous. The claim now follows by the continuous mapping theorem, since coloured subgraph densities are Lipschitz functions on $\~\cW$.

\medskip\noindent\textit{$(b)\Rightarrow (a)$}: Write $X^n_t = (\~\kappa_t^n,\~c_t^n)$. We have
\be{
  \int_0^\infty \IE[\dsub\bclr{X^n_t,X^n_{t+h}}] \ee^{-t}\dd t
  = \sum_{i \in \N} 2^{-i} \int_0^\infty \IE[\abs{t_{F_i}(X^n_t)-t_{F_i}(X^n_{t+h})}] \ee^{-t}\dd t.
}
Since weak convergence implies relative compacteness, \eqref{22} implies 
\be{
  \lim_{h\downarrow 0} \sup_{n \in \N} \int_0^\infty \IE[\abs{t_{F_i}(X^n_t)-t_{F_i}(X^n_{t+h})}]\ee^{-t}\dd t = 0
}
for every $F_i$. Next, fix $\eps>0$ and let $M$ be such that $2^M\leq \eps/2$. For each $1\leq i\leq M$, choose $h_i$ such that
\be{
  2^{-i}\sup_{n \in \N} \int_0^\infty \IE[\abs{t_{F_i}(X^n_t)-t_{F_i}(X^n_{t+h})}]\ee^{-t}\dd t 
  \leq \frac{\eps}{2M} \qquad \text{ for all $0<h\leq h_i$.}
}
Letting $h_0 = \min\{h_1,\ldots,h_M\}$, we have
\bes{
  \mel \sup_{n \in \N} \sum_{i \in \N} 2^{-i} \int_0^\infty \IE[\abs{t_{F_i}(X^n_t)-t_{F_i}(X^n_{t+h})}]\ee^{-t}\dd t\\
    & \leq \sum_{i=1}^M 2^{-i} \sup_{n \in \N} \int_0^\infty \IE[\abs{t_{F_i}(X^n_t)-t_{F_i}(X^n_{t+h})}]\ee^{-t}\dd t 
    + \frac{\eps}{2}  \leq \eps
}
for all $0<h\leq h_0$. Since $\~\cW$ is compact, it follows from Theorem~\ref{thm3}, that the sequence $(X^n)_{n\in\N}$ is tight. For any $F_1,\dots, F_d$, let $\bt^n_s = \bclr{t_{F_1}(X^n_{s}),\dots,t_{F_d}(X^n_{s})}$, and write $\bbs = (s_1,\dots,s_k)$. Now, \eqref{22} and Theorem~\ref{thm4} imply that
\be{
  \int_{\IR_+^k} \IE[f(\bt^n_{s_1},\dots,\bt^n_{s_k})]\ee^{-(s_1+\dots+s_k)}\dd \bbs
  \longto 
  \int_{\IR_+^k} \IE[f(\bt_{s_1},\dots,\bt_{s_k})]\ee^{-(s_1+\dots+s_k)}\dd \bbs
}
for every $k\in\N$ and every $f\in C_b([0,1]^{d\times k})$, in particular, for functions of the form
\be{
  f(t_1,\dots,t_k) = \prod_{i=1}^k \bbclr{\sum_{j=1}^d a_{i,j}t_{i,j}}, \qquad t_1,\dots,t_k\in[0,1]^d, a_{1,1},\dots a_{k,d}\in\IR.
}
Since the linear span of subgraph densities is dense in $C_b(\~\cW)$ with respect to the uniform norm by Lemma~\ref{lem3}, it follows that
\ben{\label{23}
  \int_{\IR_+^k} \IE[f(X^n_{s_1},\dots,X^n_{s_k})]\ee^{-(s_1+\dots+s_k)}\dd \bbs
  \longto 
  \int_{\IR_+^k} \IE[f(X_{s_1},\dots,X_{s_k})]\ee^{-(s_1+\dots+s_k)}\dd \bbs
}
for every $k\in\N$ and every $f\in C_b(\~\cW^{k})$. Thus, Theorem~\ref{thm4} implies $(a)$.

\medskip\noindent\textit{$(c)\Rightarrow (b)$}: For each $F_i$, the process $t_{F_i}(X^n)$ is a c\`adl\`ag. The limit $t_{F_i}(X)$ is continuous, hence, c\`adl\`ag. Since the finite-dimensional distributions of $t_{F_i}(X^n)$ converge to those of $t_{F_i}(X)$ almost everywhere, it follows from \cite[Corollary~2.2]{BM16} (see also \cite{CK86} and \cite{G76}) that $t_{F_i}(X^n)\Longrightarrow t_{F_i}(X)$ in~$\cM_{[0,1]}$. This implies, in particular, that $t_{F_i}(X^n)$ is tight. Hence, the finite family $(t_{F_i}(X^n))_{1\leq i\leq d}$ is tight. The claim now follows from the assumptions and Theorem~\ref{thm4}, since convergence of finite-dimensional distributions almost everywhere can be shown to imply \eqref{23}.
\end{proof}

We conclude this section with some remarks on Proposition~\ref{prop1}.

\begin{remark}
\begin{enumerate}
\item[$(i)$] 
If $(\~\kappa^n,\~c^n)_{n\in\N}$ are induced by a sequence of coloured graph processes, then $t_{F_i}$ in the left-hand side of \eqref{22} can be replaced by $t^{\inj}_{F_i}$. 
\item[$(ii)$] 
The condition that all processes involved be c\`adl\`ag can be replaced by the condition that the processes take values in any family of paths for which almost-everywhere equality implies pointwise equality, such as continuous paths or c\`agl\`ad paths; see \cite{BM16}.
\item[$(iii)$] 
Condition $(c)$ is sufficient for weak convergence, but not necessary. One can easily construct examples of sequences of processes converging weakly in this path topology, but for which the finite-dimensional distributions converge nowhere. From convergence in this path topology, one can only conclude that the finite-dimensional distributions converge along a subsequence almost everywhere. 
\item[$(iv)$] 
It is possible to replace convergence of the finite-dimensional distributions by an integral version of finite-dimensional distributions in order for $(c)$ to be equivalent to $(a)$ and $(b)$ (see  Theorem~\ref{thm4}, \cite[Theorem~4.8]{K91}, and \cite[Theorem~2.6]{BM16}). We will not need this generality though, since we are only interested in the conclusion that $(c)$ implies $(b)$ and thus $(a)$.
\end{enumerate}
\end{remark}


\section{Action of generator on coloured subgraph densities and its projection}
\label{sec10}

The results in this section will be needed for the proof of the main results in Section~\ref{sec14}. All computations are done in terms of coloured subgraph densities of $G^n$, which is the appropriate language in view of the equivalence stated in Proposition~\ref{prop1}. In Section~\ref{sec11} we explain how to use multi-sets to simplify notation. In Section~\ref{sec12} using Taylor approximation in \eqref{1}, we derive the evolution of the coloured subgraph densities of $G^n$ under $\cA_n$. This provides an approximation up to order $1/n$ and helps to identify the subset $\cW_0$ onto which the process collapses (recall that $\cW_0$ is the subset of graphons for which the colouring is constant). The full derivation of the approximation is lengthy and technical, and therefore we present the main steps only; further details can be found in Appendix~\ref{sec17}. In Section~\ref{sec13} we compute the evolution of the projection $\cP$ of $G^n$ onto~$\cW_0$.


\subsection{Quick primer on signed multisets}
\label{sec11}

In the following, we will use multisets to describe the action of the generator $\cA_n$ on coloured subgraph densities. This is done mainly for the purpose of simplifying notation. 

A multiset is a set where each element can appear more than once (multiplicity). A signed multiset is a multiset where the multiplicity of each element can be any integer. For two signed multisets $\cA$ and $\cB$, the sum $\cA+\cB$ denotes the signed multiset where the multiplicity of each element is the sum of the multiplicities of the respective multisets. Similarly, the difference $\cA-\cB$ denotes the signed multiset where the multiplicity of each element is the difference of the multiplicities of the respective multisets. We use $\sum$ to denote multiset summation. 

For a multiset $\cA$, sums indexed by $\cA$ are understood to repeat each element of the multiset as many times as its multiplicity. For instance, if $\cA = \{a,a,b\}$, then $\sum_{i\in\cA} x_i = x_a + x_a + x_b$. 

If $\cA$ is a signed multiset, then sums indexed by $\cA$ are understood to repeat each element of the multiset as many times as its multiplicity for positive multiplicities, and to subtract each element of the multiset as many times as its multiplicity for negative multiplicities. For instance, if $\cA$ is a signed multiset containing elements $a$, $b$, $c$, $d$ with multiplicities $2$, $1$, $-1$, $-2$, respectively, then $\sum_{i\in\cA} x_i = x_a + x_a + x_b - x_c - x_d - x_d$.


\subsection{Evolution of coloured subgraph densities in the finite coloured graph process}
\label{sec12}

In this section we set up some notation and describe the action of the generator on injective homomorphisms in Lemma~\ref{lem5}. Our key result is Lemma~\ref{lem6}, where we derive an approximation of the action of the generator on subgraph densities up to order $\frac{1}{n}$. To do so is a technical challenge, with a fair amount of book-keeping and tedious computations, and requires a substantial amount of notation.

To state Lemma~\ref{lem5}, we need to invoke a number of graph operations acting on finite coloured graphs. Let $F$ be a coloured graph on vertices $[k]$, and define the operations as explained in Table~\ref{tab1}. Take note that operations can be applied in sequence, read from left to right. For example, if $F$ consists of two coloured vertices labelled $1$ and $2$, connected by an edge, then $F(+ 3, c_3\gets c^F_1, c_1 \gets 1-c_1^F, 1\sim 3)$ is the graph obtained from $F$ by adding a third vertex, setting its colour to that of the first vertex, flipping the colour of the first vertex, and connecting it to the third vertex.

\begin{table}[t]
\footnotesize\centering
\begin{tabular}{ll}
\toprule
\bf Notation & \bf Action \\
\midrule
$F(c_a\gets 0)$ for $a\in[k]$ & set the colour of vertex $a$ to 0 \\[0.5ex]
$F(c_a\gets 1)$ for $a\in[k]$ & set the colour of vertex $a$ to 1 \\[0.5ex]
$F(+a)$ for $a\notin[k]$ & add a new vertex labelled $a$ \\[0.5ex]
$F(a\sim b)$ for $a,b\in[k]$ & add an edge between vertices $a$ and $b$ \\ 
&(has no effect when an edge is already present)\\[0.5ex]
$F(a\not\sim b)$ for $a,b\in[k]$ &remove an edge between vertices $a$ and $b$ \\
&(has no effect when an edge is already absent)\\[0.5ex]
$F(a = b)$ for $a,b\in[k]$ & merge vertices $a$ and $b$, set label of merged vertex to $a$,\\
& set its colour to that of vertex $a$, relabel \\
& the remaining vertices from $a+1$ to $k-1$\\[0.5ex]
$ F\smash{\displaystyle\merge{a=b}} F'$ for $a\in [k]$, $b\in [k']$ & merge vertices $a$ and $b$, set label of merged vertex to $a$,\\
& set its colour to that of vertex $a$, relabel \\
& the remaining vertices of $F'$ from $k$ to $k+k'-1$\\
\bottomrule
\end{tabular}
\caption{Graph operations acting on a coloured graph $F$ with vertex set $[k]$ and a coloured graph $F'$ with vertex set $[k']$.}
\label{tab1}
\end{table}

Based on these graph operations, we introduce a number of multisets. First, for a coloured graph $F$ on the vertex set $[k]$, define the sets $\cS_+(F)$ and $\cS_-(F)$ as in Table~\ref{tab2}. In words, the set $\cS_+(F)$ is obtained by going through each vertex, adding a new vertex of the same colour and attaching an edge between the two, and afterwards flipping the colour of the original vertex, while the set $\cS_-(F)$ is obtained by going through each vertex and attaching a new vertex of the opposite colour to it. Moreover, let the signed multiset $\cS(F)$ be the difference of the two sets.

Next, define the multisets $\cS^\circ_+(F)$ and $\cS^\circ_-(F)$ as in Table~\ref{tab2}. The multisets $\cS^\circ_+(F)$ and $\cS^\circ_-(F)$ are obtained by going through each vertex and then going through each other vertex of the same colour (in the case of $\cS^\circ_+(F)$) or of the opposite colour (in the case of $\cS^\circ_-(F))$ and adding an edge between the two. For $\cS^\circ_+(F)$, the first vertex also changes colour. If edges are already present, then the operation has no effect, but it adds to the multiplicity of the respective graph. Moreover, let the signed multiset $\cS^\circ(F)$ be the difference of the two sets.

\begin{table}[!t]
\footnotesize\centering
\begin{tabular}{l}
\toprule
\bf  Multisets related to drift coefficients \\
\midrule
  $\displaystyle\cS_+(F)  = \sum_{p\in V(F)}\bclc{F\bclr{+ (k+1),c_{k+1}\gets c_p^F, p\sim k+1, c_p\gets1-c_p^F}}$ \\[3ex]
  $\displaystyle\cS_-(F)  = \sum_{p\in V(F)}\bclc{F\bclr{+ (k+1),c_{k+1}\gets1-c_p^F,p\sim k+1}}$\\[4ex]
  $\cS(F) = \cS_+(F) - \cS_-(F)$\\
  \midrule
  $\displaystyle \cS^\circ_+(F) = \sum_{\substack{p\in V(F)}} \sum_{\substack{q\in V(F)\colon\\ q\neq p, c_q^F=c_p^F}} \bclc{F(c_p\gets1-c_p^F,p\sim q)}$  \\[6ex]
  $\displaystyle \cS^\circ_-(F)  = \sum_{\substack{p\in V(F)}} \sum_{\substack{q\in V(F)\colon\\ q\neq p, c_q^F\neq c_p^F}}\bclc{F(p\sim q)}.$\\[6.5ex]
   $\cS^\circ(F) = \cS^\circ_+(F) - \cS^\circ_-(F)$\\
  \midrule
$\displaystyle \cS^\diamond_+(F)  = \sum_{a\in V(F)}\sum_{\substack{b\in V(F)\colon\\b\neq a, a\stackrel{F}{\not\sim}b, c^{F}_a=c^{F}_b}} \bclc{F\bclr{a=b,+ k,c_{k}\gets c_a^F,c_a\gets 1-c_a^F,a\sim k}}$ \\[6.5ex]
$\displaystyle  \cS^\diamond_-(F)  = \sum_{a\in V(F)}\sum_{\substack{b\in V(F)\colon\\b\neq a, a\stackrel{F}{\not\sim}b, c^{F}_a\neq c^{F}_b}} \bclc{F\bclr{a=b,+ k,c_k\gets c_a^F,c_a\gets 1-c_a^F,a\sim k}}$\\[7ex]
$\cS^\diamond(F) = \cS^\diamond_+(F) - \cS^\diamond_-(F)$\\
\bottomrule
\end{tabular}
\caption{\label{tab2}The various (signed) multisets required in the description of the drift component of subgraph densities.}
\end{table}

Finally, let the multisets $\cT_{\scriptscriptstyle=,+}(F,F')$, $\cT_{\scriptscriptstyle\neq,+}(F,F')$, $\cT_{\scriptscriptstyle\neq,-}(F,F')$, and $\cT_{\scriptscriptstyle=,-}(F,F')$ be as in Table~\ref{tab3}. These sets are obtained by merging one vertex of $F$ with one vertex of $F'$, and then connecting the merged vertex to a newly added vertex, with four possible combinations, depending on whether the colour of the merged vertices match or not and depending on the colour of the newly added vertex. Note that, while $\cT_{\scriptscriptstyle\neq,\cdot}(F,F') \neq \cT_{\scriptscriptstyle\neq,\cdot}(F',F)$ because the graphs are labelled and the order of $F$ and $F'$ matters in the construction of $F\merge{a=b}F'$ when the colours of $a$ and $b$ are different, we nonetheless have $\cT_{\scriptscriptstyle\neq,+}(F,F') + \cT_{\scriptscriptstyle\neq,-}(F,F') = \cT_{\scriptscriptstyle\neq,+}(F',F) + \cT_{\scriptscriptstyle\neq,-}(F',F)$. 

\begin{lemma}[Action of generator on injective subgraph densities]
\label{lem5}
For some $d\geq 1$, let $F_1,\dots,F_d$ be coloured graphs. Then, for $G$ a coloured graph on $n$ vertices,
\besn{
\label{24}
(\cA_n(f\circ \bt^{\inj}))(G)
& = \sum_{m=1}^d\bclr{\eta\, \~\mu_{F_m}^{\mathrm{v}}(G)+\rho\, \~\mu_{F_m}^{\mathrm{e}}(G)}f_m\bclr{\bt^{\inj}(G)} \\
&\qquad + \frac12\sum_{m,m'=1}^d \eta\, \~\sigma_{F_m,F_{m'}}^{\mathrm{v}}(G)f_{m,m'}\bclr{\bt^{\inj}(G)} + \bigo(n^{-1}),       
}
where $\~\mu_{F}^{\mathrm{v}}(G)$, $\~\mu_{F}^{\mathrm{e}}(G)$, and $\~\sigma_{F,F'}^{\mathrm{v}}(G)$ are defined in Table~\ref{tab4}. The constant in $\bigo(n^{-1})$ depends only on $F_1,\dots,F_d$, $f$ and on the model parameters.
\end{lemma}

\begin{table}[t]
\footnotesize\centering
\begin{tabular}{l}
\toprule
\bf  Multisets related to diffusion coefficients \\
\midrule
$\displaystyle \cT_{\scriptscriptstyle=,+}(F,F') = \sum_{a\in V(F)}\sum_{\substack{b\in V(F')\colon\\ c^{F}_a=c^{F'}_b}}  \bbclc{\bclr{F\merge{a=b}F'}\bclr{+ (k+k'),c_{k+k'}\gets c_a^F,c_a\gets 1-c_a^F,a\sim k+k'}}$  \\
$\displaystyle \cT_{\scriptscriptstyle=,-}(F,F')= \sum_{a\in V(F)}\sum_{\substack{b\in V(F')\colon\\ c^{F}_a= c^{F'}_b}}  \bbclc{\bclr{F\merge{a=b}F'}(+ (k+k'),c_{k+k'}\gets1-c_a^F,a\sim k+k')}.$\\
$\displaystyle \cT_{\scriptscriptstyle\neq,+}(F,F') = \sum_{a\in V(F)}\sum_{\substack{b\in V(F')\colon\\ c^{F}_a\neq c^{F'}_b}}  \bbclc{\bclr{F\merge{a=b}F'}(+ (k+k'),c_{k+k'}\gets c_a^F,c_a\gets 1-c_a^F,a\sim k+k')}$  \\
$\displaystyle \cT_{\scriptscriptstyle\neq,-}(F,F')= \sum_{a\in V(F)}\sum_{\substack{b\in V(F')\colon\\ c^{F}_a\neq c^{F'}_b}}  \bbclc{\bclr{F\merge{a=b}F'}(+ (k+k'),c_{k+k'}\gets1-c_a^F,a\sim k+k')}$ \\[7ex]
$\cT(F,F') = \cT_{\scriptscriptstyle=,+}(F,F') + \cT_{\scriptscriptstyle=,-}(F,F') - \cT_{\scriptscriptstyle\neq,+}(F,F') - \cT_{\scriptscriptstyle\neq,-}(F,F')$\\
\bottomrule
\end{tabular}
\caption{\label{tab3}The various (signed) multisets required in the description of the diffusion component of subgraph densities.}
\end{table}

\begin{table}[t]
\footnotesize\centering
\begin{tabular}{l}
\toprule
\bf  Drift and diffusion coefficients   \\
\midrule
$\displaystyle\~\mu_{F}^{\mathrm{v}}(G) = (n-k) \sum_{H\in\cS(F)} t^{\inj}_H(G) 
+ \sum_{H\in\cS^\circ(F)} t^{\inj}_H(G)$ \\[4ex]
$\displaystyle\~\mu_{F}^{\mathrm{e}}(G) = \sum_{\{r,s\}\in E(F)} \bbcls{\bclr{s_{\mathrm{c},0}c^F_{rs} 
+ s_{\mathrm{d},0}\^c^F_{rs}}  \bclr{t^{\inj}_{F(r\not\sim s)}(G)-t^{\inj}_F(G)}-\bclr{s_{\mathrm{c},1}c^F_{rs} + s_{\mathrm{d},1}\^c^F_{rs}} t^{\inj}_F(G)}$\\[4ex]
$\displaystyle\~\sigma_{F,F'}^{\mathrm{v}}(G) = \sum_{H\in\cT(F,F')} t^{\inj}_H(G)$\\
\midrule
$\displaystyle\mu_{F}^{\mathrm{v}}(G)= (n-k)\sum_{H\in\cS(F)} t_H(G)  
+ \sum_{H\in\cS^\circ(F)} t_H(G) + \sum_{H\in\cS^\diamond(F)} t_H(G)$,\\[4ex]
$\displaystyle \mu_{F}^{\mathrm{e}}(G) = \sum_{\{r,s\}\in E(F)} \bbcls{\bclr{s_{\mathrm{c},0}c^F_{rs} 
+ s_{\mathrm{d},0}\^c^F_{rs}}  \bclr{t_{F(r\not\sim s)}(G)-t_F(G)}-\bclr{s_{\mathrm{c},1}c^F_{rs} + s_{\mathrm{d},1}\^c^F_{rs}} t_F(G)}$\\[4ex]
$\displaystyle\sigma_{F,F'}^{\mathrm{v}}(G) = \sum_{H\in\cT(F,F')} t_H(G)$\\
\bottomrule
\end{tabular}
\caption{\label{tab4} The drift and diffusion coefficients for coloured subgraph densities. The top half is for injective subgraph densities, the bottom half for regular subgraph densities.}
\end{table}

\begin{proof}
Indices $i$ and $j$ refer to vertices of the subgraphs $F_1,\dots,F_d$, and indices $u$, $v$, $u_1,\dots,u_k$ refer to vertices of the graph $G$. Sums of the form $\tsump_{u_1,\dots,u_k}$ represent a summation over all $k$-tuples with pairwise distinct entries (of which there are $(n)_k=n(n-1)\times\cdots\times(n-k+1)$ many). Recall the definition of $\cA_n$, and rewrite $r^{\mathrm{v}}_u$ and $r^{\mathrm{e}}_{u,v}$ as
\ban{
  r^{\mathrm{v}}_u(G) 
  & = \sum_{v:v\neq u}\bclr{\^c_u e_{uv}c_v+c_u e_{uv} \^c_v},\label{25}\\
  r^{\mathrm{e}}_{u,v}(G) 
  & = \^e_{uv}\clr{s_{\mathrm{c},0}c_{uv} + s_{\mathrm{d},0}\^c_{uv}} + e_{uv}\clr{s_{\mathrm{c},1}c_{uv}+s_{\mathrm{d},1}\^c_{uv}\label{26}
}}
(recall that we drop the dependence on $G$ in the useage of $c_v$ and $e_{uv}$ for convenience). In order to understand the effect of vertex and edge changes on coloured subgraph densities, we introduce the following quantities. Let $F$ be a coloured subgraph, and let $u_1,\dots,u_k$ be vertices of $G$. It is not difficult to see that the effect of a colour change at vertex $u$ on a coloured subgraph indicator at $u_1,\dots,u_k$, measuring the presence of the coloured subgraph $F$ at that location, can be expressed by 
\small
\be{
\chi^{F,\mathrm{v}}_{u_1,\dots,u_k}(u) 
= \begin{cases}
& \text{if $e^F_{ij}=1$ implies $e_{u_iu_j}=1$ for all $1\leq i<j\leq k$,} \\
+ 1 & \text{if there is $i$ such that $u_i=u$ with $c_{u_i}=1-c^F_i$, }\\
& \text{and if $c^F_i=c_{u_i}$ for all $1\leq i\leq k$ for which $u_i\neq u$;}\\[2ex]
& \text{if $e^F_{ij}=1\hence  e_{u_iu_j}=1$ for all $1\leq i<j\leq k$,} \\
- 1 & \text{if there is $i$ such that $u_i=u$ with $c_{u_i}=c^F_i$, }\\
& \text{and if $c^F_i=c_{u_i}$ for all $1\leq i\leq k$ for which $u_i\neq u$;}\\[2ex]
0 & \text{else,}
\end{cases}
}
\normalsize
Thus, we define and write
\ben{\label{27}
\D_u t^{\inj}_F(G) \coloneqq t^{\inj}_F(G^u)-t^{\inj}_F(G) 
= \frac{1}{(n)_k}\sump_{u_1,\dots,u_k} \chi^{F,\mathrm{v}}_{u_1,\dots,u_k}(u). 
}
Similarly, the effect of an edge change between $u$ and $v$ on a coloured subgraph indicator at $u_1,\dots,u_k$ can be expressed by
{\small\bes{
\mel\chi^{F,\mathrm{e}}_{u_1,\dots,u_k}(u,v)\\ 
&= \begin{cases}
& \text{if $e^F_{ij}=1\hence  e_{u_iu_j}=1$ for all $1\leq i<j\leq k$ for which $\{u_i,u_j\}\neq \{u,v\}$,} \\
+ 1 & \text{if $\exists 1\leq i<j\leq k$ such that $\{u_i,u_j\}=\{u,v\}$, $e^F_{ij}=1$ and $e_{uv}=0$,} \\ 
& \text{and if $c_{u_i}=c^F_i$ for all $1\leq i\leq k$;}\\[2ex]
& \text{if $e^F_{ij}=1$ implies $e_{u_iu_j}=1$ for all $1\leq i<j\leq k$,} \\
- 1& \text{if there are $1\leq i<j\leq k$ such that $\{u_i,u_j\}=\{u,v\}$ and $e^F_{ij}=1$,} \\ 
& \text{and if $c_{u_i}=c^F_i$ for all $1\leq i\leq k$;}\\[2ex]
0 & \text{else.}
\end{cases}
}}
\normalsize
Thus, we define and write
\ben{\label{28}
\D_{uv} t^{\inj}_F(G) \coloneqq t^{\inj}_F(G^{uv})-t^{\inj}_F(G) = \frac{1}{(n)_k}
\sump_{u_1,\dots,u_k} \chi^{F,\mathrm{e}}_{u_1,\dots,u_k}(u,v). 
}
Straightforward arguments yield
\be{
\babs{\D_{u} t^{\inj}_F(G)}\leq \frac{C_F}{n},\qquad \babs{\D_{uv} t^{\inj}_F(G)}\leq \frac{C_F}{n^2},
}
where $C_F$ only depends on the coloured graph $F$. 

Now, let $f\colon [0,1]^d\to\R$ be a three times continuously partially differentiable function. Abbreviate $\bt^{\inj}(G) = \bclr{t^{\inj}_{F_m}(G)}_{m=1}^d$ and $(\D_{u}\bt^{\inj})(G) = \bclr{\D_{u} t^{\inj}_{F_m}(G)}_{m=1}^d$. Taylor expansion yields
\bes{
\mel f\bclr{\bt^{\inj}(G^u)} - f\bclr{\bt^{\inj}(G)}
= f\bclr{\bt^{\inj}(G)+(\D_u\bt)^{\inj}(G)} - f\bclr{\bt^{\inj}(G)}\\
&\quad= \sum_{m=1}^d \D_ut^{\inj}_{F_m}(G) f_m\bclr{\bt^{\inj}(G)}
+ \tfrac12 \sum_{m,m'=1}^d \D_ut^{\inj}_{F_m}\D_ut^{\inj}_{F_{m'}}(G) f_{m,m'}\bclr{\bt^{\inj}(G)} \\
&\qquad+ R_u,
}
where 
\be{
\abs{R_u} \leq d^3 \sup_{i,j,k} \norm{\D_u t^{\inj}_{F_{i}}}\, \norm{\D_u t^{\inj}_{F_{j}}}\,
\norm{\D_u t^{\inj}_{F_{k}}}\,\norm{f_{ijk}} \leq \frac{C_{f,\bF}}{n^3}
}
with $\norm{\cdot}$ denoting the supremum norm and $\bF=(F_1,\dots,F_d)$. Put $(\D_{u,v}\bt^{\inj})(G) = \bclr{\D_{u,v} t^{\inj}_{F_m}(G)}_{m=1}^d$ and Taylor expand
\be{
\begin{aligned}
&f\bclr{\bt^{\inj}(G^{u,v})} - f\bclr{\bt^{\inj}(G)}
= f\bclr{\bt^{\inj}(G)+(\D_{u,v}\bt^{\inj})(G)} - f\bclr{\bt^{\inj}(G)}\\
&\quad = \sum_{m=1}^d \D_{u,v}t^{\inj}_{F_m}(G) f_m\bclr{\bt^{\inj}(G)} + R_{uv},
\end{aligned}
}
where 
\be{
\abs{R_{uv}} \leq d^2 \sup_{i,j} \norm{\D_{uv} t^{\inj}_{F_{i}}}\,\norm{\D_{uv} t^{\inj}_{F_{j}}}\,\norm{f_{ij}} 
\leq \frac{C_{f,\bF}}{n^4}.
}
We therefore obtain
\ban{
\mel (\cA_n(f\circ \bt^{\inj}))(G) \notag\\
&= \eta\sum_{1\leq u\leq n}r^{\mathrm{v}}_u(G) \bclr{f(\bt^{\inj}(G^u))-f(\bt^{\inj}(G))}\notag\\ 
&\qquad + \rho\sum_{1\leq u <  v\leq n} r^{\mathrm{e}}_{u,v}(G)\bclr{f(\bt^{\inj}(G^{uv}))-f(\bt^{\inj}(G))}\notag\\
&= \eta\sum_{1\leq u\leq n}r^{\mathrm{v}}_u(G) \bbbcls{\,\sum_{m=1}^d \D_{u}t^{\inj}_{F_m}(G)
f_m\bclr{\bt^{\inj}(G)}\notag\\
&\qquad\qquad\qquad\qquad 
+ \tfrac12\sum_{m,m'=1}^d \D_{u}t^{\inj}_{F_m}\D_{u}t^{\inj}_{F_{m'}}(G)f_{m,m'}\bclr{\bt^{\inj}(G)} + R_u}\notag\\
&\qquad + \rho\sum_{1\leq u <  v\leq n} r^{\mathrm{e}}_{u,v}(G)
\bbbcls{\sum_{m=1}^d \D_{u,v}t^{\inj}_{F_m}(G)f_m\bclr{\bt^{\inj}(G)} + R_{uv}}\notag\\
\bs{
&= \sum_{m=1}^d \bbbcls{\eta\sum_{1\leq u\leq n}r^{\mathrm{v}}_u(G) \D_{u}t^{\inj}_{F_m}(G)
+ \rho\sum_{1\leq u <  v\leq n} r^{\mathrm{e}}_{u,v}(G)\D_{u,v}t^{\inj}_{F_m}(G)}f_m\bclr{\bt^{\inj}(G)}\\
&\qquad + \tfrac12\sum_{m,m'=1}^d\bbbcls{ \eta\sum_{1\leq u\leq n}r^{\mathrm{v}}_u(G) 
\D_{u}t^{\inj}_{F_m}\D_{u}t^{\inj}_{F_{m'}}(G)}
f_{m,m'}\bclr{\bt^{\inj}(G)}+R, 
}\label{29}
}
where
\be{
R = \eta \sum_{1\leq u\leq n}r^{\mathrm{v}}_u(G) R_u + \rho \sum_{1\leq u <  v\leq n} r^{\mathrm{e}}_{u,v}(G) R_{uv}.
}
Since $r^{\mathrm{v}}_u(G)$ is bounded by $n$ and $r^{\mathrm{e}}_{u,v}(G)$ is bounded, we have
\bes{
\abs{R} = \frac{C_{f,\bF,\eta,\rho}}{n}.
}
Tedious but straightforward computations now yield the claim (see Appendix~\ref{sec17} for further details).
\end{proof}

The action of the generator on subgraph densities (not necessarily injective) is slightly more involved, but can be derived from Lemma~\ref{lem5}. To this end,  define the multisets $\cS^\diamond_{\scriptscriptstyle=,+}(F)$, $\cS^\diamond_{\scriptscriptstyle\neq,+}(F)$, $\cS^\diamond_{\scriptscriptstyle=,-}(F)$, and $\cS^\diamond_{\scriptscriptstyle\neq,-}(F)$ as in Table~\ref{tab2}. We are now ready to state and prove the main result of this section.

\begin{lemma}[Action of generator on (regular) subgraph densities] 
\label{lem6} 
For some $d\geq 1$, let $F_1,\dots,F_d$ be coloured graphs. Then, for $G$ a coloured graph on $n$ vertices,
\besn{
\label{30}
(\cA_n(g\circ \bt))(G)
& = \sum_{m=1}^d\bclr{\eta \mu_{F_m}^{\mathrm{v}}(G)+\rho \mu_{F_m}^{\mathrm{e}}(G)}g_m\bclr{\bt(G)} \\
&\qquad + \frac12\sum_{m,m'=1}^d \eta \sigma_{F_m,F_{m'}}^{\mathrm{v}}(G)g_{m,m'}\bclr{\bt(G)} + \bigo(n^{-1}),       
}
where $\mu_{F}^{\mathrm{v}}(G)$, $\mu_{F}^{\mathrm{e}}(G)$, and $\sigma_{F,F'}^{\mathrm{v}}(G)$ are defined in Table~\ref{tab4}. The constant in $\bigo(n^{-1})$ depends only on $F_1,\dots,F_d$, $g$ and on the model parameters.
\end{lemma}

\begin{proof}
Note that
\bes{
    t_F(G) & = \frac{1}{n^k}\sum_{\ell_1,\dots,\ell_k} T^F_{\ell_1,\dots,\ell_k}(G) \\
    & = \bbbclr{1-{k\choose 2}\frac{1}{n}}t^{\inj}_F(G)  \\
    & \qquad + \frac{1}{n^k}\sum_{1\leq a<b\leq k\,\,}\sump_{\ell_1,\dots, \ell_{b-1},\ell_{b+1},\dots,\ell_k} 
    T^F_{\ell_1,\dots, \ell_{b-1},\ell_a,\ell_{b+1},\dots,\ell_k}(G)
    + \bigo(n^{-2}),
}
where
\be{
  T^F_{\ell_1,\dots,\ell_k}(G) = \prod_{a\stackrel{F}{\sim}b} e_{\ell_a \ell_b}(G) \prod_{a\in V(F)} \I[c^G_{\ell_a}=c^F_a].
}
Now,
\be{
    T^F_{\ell_1,\dots,\ell_{b-1},\ell_a,\ell_{b+1},\dots,\ell_k}(G)
}
can be non-zero only if $c^F_a=c^F_b$ and only if $a$ and $b$ are not connected in $F$, in which case
\be{
    \sump_{\ell_1,\dots, \ell_{b-1},\ell_{b+1},\dots,\ell_k} 
    T^F_{\ell_1,\dots, \ell_{b-1},\ell_a,\ell_{b+1},\dots,\ell_k}(G)
    = \sump_{\ell_1,\dots,\ell_{k-1}}T^{F(a=b)}_{\ell_1,\dots,\ell_{k-1}}(G).
}
Hence
\be{
    t_F(G) = \bbbclr{1-{k\choose 2}\frac{1}{n}}t^{\inj}_F(G)  
    + \frac{1}{n}\sum_{\substack{1\leq a<b\leq k\\c_a^F=c_b^F,a\stackrel{F}{\not\sim}b}} t^{\inj}_{F(a=b)}(G) 
    + \bigo(n^{-2}).
}
Note that the $\bigo(n^{-2})$-term itself is also composed of injective coloured subgraph densities of $G$, but all with prefactors of order $n^{-2}$ and smaller. Define
\ba{
\cS^\diamond_{\scriptscriptstyle=}(F) 
& = \sum_{a\in V(F)}\sum_{\substack{b\in V(F)\colon\\b\neq a,a\stackrel{F}{\not\sim}b, c^{F}_a=c^{F}_b}}
\clc{F(a=b)}.
}
Next, apply Lemma~\ref{lem5} with 
\be{
    f(\bt^{\inj}(G)) = g\bbbclr{\bbbcls{\bbbclr{1-{k\choose 2}\frac{1}{n}}t^{\inj}_{F_i}(G) 
    + \frac{1}{2n}\sum_{H\in\cS^\diamond_+(F_i)} t^{\inj}_H(G) + \bigo(n^{-2})}_{1\leq k \leq d}},
}
where the vector $\bt^{\inj}(G)$ comprises the injective coloured subgraph densities of all~z$F_i$, all $F_i(a=b)$, and all other coloured subgraph densities appearing in the $\bigo(n^{-2})$-term. Up to terms of order $\bigo(n^{-1})$, the terms in \eqref{30} are easy to derive from \eqref{24}, except for 
\besn{\label{31}
  \~\mu_{F}^{\mathrm{v}}(G) 
  & = (n-k)\sum_{H\in\cS(F)}\bbbclr{1-{k\choose 2}\frac{1}{n}} t^{\inj}_H(G) 
  + \sum_{H\in\cS^\circ(F)}\bbbclr{1-{k\choose 2}\frac{1}{n}} t^{\inj}_H(G)  \\
  &\qquad + \frac{n-k}{2n}\sum_{H\in \cS^\diamond_+(F)}\sum_{H'\in\cS(H)}t^{\inj}_{H'}(G)+ \bigo(n^{-1}) \\
  &= (n-k)\sum_{H\in\cS(F)} t_H(G) 
  + \sum_{H\in\cS^\circ(F)} t_H(G)\\
  &\qquad + \frac{n-k}{2n}\sum_{H\in \cS^\diamond_{\scriptscriptstyle=}(F)}
  \sum_{H'\in\cS(H)}t_{H'}(G)
  -\frac{n-k}{2n}\sum_{H\in\cS(F)}\sum_{H'\in \cS^\diamond_{\scriptscriptstyle=}(H)}
  t_{H'}(G)\\
  &\qquad + \bigo(n^{-1}),
}
the third and fourth expressions of which require some arguments. We first compare
\ben{\label{32}
    \sum_{H\in \cS^\diamond_{\scriptscriptstyle=}(F)} \sum_{H'\in\cS_+(H)} t_{H'}(G)
    \qquad\text{and}\qquad \sum_{H\in\cS_+(F)}\sum_{H'\in \cS^\diamond_{\scriptscriptstyle=}(H)} t_{H'}(G).
}
Let $u,v,w\in V(F)$ be distinct vertices, and assume that $c_u = c_v$ and $e_{uv}=0$. Clearly, the following operations are commutative: (1) merge $u$ and $v$; $(2)$ attach a new vertex to $w$ with the same colour as $w$ and then flip the colour of $w$. Hence, we only need to consider the case where the two operations affect the same vertices. Thus, the difference of the two terms in \eqref{32} is
\be{
    \sum_{H\in \cS^\diamond_+(F)} t_{H'}(G)
    -  \sum_{H\in \cS^\diamond_-(F)} t_{H'}(G).
}
In order to compare
\ben{\label{33}
    \sum_{H\in \cS^\diamond_{\scriptscriptstyle=}(F)} \sum_{H'\in\cS_-(H)} t_{H'}(G)
    \qquad\text{and}\qquad \sum_{H\in\cS_-(F)}\sum_{H'\in \cS^\diamond_{\scriptscriptstyle=}(H)} t_{H'}(G),
}
let $u,v\in V(F)$ be distinct vertices, and assume that $c_u = c_v$ and $e_{uv}=0$. Clearly, the following operations are commutative for \emph{any} $w\in V(F)$: (1) merge $u$ and $v$; $(2)$ attach a new vertex to $w$ with the opposite colour as $w$. Hence, the two sums in \eqref{33} are identical, and consequently cancel out in the difference~\eqref{31}. 
\end{proof} 


\subsection{Evolution of coloured subgraph densities under the projection}
\label{sec13}

In this section we analyse the evolution of the subgraph densities for the projected process with the help of the generator 
$\cA_n$. The evolution is captured in Lemma~\ref{lem7}. Before we state this lemma we make a preliminary observation about the ``drift'' and the ``diffusion coefficients'' from Lemma~\ref{lem6}.

\begin{lemma}[Property of drift and diffusion coefficients]
\label{cor1}
Let $\-F$ be an uncoloured graph, and let $H'$ and $G$ be coloured graphs. Let $\mu^{\mathrm{v}}, \sigma^{\mathrm{v}}$ be as in Lemma~\ref{lem6}.  Then
\be{
\sum_{H\in\cC(\-F)} \mu^{\mathrm{v}}_H(G) = 0,
\qquad
\sum_{H\in\cC(\-F)} \sigma^{\mathrm{v}}_{H,H'}(G) = \sum_{H\in\cC(\-F)} \sigma^{\mathrm{v}}_{H',H}(G) = 0.
}
\end{lemma}

\begin{proof} 
It is easy to see that it is enough to show the following statements:
\bg{
  (i)\enskip  
  \sum_{H\in\cC(\-F)} \cS_+(H) = \sum_{H\in\cC(\-F)} \cS_-(H),
  \qquad
  (ii)\sum_{H\in\cC(F)}\cS^\circ_+(H) 
  = \sum_{H\in\cC(F)} \cS^\circ_-(H),
  \\
  (iii)\enskip 
  \sum_{H\in\cC(\-F)}\cS^\diamond_+(F) = \sum_{H\in\cC(\-F)}\cS^\diamond_-(F),
  \\ 
  (iv)\enskip
  \sum_{H\in\cC(\-F)}\cT_{\scriptscriptstyle=,+}(H,H')
  = \sum_{H\in\cC(\-F)} \bclr{\cT_{\scriptscriptstyle\neq,-}(H,H'),
  \\
  (v)\enskip \sum_{H\in\cC(\-F)}\cT_{\scriptscriptstyle=,-}(H,H')}
  = \sum_{H\in\cC(\-F)}\cT_{\scriptscriptstyle\neq,+}(H,H').
}
We only prove $(i)$, the other equations being analogous. To this end, let $H_0\in \cC(\-F)$ and $H_0'\in\cS_+(H_0)$. We show that there is exactly one $H'_1$ on the right-hand side of $(i)$ such that $H_0'=H_1'$. Indeed, let $p\in V(\-F)$ be the vertex to which the new vertex is attached in the construction of $H_0'$ from $H_0$. Let $H_1 = H_0(c_p \gets 1-c_p^{H_0})$, and take $H_1' = H_1(+ (k+1),c_{k+1}\gets1-c_p^{H_1},p\sim k+1)$, which is in $\cS_-(H_1)$. It is clear that $H_1' = H_0'$, and that this is the only way to obtain $H'_0$.
\end{proof}


\paragraph{The projection.} 
Let $(\kappa,c)$ be a coloured graphon, and let $\bar{c} = \int_{[0,1]} \dd x c(x)$ be the average colour. For a coloured graph $F$, denote by $\bar{F}$ the corresponding uncoloured graph, and denote by $\cC(\-F)$ the set of all $2^k$ colourings of $\-F$. Throughout this section, $(\kappa,c)$ is fixed and is suppressed from the notation; that is, we write $t_{\-F}$ instead of $t_{\-F}(\kappa,c)$, write $\mu^{\mathrm{v}}_H$ instead of $\mu^{\mathrm{v}}_H(\kappa,c)$, and so forth. Note that
\be{
t_{\-F} = \sum_{H\in\cC(\-F)} t_{H}.
}
Define the projection $\cP\colon \cW \to \cW_0$ by leaving the graphon unchanged and by replacing the colour function $c$ by the constant colour function $\-c$. Note that, for any coloured graph $F$,
\be{
(t_{F}\circ\cP)(\kappa,c) = \-c^{w(F)}(1-\-c)^{b(F)}\,t_{\bar{F}_m}(\kappa) = {\bar{c}}^{w(F)}(1-\bar{c})^{b(F)} \sum_{H\in\cC(\bar{F})} t_{H}(\kappa,c).
}
If $G\in\cG_n$, then $\cP(G)$ is defined through the canonical embedding of $G$ into $\cW$. We are now ready to state and prove the main result of this section.

\begin{table}[t!]
\footnotesize\centering
\begin{tabular}{l}
\toprule
\bf  Definition of Remainder Terms   \\ 
\midrule\\[-1ex]
$\displaystyle\Delta_{H} = t^{\inj}_H - \bar{c}^{w(H)}(1-\bar{c})^{b(H)} t^{\inj}_{\bar H}$\\[2ex]
$\displaystyle\Delta_{\mathrm{c},0,h}= \sum_{i=1}^d y_{\wv}^{w_i}y_{\bv}^{b_i}\, h_i(y)\sum_{\{r,s\}\in E(\bar{F}_i)} \sum_{H\in \cC(\bar{F}_i)} c_{rs}^H\bclr{\Delta_{H(r\not\sim s)}-\Delta_{H}}$\\[3ex]
$\displaystyle\Delta_{\mathrm{d},0,h}= \sum_{i=1}^d y_{\wv}^{w_i}y_{\bv}^{b_i}\, h_i(y)\sum_{\{r,s\}\in E(\bar{F}_i)} \sum_{H\in \cC(\bar{F}_i)} \^c_{rs}^H\bclr{\Delta_{H(r\not\sim s)}-\Delta_{H}}$ \\[3ex]
$\displaystyle\Delta_{\mathrm{c},1,h}=-\sum_{i=1}^d y_{\wv}^{w_i}y_{\bv}^{b_i}\, h_i(y)\sum_{\{r,s\}\in E(\bar{F}_i)} \sum_{H\in \cC(\bar{F}_i)} c_{rs}^H \Delta_{H}$\\[3ex]
$\displaystyle\Delta_{\mathrm{d},1,h}=-\sum_{i=1}^d y_{\wv}^{w_i}y_{\bv}^{b_i}\, h_i(y)\sum_{\{r,s\}\in E(\bar{F}_i)} \sum_{H\in \cC(\bar{F}_i)} \^c_{rs}^H \Delta_{H}$\\[3ex]
$\displaystyle\Delta_{\wv,\wv} = 2\Delta_{\whiteblackedge}y_{\wv}^{-2}\sum_{i=1}^d (w_i)_2 y_{F_i} h_i(y)+ 2\Delta_{\whiteblackedge} {y}_{\wv}^{-2}\sum_{i=1}^d\sum_{j=1}^d w_i w_j y_{F_i} y_{F_j} h_{i,j}(y)$\\[3ex]
$\displaystyle\Delta_{\wv,\bv} = 2\Delta_{\whiteblackedge}y_{\wv}^{-1}y_{\bv}^{-1}\sum_{i=1}^d w_i b_i y_{F_i} h_i(y)+ 2\Delta_{\whiteblackedge} y_{\wv}^{-1}y_{\bv}^{-1}\sum_{i=1}^d\sum_{j=1}^d w_ib_j y_{F_i} y_{F_j} h_{i,j}(y)$\\[3ex]
$\displaystyle\Delta_{\bv,\bv} = 2\Delta_{\whiteblackedge}y_{\bv}^{-2}\sum_{i=1}^d (b_i)_2 y_{F_i} h_i(y)+ 2\Delta_{\whiteblackedge} y_{\bv}^{-2}\sum_{i=1}^d\sum_{j=1}^d b_ib_j y_{F_i} y_{F_j} h_{i,j}(y)$\\[3ex]
\bottomrule
\end{tabular}
\caption{\label{tab5} The $\Delta$-quantities used in Lemma~\ref{lem7}. We have suppressed the dependence on the underlying graph $G$ with respect to which the coloured subgraph densities are evaluated.}
\end{table}

\begin{lemma}[Action of generator on subgraph densities under projection] 
\label{lem7}
Let $h$ be a three times continuously partially differentiable function, let $F_1,\dots,F_d$ be coloured graphs, let $\bt = \bclr{t_{F_1},\dots,t_{F_d}}$, let $G\in\cG_n$, and let $y_F = t_F\circ\cP(G)$ for any $F$. Then 
\besn{\label{34}
  \mel (\cA_n(h\circ \bt \circ \cP))(G) \\
  &= \rho\sum_{i=1}^d \bbbclc{\sum_{\{r,s\} \in E(F_i)}\bbcls{s_{\mathrm{c},0}\bclr{y_{\wv}^2+y_{\bv}^2}
  + 2s_{\mathrm{d},0 }y_{\wv}y_{\bv}}\bclr{y_{F_i(r\not\sim s)} - y_{F_i}}\\
  & \kern10em - e(F_i)\bbcls{s_{\mathrm{c},1}\bclr{y_{\wv}^2+y_{\bv}^2}+2s_{\mathrm{d},1}y_{\wv}y_{\bv}} y_{F_i}}h_i(y)\\
  &\quad + \eta\, y_{\blackwhiteedge}\sum_{i=1}^d\bbclr{\bclr{w_i y_{\wv}^{-1}- b_i y_{\bv}^{-1}}^2 
  - \bclr{w_i y_{\wv}^{-1}+ b_i y_{\bv}^{-1}}} y_{F_i}h_i(y) \\
  &\quad + \eta\, y_{\blackwhiteedge}\sum_{i=1}^d\sum_{j=1}^d \bclr{w_i y_{\wv}^{-1}-b_i y_{\bv}^{-1}}
  \bclr{w_j y_{\wv}^{-1}-b_j y_{\bv}^{-1}} y_{F_i} y_{F_j} h_{i,j}(y)\\
  &\quad + \rho\,\bclr{s_{\mathrm{c},0}{\Delta_{\mathrm{c},0,h}} + s_{\mathrm{d},0}{\Delta_{\mathrm{d},0,h}} 
  + s_{\mathrm{c},1}{\Delta_{\mathrm{c},1,h}} + s_{\mathrm{d},1}{\Delta_{\mathrm{d},1,h}}}\\
  & \quad + \eta\,\bclr{\Delta_{\wv,\wv} + \Delta_{\wv,\bv} + \Delta_{\bv,\bv}} + \bigo(n^{-1}),
}  
where the $\Delta$-terms are defined in Table~\ref{tab5}, and where $w_i=w(F_i)$ and $b_i=b(F_i)$ for $1\leq i\leq d$.
\end{lemma}

\begin{proof} 
To simplify notation, for any coloured graph $F$, write
\be{
x_F = x_F(G) =  t_{F}, \qquad x_{\bar{F}} = \sum_{H\in\cC(\bar{F})} x_H.
}  
In particular, $x_{\whitevertex} = y_{\whitevertex} = t_{\whitevertex}(G)$ and $x_{\blackvertex} = y_{\blackvertex} = t_{\bv}(G)$. Note that, for any coloured graph $F$,
\be{
(t_{F}\circ\cP)(G) = \-c^{w(F)}(1-\-c)^{b(F)}\,t_{\bar{F}_m}(\kappa) = {\bar{c}}^{w(F)}(1-\bar{c})^{b(F)} \sum_{H\in\cC(\bar{F})} t_{H}.
}
Let $F_0 = \whitevertex$. We can write
\be{
y_F = x_{\wv}^{w(F)} x_{\bv}^{b(F)} x_{\-F}.
}
We consider the action of the operator $\cA_n$ on functions of the form
\be{
f(x) = h(y), \qquad y=(y_{F_1},\dots,y_{F_d}),
}
with 
\be{
x = \clr{x_{iH}}_{0\leq i\leq d,H\in\cC(\-F_i)}.
}
Note that $x_H$ is short-hand for $t_H(G)$, whereas $x_{iH}$ is a variable. Also note that even when $F_i = \wv$ for some $1\leq i\leq d$, we consider $x_{0\wv}$ and $x_{i\wv}$ as different variables for the sake of calculating derivatives. The same is true for $\bv$.
Furthermore, note that, while $h$ takes $d$ arguments, the function $f$ takes $2+2^{k_1}+\cdots+2^{k_d}$ arguments. We have
\be{
\bclr{\cA_n \bclr{h \circ\bt\circ \cP}}\clr{G} = \bclr{\cA_n \bclr{h\circ y}}\clr{G} = \bclr{\cA_n \bclr{f\circ x}}(G),
}
and \eqref{29} yields
\bes{
(\cA_n(h\circ \bt \circ \cP))(G) 
& = \sum_{\substack{0\leq j\leq d,\\ H\in\cC(\bar{F}_j)}}
\bclr{\eta\,\mu_{H}^{\mathrm{v}}+\rho\,\mu_{H}^{\mathrm{e}}}f_{jH}(x)\\
&\qquad + \frac12\sum_{\substack{0\leq j\leq d,\\ H\in\cC(\bar{F}_j)}}\sum_{\substack{0\leq j'\leq d,\\ 
H'\in\cC(\bar{F}_{j'})}} \eta\,\sigma_{H,H'}^{\mathrm{v}}f_{jH,j'H'}(x) + \bigo(n^{-1}).       
}
Here, $f_{jH}$ and $f_{jH,j'H'}$ denote partial derivatives in directions $x_{jH}$ and $(x_{jH},x_{j'H'})$, respectively. 
We proceed by calculating the drift coefficients induced by the vertex changes and the edge changes and the diffusion coefficient induced by the edge changes.


\medskip\noindent\textit{$\bullet$ Drift coefficient induced by vertex changes.} 
The partial derivatives of $f$ are given by
\ba{
f_{0\,\wv}(x) &= \sum_{i=1}^d w_i x_{0\,\wv}^{w_i-1} x_{0\,\bv}^{b_i} x_{\-F_i} h_i(y),\\
f_{0\,\bv}(x) &= \sum_{i=1}^d b_i  x_{0\,\wv}^{w_i} x_{0\,\bv}^{b_i-1} x_{\-F_i} h_i(y),
}
and
\be{
f_{iH}(x) = x_{0\,\wv}^{w_i} x_{0\,\bv}^{b_i}h_i(y),\qquad \text{$1\leq i\leq d$, $H\in\cC(\bar{F}_i)$.}
}
It is immediate that $\mu_{\wv}^{\mathrm{v}}=0$ and $\mu_{\bv}^{\mathrm{v}}=0$. For the remaining sum, we have
\be{
\sum_{\substack{1\leq i\leq d,\\ H\in\cC(\bar{F}_i)}} \mu_{H}^{\mathrm{v}}f_{iH}(x)
= \sum_{i=1}^d x_{\wv}^{w_i} x_{\bv}^{b_i} h_i(y) \sum_{H\in\cC(\bar{F}_i)} \mu_{H}^{\mathrm{v}}  = 0
}
by Lemma~\ref{cor1}.


\medskip\noindent\textit{$\bullet$ Drift coefficient induced by edge changes.}
Similarly as before, $\mu_{\wv}^{\mathrm{e}}=0$ and $\mu_{\bv}^{\mathrm{e}}=0$. For the remaining sum we have
\bes{
\mel\sum_{\substack{1\leq i\leq d,\\ H\in\cC(\bar{F}_i)}} \mu_{H}^{\mathrm{e}}f_{iH}(x)\\
& = \sum_{i=1}^d x_{\wv}^{w_i}x_{\bv}^{b_i}\, h_i(y) \sum_{\{r,s\}\in E(\-F_i)}\sum_{H\in \cC(\bar{F}_i)}\bbcls{\bclr{s_{\mathrm{c},0}c^H_{rs} + s_{\mathrm{d},0}\^c^H_{rs}} 
\bclr{x_{H(r\not\sim s)}-x_{H} } \\
& \kern22em - \bclr{s_{\mathrm{c},1}c^H_{rs} + s_{\mathrm{d},1}\^c^H_{rs}}\,x_{H}},
}
where we use that if $H\in \cC(\bar{F}_i)$, then $E(H) = E(\bar{H}_i)$. Now, for instance,
\bes{
\mel \sum_{i=1}^d x_{\wv}^{w_i}x_{\bv}^{b_i}\, h_i(y)\sum_{\{r,s\}\in E(\bar{F}_i)}
\sum_{H\in \cC(\bar{F}_i)}c^H_{rs} \bclr{x_{H(r\not\sim s)}-x_{H} }\\
&= \sum_{i=1}^d x_{\wv}^{w_i}x_{\bv}^{b_i}\, h_i(y)\sum_{\{r,s\}\in E(\bar{F}_i)}
\sum_{H\in \cC(\bar{F}_i)}c^H_{rs}\bclr{y_{H(r\not\sim s)} -y_{H}}\\
&\qquad + \sum_{i=1}^d x_{\wv}^{w_i}x_{\bv}^{b_i}\, h_i(y)\sum_{\{r,s\}\in E(\bar{F}_i)}
\sum_{H\in \cC(\bar{F}_i)}c^H_{rs}\bclr{\Delta_{H(r\not\sim s)}-\Delta_{H}},
}
where, for any coloured graph $H$, we define
\be{
\Delta_{H} = x_H - y_H = t^{\inj}_H - \-{c}^{w(H)}(1-\bar{c})^{b(H)} t^{\inj}_{\bar H};
}
see Table~\ref{tab5}. By the binomial theorem, we have
\bes{
\mel\sum_{H\in \cC(\bar{F}_i)}c^H_{rs} \bclr{y_{H(r\not\sim s)} -y_{H}} \\
&  = \sum_{H\in \cC(\bar{F}_i)}c^H_{rs}\, x_{\wv}^{w(H)}x_{\bv}^{b(H)}\bclr{x_{\-F_i(r\not\sim s)} -x_{\-F_i}} \\ 
 &   = \bclr{x_{\wv}^2+x_{\bv}^2}(x_{\wv}+x_{\bv})^{(k_i-2)\vee0}\bclr{x_{\-F_i(r\not\sim s)} -x_{\-F_i}}.
}
Thus, using the fact that $x_{\wv}+x_{\bv} = 1$, we have
\bes{
\mel \sum_{i=1}^d x_{\wv}^{w_i}x_{\bv}^{b_i}\, h_i(y)
\sum_{\{r,s\}\in E(\-F_i)}\sum_{H\in \cC(\bar{F}_i)}c^H_{rs} 
\bclr{x_{\-H(r\not\sim s)}-x_{\-H}}\\
&= \sum_{i=1}^d\sum_{\{r,s\}\in E(\bar{F}_i)}(x_{\wv}^2+x_{\bv}^2) x_{\wv}^{w_i}x_{\bv}^{b_i}
\bclr{x_{\bar{F}_i(r\not\sim s)} - x_{\bar{F}_i}}\, h_i(y) + \Delta_{\mathrm{c},0,h}\\
&= \sum_{i=1}^d\sum_{\{r,s\}\in E(\bar{F}_i)}(x_{\wv}^2+x_{\bv}^2)\bclr{y_{F_i(r\not\sim s)} - y_{F_i}}\, h_i(y)
+ \Delta_{\mathrm{c},0,h},\\
}
where $\Delta_{\mathrm{c},0,h}$ is as defined in Table~\ref{tab5}. Doing similar calculations for the other terms as well, we arrive at
\bes{
\mel\sum_{\substack{0\leq i\leq d,\\ H\in\cC(\bar{F}_i)}} 
\bclr{\eta\,\mu_{H}^{\mathrm{v}}+ \rho\,\mu_{H}^{\mathrm{e}}}f_{iH}(x)\\
&= \rho\sum_{i=1}^d \bbbclc{\sum_{\{r,s\} \in E(F_i)}\bbcls{s_{\mathrm{c},0}\bclr{x_{\wv}^2+x_{\bv}^2}
+2s_{\mathrm{d},0}x_{\wv}x_{\bv}}\bclr{y_{F_i(r\not\sim s)} - y_{F_i}}\\
&\kern15em - k_i\bbcls{s_{\mathrm{c},1}\bclr{x_{\wv}^2+x_{\bv}^2}+2s_{\mathrm{d},1}x_{\wv}x_{\bv}} y_{F_i}}h_i(y)\\
&\qquad + \rho\,\bclr{s_{\mathrm{c},0}{\Delta_{\mathrm{c},0,h}} + s_{\mathrm{d},0}{\Delta_{\mathrm{d},0,h}} 
+ s_{\mathrm{c},1}{\Delta_{\mathrm{c},1,h}} + s_{\mathrm{d},1}{\Delta_{\mathrm{d},1,h}}}.
}


\medskip\noindent\textit{$\bullet$ Diffusion coefficient induced by vertex changes.} 
The second partial derivatives of $f$ are as follows:
\ba{
f_{0\,\wv,0\,\wv}(x) 
&= \sum_{i=1}^d (w_i)_2 x_{0\,\wv}^{w_i-2} x_{0\,\bv}^{b_i} x_{\-F_i} h_i(y)
+ \sum_{i=1}^d \sum_{j=1}^d w_iw_j x_{0\,\wv}^{w_i+w_j-2} x_{0\,\bv}^{b_i+b_j}x_{\-F_i} x_{\-F_j} h_{i,j}(y), \\
f_{0\,\wv,0\,\bv}(x) 
&= \sum_{i=1}^d w_ib_i x_{0\,\wv}^{w_i-1} x_{0\,\bv}^{b_i-1} x_{\-F_i} h_i(y)
+ \sum_{i=1}^d \sum_{j=1}^d w_ib_i x_{0\,\wv}^{w_i-1} x_{0\,\bv}^{b_i-1}x_{\-F_i}x_{\-F_j} h_{i,j}(y),\\
f_{0\,\bv,0\,\bv}(x) &= \sum_{i=1}^d (b_i)_2 x_{0\,\wv}^{w_i} x_{0\,\bv}^{b_i-2} x_{\-F_i} h_i(y) 
+ \sum_{i=1}^d \sum_{j=1}^d b_ib_j x_{0\,\wv}^{w_i+w_j} x_{0\,\bv}^{b_i+b_j-2} x_{\-F_i}x_{\-F_j} h_{i,j}(y).
}
For $1\leq i\leq d$ and $H\in\cC(\bar{F}_i)$, 
\ba{
f_{0\,\wv, iH}(x) &= w_i x_{0\,\wv}^{w_i-1} x_{0\,\bv}^{b_i} h_i(y)
+ \sum_{j=1}^d w_j x_{0\,\wv}^{w_i+w_j-1} x_{0\,\bv}^{b_i+b_j} x_{\-F_j} h_{j,i}(y),\\
f_{0\,\bv, iH}(x) &= b_i x_{0\,\wv}^{w_i}x_{0\,\bv}^{b_i-1} h_i(y)
+ \sum_{j=1}^d b_j x_{0\,\wv}^{w_i+w_j} x_{0\,\bv}^{b_i+b_j-1} x_{\-F_j} h_{j,i}(y).
}
Finally, for $1\leq i\leq d$ and $H\in\cC(\bar{F}_i)$, and $1\leq j\leq d$ and $H'\in\cC(\bar{F}_j)$, we have
\bes{
f_{iH, jH'}(x) &= x_{0\,\wv}^{w_i+w_j} x_{0\,\bv}^{b_i+b_j}h_{i,j}(y).
}
Now,
\bes{
\mel\sum_{\substack{0\leq j\leq d,\\ H\in\cC(\bar{F}_j)}}\sum_{\substack{0\leq j'\leq d,\\ 
H'\in\cC(\bar{F}_{j'})}} \sigma_{H,H'}^{\mathrm{v}}f_{jH,j'H'}(x)\\
& = \sum_{\substack{1\leq j\leq d,\\ H\in\cC(\bar{F}_j)}}\sum_{\substack{1\leq j'\leq d,\\ 
H'\in\cC(\bar{F}_{j'})}} \sigma_{H,H'}^{\mathrm{v}}f_{jH,j'H'}(x)\\
&\quad + 2\sum_{\substack{1\leq j\leq d,\\ H\in\cC(\bar{F}_j)}} \sigma_{\wv,H}^{\mathrm{v}}f_{0\,\wv,jH}(x)
 + 2\sum_{\substack{1\leq j\leq d,\\ H\in\cC(\bar{F}_j)}} \sigma_{\bv,H}^{\mathrm{v}}f_{0\,\bv,jH}(x)\\
&\quad + \sigma_{\wv,\wv}^{\mathrm{v}}f_{0\,\wv,0\,\wv}(x) 
+ \sigma_{\bv,\bv}^{\mathrm{v}}f_{0\,\wv,0\,\wv}(x)
+ 2\sigma_{\wv,\bv}^{\mathrm{v}}f_{0\,\wv,0\,\bv}(x).
}
By Lemma~\ref{cor1}, 
\bes{
\mel\sum_{\substack{1\leq j\leq d,\\ H\in\cC(\bar{F}_j)}}\sum_{\substack{1\leq j'\leq d,\\ 
H'\in\cC(\bar{F}_{j'})}} \sigma_{H,H'}^{\mathrm{v}}f_{jH,j'H'}(x) \\
& = \sum_{1\leq j\leq d}\sum_{1\leq j'\leq d} x_{0\,\wv}^{w_i+w_j} x_{0\,\bv}^{b_i+b_j}h_{i,j}(y)
\sum_{H\in\cC(\bar{F}_j)}\sum_{H'\in\cC(\bar{F}_{j'})}\sigma_{H,H'}^{\mathrm{v}} = 0.
}
Similarly, we have
\be{
\sum_{\substack{1\leq j\leq d,\\ H\in\cC(\bar{F}_j)}} \sigma_{\wv,H}^{\mathrm{v}}f_{0\,\wv,jH}(x) = 0,
\qquad
\sum_{\substack{1\leq j\leq d,\\ H\in\cC(\bar{F}_j)}} \sigma_{\bv,H }^{\mathrm{v}}f_{0\,\bv,jH}(x) = 0.
}
For the remaining terms, it is easy to deduce that
\be{
\sigma_{\wv,\wv}^{\mathrm{v}} = \sigma_{\bv,\bv}^{\mathrm{v}}  = 2 x_{\whiteblackedge}= -\sigma_{\wv,\bv}^{\mathrm{v}}.
}
Thus,
\ba{
\mel\sigma_{\wv,\wv}^{\mathrm{v}}f_{0\,\wv,0\,\wv}(x)\\
&= 2x_{\whiteblackedge} x_{\wv}^{-2}\sum_{i=1}^d (w_i)_2y_{F_i} h_i(y)
+ 2x_{\whiteblackedge} x_{\wv}^{-2}\sum_{i=1}^d\sum_{j=1}^d w_iw_j  y_{F_i} y_{F_j} h_{i,j}(y),\\
\mel\sigma_{\wv,\bv}^{\mathrm{v}} f_{0\,\wv,0\,\bv}(x)\\
&= -2x_{\whiteblackedge} x_{\wv}^{-1}x_{\bv}^{-1}\sum_{i=1}^d w_ib_i y_{F_i} h_i(y)
- 2x_{\whiteblackedge}x_{\wv}^{-1}x_{\bv}^{-1}\sum_{i=1}^d\sum_{j=1}^d w_ib_j  y_{F_i} y_{F_j} h_{i,j}(y),\\
\mel\sigma_{\bv,\bv}^{\mathrm{v}}f_{0\,\bv,0\,\bv}(x)\\
&= 2x_{\whiteblackedge}x_{\bv}^{-2}\sum_{i=1}^d (b_i)_2 y_{F_i} h_i(y)
+ 2x_{\whiteblackedge} x_{\bv}^{-2}\sum_{i=1}^d\sum_{j=1}^d b_ib_j y_{F_i} y_{F_j} h_{i,j}(y).
}
Replacing $x_{\whiteblackedge}$ by $y_{\whiteblackedge}$, and adding the difference, we get, for example,
\bes{
\mel\sigma_{\wv,\wv}^{\mathrm{v}}f_{0\,\wv,0\,\wv}(x)\\
&= 2y_{\whiteblackedge} x_{\wv}^{-2}\sum_{i=1}^d (w_i)_2y_{F_i} h_i(y)
+ 2y_{\whiteblackedge} x_{\wv}^{-2}\sum_{i=1}^d\sum_{j=1}^d w_iw_j  y_{F_i} y_{F_j} h_{i,j}(y) + \Delta_{\wv,\wv},
}
where again $\Delta_{\wv,\wv}$ is defined in Table~\ref{tab5}. The other terms can be rewritten in a similar way via analogous quantities $\Delta_{\wv,\bv}$ and~$\Delta_{\bv,\bv}$. Collecting and rearranging all terms, and recalling that $x_{\wv} = y_{\wv}$ and $x_{\bv}=y_{\bv}$, we arrive at~\eqref{34}.
\end{proof}


\section{Proof of the main results}
\label{sec14}

We are now ready to prove our main results. In Section~\ref{sec15} we prove pathwise existence and uniqueness of the limiting process, and use this to prove Theorem~\ref{thm2}. We also derive the dynamics governing the subgraph density of the limiting processes. In Section~\ref{sec16} we show that the distance (under the metric that induces the path topology of convergence in measure) between the finite-$n$ process and its projection converges to zero in mean. In Section~\ref{62} we prove Theorem~\ref{thm2}.


\subsection{Proof of Theorem~\ref{thm1} and evolution of subgraph densities} 
\label{sec15}

We first prove Lemma~\ref{lem8}, that is, there is a pathwise unique solution to \eqref{2}. Afterwards we use this to prove Theorem~\ref{thm1}.

\begin{lemma}[Existence] 
\label{lem8}
There exists a Markov process $(\kappa_t,q_t)_{t\geq0}$ that is the pathwise unique solution to \eqref{2} with initial conditions $(\kappa_0,q_0)$, where $\kappa_0$ is a graphon and $q_0\in[0,1]$. Moreover, $\kappa_t$ is a graphon and $q_t \in[0,1]$ for all $t$.
\end{lemma}

\begin{proof}
 Let $\kappa_0$ be a graphon and $q_0\in(0,1)$. Let $W = (W_t)_{t \ge 0}$ be a standard Brownian motion on $\R$, and the function $V:[0,1]\times [0,1] \rightarrow [0,1]$ be given in \eqref{3}. Consider the coupled equations on $[0,1]\times[0,1]$ given by
\ben{ \label{35}
  \left\{ 
  \begin{aligned}
  \dd{p_t} &= \rho\, V\bclr{p_t,q_t}\dd t,\\
  \dd{q_t} &= \sqrt{2\eta\, p_t\,q_t(1-q_t)} \dd{W_t},
  \end{aligned}
  \right.
}
with initial values 
\be{
  p_0 = \int_{[0,1]^2}\dd x \dd y \kappa_0(x,y),
  \qquad
  q_0 = \int_{[0,1]}\dd x c_0(x).
}
Existence and uniqueness of  a  unique strong solution $(p_t,q_t)_{t\geq0}$ to \eqref{35}  will follow from \cite{S65} and  \cite[Theorem~1]{YW71}  (Conditions~$(i)$ and $(ii)$ of that theorem are easily verified). Moreover, it follows from \cite[Proposition~2]{YW71} that this bivariate process has the strong Markov property. From the integral representation 
\be{
  q_t = q_0 + \int_0^t \sqrt{2\eta\, p_s\,q_s(1-q_s)}\dd{W_s},
}
it is also clear that $q_t\in[0,1]$ for all $t\geq0$. Now, for each realisation of $(W_t)_{t \geq 0}$ we can consider $q_t$ as  given by the above. Then as for each $x$ and $y$, since $V(0,q) \leq 0$ and $V(1,q) \geq 0$, and that $V(\cdot,q)$ is linear in between there is a unique solution to 
\ben{\label{36}
  \kappa_t(x,y) = \kappa_0(x,y) + \int_0^t \rho\, V\bclr{\kappa_s(x,y),q_{s}}\dd s.
}
Thus we may obtain that $\kappa_t(x,y)\in[0,1]$ for all $t\geq0$, $x,y\in[0,1]$ explicitly as function of $p_t$ and $q_t$ (though it does not give insight). Clearly $\kappa_t(x,y)$ is symmetric in $x$ and $y$ as $\kappa_0$ is a graphon. Measurability of $\kappa_t$ follows by approximating $\kappa_0$ by step functions, and then using stability of initial conditions, Gr\"onwall's inequality, linearity of $V$ and uniqueness of the solution of \eqref{36}.
\end{proof}

\begin{proof}[Proof of Theorem~\ref{thm1}] Assertion $(a)$ follows from Lemma~\ref{lem8}. Suppose that $(\kappa_0,q_0)$ and $(\kappa_0',q_0')$ are as in Assertion $(b)$. Then there exist measure-preserving maps $\lambda$ and $\lambda'$ on $[0,1]$ such that 
\be{
  \kappa_0\bclr{\lambda(x),\lambda(y)} = \kappa_0'\bclr{\lambda'(x),\lambda'(y)}, \qquad x,y\in[0,1].
} 
Since $p_t$ remains unchanged under such maps and $q_0=q_0'$, it is clear that we can take $q'_t=q_t$ and $p'_t=p_t$. From the representation in \eqref{36} and the linearity of $V$, it is therefore clear that  
\ben{\label{37}
  \kappa_t(\lambda(x),\lambda(y)) = \kappa_t'(\lambda'(x),\lambda'(y)), \qquad x,y \in[0,1],
}
for all $t \geq 0$, which proves Assertion $(b)$. The first part of Assertion $(c)$ follows in part from the fact that $(\kappa_t,c_t)$ are unique strong solutions to \eqref{2}. Let $\{{\cal F}_t\}_{t \geq 0}$ the filtration generated by $\{ (\kappa_{s},c_{s})\colon 0 \leq s \leq t \}$. For any $t \geq 0$ and any measure-preserving map $\lambda$ on $[0,1]$, we have  $(\~\kappa^\lambda_t,\~c^\lambda_t) = (\~\kappa_t, \~c_t)$ (recall that $\~\kappa$ denotes the equivalence class associated with $\kappa$), and from \eqref{37} we have that $(\kappa^\lambda_{t},c^\lambda_{t})$ has the same law as $(\kappa_{t},c_{t})$.  Thus, by the fact that $(\kappa_t,c_t)_{t \geq 0}$ is a Markov process, for any $0<s<t$ the conditional law of $(\~\kappa_t, \~c_t)$ given ${\cal F}_s$ depends on $((\kappa_t,c_t), (\kappa_s,c_s))$ only.  From \eqref{37}, we also have that $((\kappa^\lambda_{t},c^\lambda_{t}), (\kappa^\lambda_{s},c^\lambda_{s}))$ has the same joint law as $ ((\kappa_{t},c_{t}), (\kappa_{s},c_{s}))$ for any measure-preserving map $\lambda$ on $[0,1]$. So, for any $0<s<t$ the conditional law of $(\~\kappa_t, \~c_t)$ given $(\~\kappa_s, \~c_s)$ is the same as the conditional law of $(\~\kappa_t, \~c_t)$ given $(\~\kappa^\lambda_s, \~c^\lambda_s)$. Consequently, the conditional law of $(\~\kappa_t, \~c_t)$ given ${\cal F}_s$ is the same as the the conditional law of $(\~\kappa_t, \~c_t)$ given $(\~\kappa_s, \~c_s)$, which implies that  $(\~\kappa_t, \~c_t)_{t \geq 0}$ is a Markov process. 
\end{proof}

We conclude this section by identifying the stochastic differential equations for the subgraph densities of the limiting process. This will be needed in the proof of convergence of the projection to the limiting process (Lemma~\ref{lem16}), and is an essential ingredient in the proof of Theorem~\ref{thm2}.

\begin{lemma}[Flow of coloured subgraph densities] 
\label{lem9}
Let $F$ be a coloured subgraph on $k$ vertices, of which $w$ are white and $b$ are black, and let $(\kappa_t,q_t)$ be the solution to \eqref{2} with initial conditions $(\kappa_0,q_0)$. Then the coloured subgraph density $t_{F}(\kappa_t,q_t)$ satisfies the stochastic differential equation
\ban{ \label{38}
&\dd{t_{F}(\kappa_t,q_t)}  \nonumber \\
&\quad=\rho\, \bbbcls{\sum_{\{r,s\}\in E(F)}\bcls{s_{\mathrm{c},0}(q_t^2+(1-q_t)^2) 
+ 2s_{\mathrm{d},0}q_t(1-q_t)}(t_{F(r\not\sim s)}(\kappa_t,q_t)-t_F(\kappa_t,q_t))  \nonumber \\
&\kern8em - e(F)\bcls{s_{\mathrm{c},1}(q_t^2+(1-q_t)^2)+2s_{\mathrm{d},1}q_t(1-q_t)}t_F(\kappa_t,q_t)}\dd t  \nonumber  \\
&\qquad + \eta\,\bclr{(w)(w-1)q_t^{-1}(1-q_t) - 2wb + (b)(b-1)q_t(1-q_t)^{-1}}\,t_{F}(\kappa_t,q_t)\,p_t \dd t  \nonumber \\
&\qquad + \bclr{wq_t^{-1} - b (1-q_t)^{-1}}t_{F}(\kappa_t,q_t) \sqrt{2\eta\, p_t q_t (1-q_t)}\dd{W_t},
}
where $(W_t)_{t \geq 0}$ is the Brownian motion from \eqref{2}. The statement is simultaneously true for any finite collection of coloured subgraphs with respect to the same Brownian motion.
\end{lemma}

\begin{proof}
Let $F$ be a coloured subgraph on $k$ vertices, of which $w$ are white and $b$ are black, and let $(\kappa_t,q_t)$ be the solution to \eqref{2} with initial conditions $(\kappa_0,q_0)$. Then
\be{ 
t_F(\kappa_t,q_t) = q_t^w(1-q_t)^b t_{\-F}(\kappa) = \int_{[0,1]^k} q_t^w(1-q_t)^b \prod_{i\stackrel{F}{\sim}j}
\kappa_t(x_i,x_j)\prod_{i=1}^k \dd{x_i},
}
Fix $\{x_1,\ldots, x_n\} \in [0,1]$. Then, by using the Ito formula (see, for example, \cite[Section~II.7]{P05}), we obtain, for $t \geq 0$,
\besn{\label{39}
&q_t^w(1-q_t)^b \prod_{i\stackrel{F}{\sim}j} \kappa_t(x_i,x_j)  \\  
&\quad= q_0^w(1-q_0)^b \prod_{i\stackrel{F}{\sim}j} \kappa_0 (x_i,x_j)  \\ 
&\qquad +  \int_0^t \bclr{wq_s^{w-1}(1-q_s)^b - b q_s^w(1-q_s)^{b-1}} 
\prod_{i\stackrel{F}{\sim}j} \kappa_t(x_i,x_j)  \sqrt{2\eta\, p_s q_s (1-q_s)}\dd{W_s}  \\
&\qquad + \int_0^t\left[\rho q_s^w(1-q_s)^b\sum_{\{u,v\}\in E(F)} \int_{[0,1]^k} 
\prod_{i\stackrel{\mathclap{F(u\not\sim v)}}{\sim}j}
\kappa(x_i,x_j) V\bclr{\kappa_s(x_u,x_v),q_s}\right]\dd s  \\
&\qquad + \int_0^t\bclr{w(w-1)q_s^{w-2}(1-q_s)^b - 2wbq_s^{w-1}(1-q_s)^{b-1} 
+ b(b-1)q_s^w(1-q_s)^{b-2}} \\[2ex] 
&\qquad \qquad \times \prod_{i\stackrel{F}{\sim}j} \kappa_t(x_i,x_j)  \eta\, p_s q_s (1-q_s)\dd s.
}
For each fixed $t \geq 0$, by integrating both sides of \eqref{39} with respect to $\prod_{i=1}^k \dd{x_i}$ over $[0,1]^k$, and using the linearity of $V$ and Fubini's theorem, we obtain \eqref{38}.
\end{proof}


\subsection{Distance between the finite coloured graph process and its projection}
\label{sec16}

Our main result in this section is the following lemma, which shows that the coloured graph process $G^n$, respectively, its graphon embedding $\kappa^n$ in $\cW$, is close to its projection $\cP(G^n)\in\cW_0$ in the path topology of convergence in measure. Recall the metric $\dMZ$ from \eqref{17}.

\begin{lemma}[Convergence to projection]
\label{lem10} 
Let $X^n = (X^n_t)_{t\geq 0} = (\kappa^n_t,c^n_t)_{t\geq 0}$ be the graphon embedding of the coloured graph process $G^n$. If\/ $\lim_{n\to\infty} \dsub(X^n_0,\kappa_0)$ $=0$ and\/ $\liminf_{n\to\infty}\noverlap(X^n_0)>0$, then 
\ben{
\label{40}
\lim_{n\to\infty} \EE\cls{\dMZ(\kappa^n,\cP(\kappa^n))} = 0,
} 
\end{lemma}

 To prove this, we first show in Lemma~\ref{lem11} that the random walk on the evolving random graph mixes on time scale $1/n$. We show in Lemma~\ref{lem15} that the coloured empirical graphon process collapses from the space $\cW$ to the subspace $\cW_0$ within any time~$\eps_n$ satisfying $1/n \ll \eps_n \ll 1/\sqrt{n}$ as $n\to\infty$, provided the graph connectivity is above some $\delta >0$. We will use pathwise duality of the voter model to continuous random walks here. We guarantee graph connectivity is above some $\delta >0$ in Lemma~\ref{lem12}.

Recall the definition of $\noverlap$ from \eqref{4}. It is easy to see that 
\be{
\noverlap(G) = \inf_{i,j \in [n]} \frac{\babs{N_i^G\cap N^G_j}}{n}, \qquad G\in\cG_n,
}
where $N_i^G$ denotes the set of neighbours of vertex $i$ in $G$. In what follows, $d_{\mathrm{TV}}$ stands for total variation distance and $\law$ stands for law. 

\begin{lemma}[Random walk mixing time]
\label{lem11}
Let $G\in\cG_n$. Let $Y^i = (Y^i_t)_{t \geq 0}$ be the continuous-time random walk on $G$ starting from vertex $i$ and jumping at rate $\eta$ along any edge that is adjacent to its current location (that is, the total jumping rate is $\eta$ times the degree of the current vertex). Then 
\be{
  \sup_{i \in [n]} \dtv\bclr{\law(Y^i_t), \law(U)} \leq
  \exp\bclr{-\eta \noverlap(G)^2n t},  } for all $t \geq 0$ and where $U$ is
the uniform distribution on $[n]$.
\end{lemma}

\begin{proof} 
Since $U$ is the stationary distribution of $Y^i$, we have
\be{
\dtv\bclr{\law(Y^i_t), \law(U)} \leq \sup_{j \in [n]} \dtv\bclr{\law(Y^i_t), \law(Y^j_t)}. 
}
We upper bound the right-hand side via a coupling argument as follows. Couple $Y^i$ and $Y^j$ such that the probability that they jump simultaneously to the same common neighbour equals
\be{
\frac{\abs{N_i\cap N_j}}{\abs{N_i}+\abs{N_j}-\abs{N_i\cap N_j}} 
\geq \frac{\noverlap n}{2n-\noverlap n} \geq \noverlap,
}
where $\noverlap = \noverlap(G)$. This can be achieved by using the same exponential clocks of rate $\eta$ for the common neighbours of $i$ and $j$, and independent exponential clocks for the non-common neighbours. The rate at which any of the random walks jumps is $\abs{N_i} \vee \abs{N_j} \geq \noverlap n$. If they jump simultaneously to any of the common neighbours, then they are coupled; if not, then we repeat the argument. Hence, the coupling time is upper bounded by a random sum of independent exponential random variables of rate $\eta \noverlap n$, where the number of summands is independent of the summands themselves and has a geometric distribution on $\N$ with success probability $\noverlap$. This random sum is exponentially distributed with rate $\eta\noverlap^2 n$, and so the claim follows.
\end{proof}

In order to use the mixing time estimate of the random walk effectively, we need an estimate of the graph connectivity at any time~$t \geq 0$. We provide this estimate in the next lemma.

\begin{lemma}[Guarantees on graph connectivity]
\label{lem12} 
Let $(G^n_t)_{t \geq 0}$ be the graph process and assume that $\noverlap(G^n_0)\geq \noverlap_0 >0$ for $n$ large enough and for some $\nu_0>0$. Let 
\be{
z_0 = s_{\mathrm{c},0} \wedge s_{\mathrm{d},0}, \qquad z_1 = s_{\mathrm{c},1} \vee s_{\mathrm{d},1}, \qquad p = \frac{z_0}{z_0+z_1}.
}
If\/ $p>\noverlap_0$, then, for all $n$ large enough,
\be{
\sup_{t \geq 0} \IP\bclr{\noverlap(G^n_t) \leq \tfrac12 \noverlap_0^3} \leq \ee^{-\tfrac13\noverlap_0^5 n}.
}
If $z_0=0$ and $z_1>0$, then, for any $T>0$ and for all $n$ large enough,
\be{
  \sup_{t \in [0,T]} \IP\bclr{\noverlap(G^n_t) \leq \tfrac12 \noverlap_0\ee^{-2\rho z_1 t}} 
  \leq \exp\bclr{-\tfrac13 \noverlap_0^2 \ee^{-4 \rho z_1 T}n}.
}
If $z_1=0$, then, for all $n$,
\be{
  \sup_{t \geq 0} \IP\bclr{\noverlap(G^n_t) < \nu_0 } = 0.
}
\end{lemma}

\begin{proof}
It is easy to see that the $n \choose 2$ edge processes in the graph process between the $n$ vertices can be stochastically lower bounded by $n \choose 2$ independent two-state Markov processes on $\{0,1\}$, each having transition rate $\rho z_0$ for the transition $0 \mapsto 1$ and transition rate $\rho z_1$ for the transition $1 \mapsto 0$, and each starting from the initial state of the corresponding edge in the graph process. The distribution of each Markov process at time~$t$ is $\Be(p (1-\ee^{-\rho z_0 t}))$ when started in $0$ and $\Be(\ee^{-\rho z_1 t} + p(1-\ee^{-\rho z_1 t}))$ when started in $1$. Using the fact that the number of common neighbours at time~$0$ is at least $\noverlap_0 n$ and that $\ee^{-\rho z_1 t} + p(1-\ee^{-\rho z_1 t}) \geq p$ for all $t \geq 0$, we can stochastically lower bound the number of common neighbours between any two vertices at time~$t$ by $\Bi(\noverlap_0 n,p^2)$ for all $t \geq 0$. Using a union bound, estimating $\nu_0 \leq p$, and applying Hoeffding's inequality, we obtain
\bes{
\IP\bclr{\noverlap(G^n_t)\leq \tfrac12 \noverlap_0^3} 
&\leq {n \choose 2} \IP\bigl(\Bi(\noverlap_0 n,p^2)\leq \tfrac12 \noverlap_0^3n\bigr) \\ 
&\leq {n\choose 2} \IP\bigl(\Bi(\noverlap_0 n,\noverlap_0^2)\leq \tfrac12 \noverlap_0^3n\bigr) \\ 
&\leq {n \choose 2} \exp\bclr{-2\noverlap_0 n\bclr{\noverlap_0^2-\tfrac12\noverlap_0^2}^2}
\leq \exp\bclr{-\tfrac12 \noverlap_0^5 n + 2 \log n}.
} 
The claim follows by noting that the right-hand side does not depend on $t$ and by choosing $n$ large enough. 

For the second case, for any time~$t \geq 0$, the density of 1's equals $\noverlap_0\ee^{-\rho z_1 t}$, and the above estimate yields
\bes{
\IP\bclr{\noverlap(G^n_t) \leq \tfrac12 \noverlap_0\ee^{-2\rho t}}  
&\leq {n \choose 2} \IP\bclr{\Bi(\noverlap_0 n,\ee^{-2 \rho z_1 t}) \leq \tfrac12 \noverlap_0\ee^{-2\rho z_1 t} n} \\ 
&\leq {n \choose 2} \exp\bclr{-2 (\noverlap_0 \ee^{-2 \rho z_1 t} - \tfrac12 \noverlap_0 \ee^{-2 \rho z_1t})^2 n} \\
&\leq \exp\bclr{-\tfrac12 \noverlap_0^2 \ee^{-4 \rho z_1 t}\, n + 2 \log n}.
}
The third case is trivial, since edges can only be added or no edge can change at all.
\end{proof}

We now prove some preliminary estimates using pathwise duality of the voter model to continuous random walks.

\begin{lemma} \label{lem13} 
Fix a coloured graph $F$, assume $\nu(G) > 0$, and suppose that $(\eps_n)_{n\in\N}$ is such that
\ben{\label{41}
  \lim_{n\to\infty} \sqrt{n}\,\eps_n = 0 \qquad\text{and}\qquad \lim_{n\to\infty} n\eps_n /\log n = \infty.
}
Then
 \ben{\label{42} 
   \lim_{n \to \infty} \EE\bcls{\abs{t_{F}(\kappa^n_0,c^n_{\eps_n}) - t_{F_m}(\kappa^n_0,\bar{c}^n_0)}} = 0.
}
\end{lemma}

\begin{proof}
Let $k$ denote the number of vertices in $F$.  Using \eqref{16} we can write
\ban{
 & t_{F}(\kappa^n_0,c^n_{\eps_n})   = \frac{1}{(n)_{k}} \sump_{a \in [n]^{k}} I_a J_a + \bigo(n^{-1}), \nonumber\\
  &   t_{F}(\kappa^n_0,\bar{c}^n_0) =\frac{1}{(n)_{k}} \sump_{a \in [n]^{k}}\IE\bcls{I''_a} J_a   + \bigo(n^{-1}),  \label{43}
}
where  $\sum'$ denotes summation over all tuples of indices that are mutually distinct, and where
\be{
    I_a = \prod_{j\in [k]}\I\bcls{c^n_{\eps_n}\left(\tfrac{a_j}{n}\right) = c^{F}_j},
  \quad  J_a  = \prod_{u\stackrel{F}{\sim} v} \kappa^n_0\left(\tfrac{a_u}{n},\tfrac{a_v}{n}\right)
\quad \mbox{and }
 I''_a  = \prod_{j\in[k]}\I\bbcls{c^n_0\bbclr{\tfrac{U^{a_j}}{n}} = c^{F}_j }
}
with $(U^{a_j})_{j \in [k]}$ being $k$ mutually independent uniform random variables on $[n]$ (we will couple these suitably in the proof below). Applying the triangle inequality we have,
\bes{
  \mel \EE\bcls{\abs{t_{F}(\kappa^n_0,c^n_{\eps_n}) - t_{F}(\kappa^n_0,\bar{c}^n_0)}}\\
  & \leq  \EE\bcls{\abs{t_{F}(\kappa^n_0,c^n_{\eps_n}) - \EE t_{F}(\kappa^n_0,c^n_{\eps_n})}} +   \EE\bcls{\abs{\EE t_{F}(\kappa^n_0,c^n_{\eps_n}) - t_{F}(\kappa^n_0,\bar{c}^n_0)}} \nonumber\\
  & \leq  \sqrt{\Var t_{F}(\kappa^n_0,c^n_{\eps_n})} +  \EE\bcls{\abs{\EE t_{F}(\kappa^n_0,c^n_{\eps_n}) - t_{F}(\kappa^n_0,\bar{c}^n_0)}},
}
where we have used  Cauchy-Schwarz inequality in the last step. Now using \eqref{43} in the above we have
\ban{
  \mel \EE\bcls{\abs{t_{F}(\kappa^n_0,c^n_{\eps_n}) - t_{F}(\kappa^n_0,\bar{c}^n_0)}}\nonumber\\
  & \leq \sup_{\substack{a,a'\in[n]^k,\\ \abs{a\cap a'}=\emptyset}}\Cov(I_a J_a, I_{a'} J_{a'})^{1/2} + \bigo(n^{-1/2}) \label{44} \\  
  & \qquad +  \frac{1}{(n)_{k}} \sump_{a \in [n]^{k}}\IE\bcls{\abs{I_a- I''_a} }   + \bigo(n^{-1}). \label{45}
}

In what follows we show that the terms in \eqref{44} and \eqref{45} converge to $0$ as $n \rightarrow \infty$ whenever $\nu(G) \geq \delta$. We do this in a sequence of 4 steps. In Step 1 we construct coalescing random walks that mimic the classical pathwise duality between the voter model and a system of coalescing random walks (see e.g.\ the monograph \cite{L85} or the survey of \cite[Section~4]{JK14}). Because the underlying graph changes over time, we cannot apply the classical duality directly. In Steps 2 and 3 we control the probability of jumps of the random walks along with the number of edge changes in the network. In Step 4 we use these estimates to finish the proof.

\paragraph{Step 1 (Construction of  Coalescing Random walks).}  
Let $\{G^n_t\}_{0 \leq t \leq \eps_n}$ be the co-evolving process. We will us the classical notion of the graphical representation of the process. For each ordered pair of distinct vertices $(i,j)$, let $(\cV^{ij}_t)_{t \geq 0}$ be a Poisson process of rate $\eta$, independent of all else. At the arrival times, if there is an edge between vertices~$i$ and~$j$, then vertex~$i$ imposes its colour onto vertex~$j$ (or equivalently, vertex $j$ adopts the colour of vertex $i$).

For $k\geq 1$ and $a = (a_1,\ldots,a_k) \in [n]^{k}$ with mutually distinct entries, let $(Y^{a_j}_{t})_{j \in [k], t \in [0,\eps_n]}$ be the locations of~$k$ coalescing random walks on the fixed graph $G^n_0=G$, starting from vertices $(a_j)_{j \in [k]}$. To determine the jump times and targets, we use the processes $(\cV^{ij}_t)_{t \geq 0}$ in the reverse direction and reverse time but only considering jumps when the corresponding edge is present in $G^n_0$; that is, $\cV^{ij}_{\eps_n-t}-\cV^{ij}_{\eps_n-t-}=1$ represents a jump from $j$ to $i$ at time $t$ if $i$ and $j$ are connected in $G^n_0$. As is easily verified, this results in the processes $(Y^{a_j}_{t})_{j \in [k], t \in [0,\eps_n]}$ jumping at rate $\eta$ along each of the edges of $G^n_0$, independently of the other $k-1$ random walks until they meet, from which point onwards they move together. 

If the graph does not evolve, we can determine the colours of vertices $(a_j)_{j \in [k]}$ in the evolcing graph process at time~$\eps_n$ via the colours of the vertices $(Y^{a_j}_{\eps_n})_{j \in [k]}$ of the evolving graph process at time~$0$. This will require that the jump times and jump targets determined by the processes $(\cV^{ij}_t)_{t \geq 0}$ correspond to edges for which the edge status has not changed between time $0$ and time $\eps_n$. In order to control the probability of this event, we will need another set of coalescing random walks.

Let $(Z^{a_j}_{t})_{j \in [k], t \in [0,\eps_n]}$ be the locations of~$k$ coalescing random walks, starting from vertices $(a_j)_{j \in [k]}$, again employing the processes $(\cV^{ij}_{\eps_n-t})_{t \geq 0}$ to determine the jump times and targets, but now considering jumps at time $t$ only when the corresponding edge is present in $G^n_{\eps_n-t}$. Note that the colours of the vertices $(Z^{a_j}_{\eps_n})_{j \in [k]}$ determine the colours of vertices $(a_j)_{j \in [k]}$ at time $\eps_n$. However, also note that the process $(Z^{a_j}_{t})_{j \in [k], t \in [0,\eps_n]}$ is not Markov, since whether or not an edge is present at time $t$ depends on a future event from the view point of $Z^{a_j}_{t}$.

To bound the probability that $(Y^{a_j}_{t})_{j \in [k], t \in [0,\eps_n]}$ is not affected by any edge changes and thus represents the true ancestry of the colours of the vertices $a_1,\dots, a_k$ in the evolving graph process at time $\eps_n$ all the way back to time $0$, we will need to understand
\be{
  L_a = \prod_{j=1}^k \I\bcls{\text{$Y^{a_j}_t = Z^{a_j}_t$  for all $0\leq t \leq \eps_n$}}.
}
Before this, choose the $k$ independent uniform random variables $(U^{a_j})_{j \in [k]}$ on $[n]$ to be maximally coupled to $(Y^{a_j}_{\eps_n})_{j \in [k]}$ and let $I''_a$ be defined as before. 


\paragraph{Step 2 (Comparing path duality and simple random walk).} 
We next prove that
\beqn{ \label{46}
    \IP(L_a= 0) \to 0  \qquad\text{as $n \rightarrow \infty$ }
}
uniformly in $a \in [n]^{k}$ and $\nu(G^n_0) \geq \delta$. As noted earlier, for any edge the probability of it having changed between time $0$ and time $\eps_n$ is at most $1-\ee^{-\rho \-s\eps_n}\leq \rho \-s\eps_n$, where $\-s = \max\{s_{\mathrm{c},0},s_{\mathrm{d},0},s_{\mathrm{c},1},s_{\mathrm{d},1}\}$. Therefore, for each $j\neq i$, the rate at which the edge status between $i$ and $j$ changes is upper bounded by $\rho \-s$. Hence, the number $M_i$ of vertices $j$ for which the edge status between $i$ and $j$ changes between time $0$ and time $\eps_n$ is stochastically upper bounded by $\Bi(n,1-\ee^{-\rho \-s \eps_n})$. By Chernoff's bound,
\besn{\label{47}
  \mel\IP\bclr{\text{$M_i > 2\rho \-sn\eps_n$ for any $i\in[n]$}}
  \leq n \IP\bclr{M_1 > 2\rho \-s n\eps_n} \\
  & \leq n \exp\bclr{-\tfrac13 n(1-\ee^{-\rho \-s \eps_n})} 
  \leq n \exp\bclr{-\tfrac16\rho \-s n \eps_n}
}
for $n$ large enough (say, for $n$ such that $\eps_n<1/\rho \-s$).

We next control the number of jumps performed by $Y^{a_i}_t$ by time~$\eps_n$; call this number $K_{a_i}$. Since the rate of jumping is at most $\eta n$, the number of jumps is stochastically upper bounded by $\Po(\eta n \eps_n)$. Hence, using again Chernoff-type bounds, we obtain
 \besn{\label{48}
   \IP\bclr{\text{$K_{a_i} > 2n\eta\eps_n$ for any $i\in[k]$}} 
   \leq k \IP\bclr{K_1 > 2n\eta\eps_n}
   \leq k \exp\bclr{-\tfrac13\eta n \eps_n}.
}
Define the event
\be{
  A_a = \bclc{\text{$M_i \leq 2\rho \-sn\eps_n$ for all $i\in[n ]$}} \cap \bclc{\text{$K_{a_i} \leq 2n\eta\eps_n$ for all $i\in[k]$}}.
}
From \eqref{47} and \eqref{48}, and invoking \eqref{41}, we have 
\ben{\label{49}
  \IP(A_a^c) \to 0 \qquad \text{as $n\to\infty$}
} 
uniformly in $a \in [n]^{k}$.

Note that if $\-s=0$, then $\IP(L_a= 0)=0$, so assume $\-s>0$. Also note that the number of vertices attached to, say, vertex $i$ in $G^n_0$ is at least~$\nu n$. Hence, on the event $A_a$ and on the graph $G^n_t$, the number of vertices that are attached to vertex $i$ and that do not change their edge status is at least~$\nu n - 2\rho \-s n\eps_n$. Moreover, the absolute difference of edges attached to vertex $i$ in $G^n_t$ as compared to that in $G^n_0$ is no more than $2\rho \-sn\eps_n$. Hence, the probability that $Z^{a_i}$ jumps along an edge that has appeared after time 0 or that $Y^{a_i}$ jumps along an edge that has disappeared after time~$0$ is at most 
\be{ 
  \frac{2\rho \-s n\eps_n}{\nu n - 2\rho \-sn\eps_n} = \frac{2\rho \-s\eps_n}{\nu - 2\rho \-s\eps_n}\leq \frac{4\rho \-s \eps_n}{\nu}\leq \frac{4\rho \-s \eps_n}{\delta}
}
for $n$ large enough (say, for $n$ such that $\eps_n < \delta/(4\rho\-s)$). Thus, 
\bes{
  \IP(L_a=0,A_a) 
  & \leq 1-\bclr{1- 4\rho \-s \eps_n\delta^{-1}}^{2\eta kn\eps_n} \leq
  8\rho \-s \eta kn\eps_n^2\delta^{-1}
}
by Bernoulli's inequality for $n$ large enough, so that $4\rho \-s \eps_n\delta^{-1}<1$ and $2\eta k n\eps_n>1$. Invoking \eqref{41} and \eqref{49}, we obtain
\besn{\label{50}
  \IP(L_a= 0) & \leq \IP(A_a^c) + \IP(L_a=0,A_a) \to 0 \qquad \text{as $n\to\infty$,}
}
uniformly in $a \in [n]^{k}$ and $\nu \geq \delta$, which is \eqref{46}.
  
  
\paragraph{Step 3 (Closeness to uniformity).} 
We next prove that
\beqn{ \label{51}
  \EE\big[\abs{I'_a - I''_a }\big]   \to 0 \qquad\text{as $n\to\infty$}.}
To this end, note that
\bes{
  \EE\big[\abs{I'_a - I''_a }\big] 
  & \leq \dtv\bclr{\law(Y^a_{\eps_n}),\law(U^a)}\\
  & \leq k \sup_{i \in [k]}  \dtv\bclr{\law(Y^{a_i}_{\eps_n}),\law(U^{a_i})} \\
  & \quad + \binom{k}{2} \sup_{{i,j \in [k]\colon i \neq j} } 
      \PP\bclr{\text{$Y^{a_i}_{t} = Y^{a_j}_t$ for some $t\in[0,\eps_n]$}}.
}
Indeed, when none of the $k$ coalescing random walks meet, they evolve as independent random walks, and we can use the union bound to estimate the joint total variation distance by the marginal total variation distances.  Since two random walks meet at rate $2\eta$ when they are at opposite ends of an edge and at rate $0$ otherwise, we have
\bes{
  \PP\bclr{\text{$Y^{a_i}_{t} = Y^{a_j}_t$ for some $t\in[0,\eps_n]$}} \leq 1-\ee^{-2\eta\eps_n} \leq 2\eta\eps_n.
}
Moreover, by Lemma~\ref{lem11},
\bes{
  \sup_{\ell \in [n]} \dtv\bclr{\law(Y^\ell_{\eps_n}),\law(U)} \leq \ee^{-\eta\noverlap^2n\eps_n} \leq \ee^{-\eta\delta^2n\eps_n}.
}
Combining these estimates, we obtain
\besn{\label{52}
  \EE\big[\abs{I'_a - I''_a }\big] \leq k\ee^{-\eta \delta^2 n\eps_n} + \binom{k}{2}\,2\eta\eps_n
  \to 0 \qquad\text{as $n\to\infty$}
}
uniformly in $a \in [n]^{k}$ and $\nu(G^n_0)  \geq \delta$, where we invoke \eqref{41}.
This implies \eqref{51}.


\paragraph{Step 4.} 
Since $I_a = I'_a$\/ when $L_a=1$, we can write
\be{
  I_a = I'_a + (I_a - I'_a) (1- L_a).
}
Hence
\besn{ \label{53}
  \Cov(I_a J_a, I_{a'} J_{a'}) 
  & = \Cov\bclr{I'_a J_a, I'_{a'} J_{a'}} + R_1,
}
where $R_1$ is a sum of terms each containing at least one factor of the form $1- L_a$ or $1- L_{a'}$. Hence, by \eqref{50}, 
\beqn{ \label{54}
  \abs{R_1} \to 0 \qquad \text{as $n\to\infty$}
}
uniformly in $a$ and $a'$ and $\nu(G^n_0) \geq \delta$. Now, write $I'_a = I''_a + (I'_a - I''_a)$. Then,
\bes{
  \Cov\bclr{I'_a J_a, I'_{a'} J_{a'}} 
  & = \Cov\bclr{I''_a J_a, I''_{a'} J_{a'}} + R_2,
}
where $R_2$ is a sum of terms each containing at least one factor of the form $I'_a - I''_a$ or $I'_{a'} - I''_{a'}$.  Since the entries of $a,a'$ are disjoint, we have $\Cov\bclr{I''_a J_a, I''_{a'} J_{a'}} = 0$. By \eqref{52},
\beqn{ \label{55}
  \abs{R_2} \to 0 \qquad \text{as $n\to\infty$}
}
uniformly in $a,a'$ and $\nu (G^n_0) \geq \delta$. Using \eqref{54}, \eqref{54} and \eqref{55} we have shown that the term in \eqref{44} goes to $0$. Using essentially the same argument with a similar decomposition as used in this step the one can see that the term in \eqref{45} goes to~$0$.
\end{proof}

\begin{lemma} \label{lem14} 
Fix a coloured graph $F_m$ on $k_m$ vertices, assume $\nu(G) > 0$, and suppose that $(\eps_n)_{n\in\N}$ is such that \eqref{41} holds. Then
 \besn{\label{56}
  \lim_{n\to\infty} \EE\bcls{\abs{t_{F_m}(\kappa^n_0,\bar{c}^n_0) - t_{F_m}(\kappa^n_0,\bar{c}^n_{\eps_n})}} \to 0.}
\end{lemma}

\begin{proof}
For $a,b \in \N_0$, let $f_{a,b} \colon\,[0,1] \rightarrow [0,1]$ be given by $f_{a,b}(c) = c^a(1-c)^b$, $c \in [0,1]$. It is easy to see that
$\sup_{c \in [0,1]} |f'_{a,b}(c)| \leq a + b$. Hence, using the mean value theorem with $a=w(F_m)$ and $b = b(F_m)$ in the above, we have
\bes{
  \EE\big[\babs{ (\bar{c}^n_{\eps_n})^{w(F_m)}(1- \bar{c}^n_{\eps_n})^{b(F_m)} 
  - (\bar{c}^n_0)^{w(F_m)}(1- \bar{c}^n_0)^{b(F_m)}}\big]
  \leq {k_m}\EE\big[\abs{\bar{c}^n_{\eps_n} - \bar{c}^n_0}\big].
}
Using the Cauchy-Schwarz inequality, we have
\bes{
  \EE\big[\abs{\bar{c}^n_{\eps_n} - \bar{c}^n_0}\big] \leq \sqrt{\EE\big[(\bar{c}^n_{\eps_n} - \bar{c}^n_0)^2\big]}.
}
Hence,
\bes{
  \IE\bcls{(\-c^n_{\eps_n}-\-c^n_0)^2}
  = \frac{1}{n^2} \sum_{i,j=1}^n \IE\bcls{\bclr{c^n_{\eps_n}(i)-c^n_0(U^i)}\bclr{c^n_{\eps_n}(j)-c^n_0(U^j)}},
}
where $U^i$ and $U^j$ are independent uniform random variables on $[n]$. For $i\neq j$, 
\bes{
  \mel \IE\bcls{\babs{\bclr{c^n_{\eps_n}(i)-c^n_0(U^i)}\bclr{c^n_{\eps_n}(j)-c^n_0(U^j)}}} \\
  & \leq \IE\bcls{\babs{\bclr{c^n_{0}(Y^i_{\eps_n})-c^n_0(U^i)}\bclr{c^n_{0}(Y^j_{\eps_n})-c^n_0(U^j)}L_{i,j}}} 
  + \IE\cls{(1-L_{i,j})} \\
  & \leq \dtv\bclr{\law(Y^{i}_{\eps_n},Y^{j}_{\eps_n}),\law(U^{i},U^{j})} + \IP\clr{L_{i,j}=0} 
   \to 0 \qquad \text{as $n\to\infty$},
} 
uniformly in $i$ and $j$, by similar arguments as above. Using a trivial upper bound for $i=j$, we obtain
\be{
  \EE\big[\abs{\bar{c}^n_{\eps_n} - \bar{c}^n_0}\big] \to 0 \qquad \text{as $n\to\infty$},
}
and so the claim follows.
\end{proof}

We prove our last key ingredient in the proof of Lemma~\ref{lem10} where pathwise duality to random walks is used.

\begin{lemma}[Closeness to projection]
\label{lem15}
Suppose that $(\eps_n)_{n\in\N}$ is such that \eqref{41} holds. Then, for any $\delta>0$,
\besn{\label{57}
  \lim_{n\to\infty}\sup_{\substack{G\in \cG^n:\\ \noverlap(G)\geq \delta}}
  \EE\bcls{\dsub\bclr{G^n_{\eps_n},\cP(G^n_{\eps_n})} \given G^n_{0}=G} = 0.
}
Furthermore, assume that $\noverlap(G^n_0) \geq \nu_0>0$ for $n$ large enough and for some $\nu_0>0$. Then, for any $T>0$,
\besn{\label{58}
  \lim_{n\to\infty}\sup_{\eps_n\leq s\leq T} \EE\bcls{\dsub\bclr{G^n_{s},\cP(G^n_{s})}} = 0.
}
\end{lemma}

\begin{proof} 
We first prove \eqref{57}. Let $\delta >0$ be given. Denote by  $(\kappa^n,c^n)$ the coloured graphon process induced by $G^n$ and by $(\kappa^n,\bar{c}^n)$ the coloured graphon process induced by $\cP(G^n)$. Fix $G\in\cG_n$ with $v(G)\geq \delta$, and let $(\kappa^n_0,c^n_0)$ be the corresponding coloured graphon induced by~$G$. For $M \in \N$, we have
\ben{ \label{59}
\EE\big[d_\sub\bclr{(\kappa^n_{\eps_n},c^n_{\eps_n}),(\kappa^n_{\eps_n},\bar{c}^n_{\eps_n})}\big]
\leq \sum_{m=1}^M \frac{1}{2^m} \EE\bcls{| t_{F_m}(\kappa^n_{\eps_n},c^n_{\eps_n})-t_{F_m}(\kappa^n_{\eps_n},\bar{c}^n_{\eps_n})|} + \frac{1}{2^M}.
}
Fix $\eps>0$ arbitrary, and $M$ large enough so that $2^{-M} < \eps/2$. For $1 \leq m \leq M$, observe that
\besn{ \label{60}
    \left| t_{F_m}(\kappa^n_{\eps_n},c^n_{\eps_n}) -  t_{F_m}(\kappa^n_{\eps_n},\bar{c}^n_{\eps_n}) \right| & \leq \babs{t_{F_m}(\kappa^n_{\eps_n},c^n_{\eps_n}) -  t_{F_m}(\kappa^n_0,c^n_{\eps_n})}\\
    &+  \left| t_{F_m}(\kappa^n_0,c^n_{\eps_n}) -  t_{F_m}(\kappa^n_0,\bar{c}^n_{0}) \right| \\
    & +  \left| t_{F_m}(\kappa^n_0,\bar{c}^n_{0}) -  t_{F_m}(\kappa^n_0,\bar{c}^n_{\eps_n}) \right| \\
    &+  \babs{  t_{F_m}(\kappa^n_0,\bar{c}^n_{\eps_n}) -t_{F_m}(\kappa^n_{\eps_n},\bar{c}^n_{\eps_n})}. 
} 
Let $k_m$ be the number of vertices in $F_m$. Note that, for any edge, the probability of having changed between time$0$ and time $\eps_n$ is at most $1-\ee^{-\rho \-s\eps_n} \leq \rho \-s\eps_n$ where $\-s = \max\{s_{\mathrm{c},0},s_{\mathrm{d},0}, s_{\mathrm{c},1},s_{\mathrm{d},1}\}$. So we have
\begin{align}\label{61}
&\EE[\babs{t_{F_m}(\kappa^n_{\eps_n},c^n_{\eps_n}) -  t_{F_m}(\kappa^n_0,c^n_{\eps_n})}] +  \EE[ \babs{  t_{F_m}(\kappa^n_0,\bar{c}^n_{\eps_n}) -t_{F_m}(\kappa^n_{\eps_n},\bar{c}^n_{\eps_n})}]\} \nonumber \\ & \hspace{1in}\leq 2 \EE\bcls{| t_{\bar{F}_m}(\kappa^n_{\eps_n}) - t_{\bar{F}_m}(\kappa^n_0) |} \nonumber\\
&\hspace{1in} \leq 2 \binom{{k_m}}{2} \rho\-s\eps_n.
\end{align}
Using \eqref{60} along with \eqref{61} and Lemmas~\ref{lem13} and \ref{lem14}, we see that there is an $n$ large enough such that
\bes{ 
\EE\bcls{ \left| t_{F_m}(\kappa^n_0,c^n_{\eps_n}) -  t_{F_m}(\kappa^n_0,\bar{c}^n_{\eps_n}) \right|} < \frac{\eps}{2M}
}
for any $1 \leq m \leq M$. Consequently, using this and our choice of $M$ in \eqref{59}, we obtain that
\bes{
\EE\bcls{d_\sub\bclr{(\kappa^n_{\eps_n},c^n_{\eps_n}),(\kappa^n_{\eps_n},\bar{c}^n_{\eps_n})}} < \eps
}
for $n$ large enough. As $\eps>0$ was arbitrary, this implies \eqref{57}.
 
We next prove \eqref{58}. Note that, for $\nu_0>0$ and $s >0$,
\bes{
  \mel\EE\bcls{d_\sub\bclr{G^n_{s+\eps_n},\cP(G^n_{s+\eps_n})}} \\
  & = \IE\Big[\EE\big[d_\sub\bclr{G^n_{s+\eps_n},\cP(G^n_{s+\eps_n})} \given G^n_s\big] 
  \I[\noverlap(G^n_s)\geq \nu_0\Big] \\[0.2cm]
  &\qquad +  \IE\Big[\EE\big[d_\sub\bclr{G^n_{s+\eps_n},\cP(G^n_{s+\eps_n})}
  \given 
  G^n_s\big] \I[\noverlap(G^n_s)< \nu_0\Big]\\[0.2cm]
  &\leq \sup_{s\geq 0} \sup_{\substack{G\in \cG^n:\\ \noverlap(G)\geq \nu_0}}\EE\bcls{\dsub\bclr{G^n_{s+\eps_n},\cP(G^n_{s+\eps_n})}\given G^n_{s}=G} +\IP(\noverlap(G^n_s)< \nu_0).
}
Using Lemma~\ref{lem12} and \eqref{57} with $\delta=\nu_0$, we have that if $\nu(G^n_0) \geq \nu_0$ for $n$ large enough, then
\bes{
  \lim_{n\to\infty}\sup_{s\in[0,T]} \EE\big[d_\sub\bclr{(\kappa^n_{s+\eps_n},c^n_{s+\eps_n}),
  (\kappa^n_{s+\eps_n},\bar{c}^n_{s+\eps_n})}\big] 
  = 0
}
for any $T>0$, which implies \eqref{58}.
\end{proof}

\begin{proof}[Proof of Lemma~\ref{lem10}]
Combine Remark \ref{rem1}$(a)$ with \eqref{58} in Lemma~\ref{lem15}.
\end{proof}


\subsection{Proof of Theorem~\ref{thm2}}
\label{62}

Before we present the proof of Theorem~\ref{thm2}, we need one last essential ingredient. Lemma~\ref{lem16} below shows that the projected process converges weakly to the limiting process in the path topology of convergence in measure. The proof of this lemma uses the criteria given by Proposition~\ref{prop1}, and the stochastic differential equations given by Lemma~\ref{lem8} and Lemma~\ref{lem9}. The  the graph connectivity shown in Lemma~\ref{lem12},  closeness to the projection shown in Lemma~\ref{lem15}, and results in \cite[Chapter 8]{EK86} regarding weak convergence via approximating generators are used in an essential way.

\begin{lemma}[Convergence of projection] 
\label{lem16}
On $\~\cW_0$, the processes $\cP(G^n)_{t\geq 0}$ converge weakly to $(\kappa_t,q_t)_{t\geq 0}$ as $n\to\infty$ in the path topology of convergence in measure, where the initial condition $(\kappa_0,q_0)$ of the limiting process is given by $\kappa_0 = \lim_{n\to\infty} G^n_0$ and $q_0 = \lim_{n\to\infty} t_{\wv}(G^n_0) = \lim_{n\to\infty} \-c_0^n$.
\end{lemma}

\begin{proof} 
Since all finite-$n$ processes are c\`adl\`ag, and the limiting process is continuous and hence, c\`adl\`ag as well, we can apply $(c)$ of Proposition~\ref{prop1}. Let $\cF_k$ denote the set of all coloured graphs with at most $k$ vertices, let $d=\abs{\cF_k}$, and let $F_1, \dots, F_d$ be an enumeration of the coloured graphs in $\cF_k$. It is clear that it suffices to prove convergence of finite-dimensional distributions for the family of subgraph densities given by $\cF_k$ for all $k$. Fix $k\geq2$, and let 
\bes{
  \cU_k = \bclc{u \in [0,1]^{\cF_k}\colon  \text{$\exists (\kappa,c)\in\cW_0$ such that $u_F = t_{F}(\kappa,c)$  $\forall\, F\in\cF_k$} }.
}
The set $\cU_k$ is a compact subset of $[0,1]^{\cF_k}$, since it is a continuous image of the compact set~$\cW_0$. The process $(\kappa_t,q_t)_{t\geq 0}$ from Lemma~\ref{lem8} induces, via the subgraph densities $(t_{F_i})_{i\in\cF_k}$, a $d$-dimensional diffusion on $\cU_k$ satisfying the stochastic differential equations of Lemma~\ref{lem9}. Moreover, via Lemma~\ref{lem9}, the action of its generator on a smooth enough function $h$ is given by
\bes{
  (\cA h)(\bu) 
  &= \rho\sum_{i=1}^d \bbbclc{\sum_{\{r,s\} \in E(F_i)}\bbcls{s_{\mathrm{c},0}\bclr{u_{\wv}^2+u_{\bv}^2}
  + 2s_{\mathrm{d},0 }u_{\wv}u_{\bv}}\bclr{u_{F_i(r\not\sim s)} - u_{F_i}}\\
  & \kern10em - e(F_i)\bbcls{s_{\mathrm{c},1}\bclr{u_{\wv}^2+u_{\bv}^2}
  +2s_{\mathrm{d},1}u_{\wv}u_{\bv}} u_{F_i}}h_i(\bu)\\
  &\quad + \eta\, u_{\blackwhiteedge}\sum_{i=1}^d\bbclr{\bclr{w_i u_{\wv}^{-1}- b_iu_{\bv}^{-1}}^2 
  - \bclr{w_i u_{\wv}^{-1}+ b_iu_{\bv}^{-1}}} u_{F_i}h_i(\bu) \\
  &\quad + \eta\, u_{\blackwhiteedge}\sum_{i=1}^d\sum_{j=1}^d \bclr{w_iu_{\wv}^{-1}-b_iu_{\bv}^{-1}}
  \bclr{w_ju_{\wv}^{-1}-b_ju_{\bv}^{-1}} u_{F_i} u_{F_j} h_{i,j}(\bu).
}
Note that, for any coloured $F\in \cF_k$ and any $u\in \cU_k$, 
\ben{\label{63}
  u_{F} = u_{\wv}^{w(F)}u_{\bv}^{b(F)} \sum_{H\in \cC(\-F)} u_{H} \leq u_{\wv}^{w(F)}u_{\bv}^{b(F)}.
}
Hence, it is easy to see that, on the set $\cU_k$, the coefficients of the diffusion given by $\cA$ are bounded and Lipschitz. It is straightforward to extend the coefficients to the whole of $[0,1]^d$ so that they are Lipschitz on all of $[0,1]^d$. This can be done, for instance, by applying \eqref{63} to $u_{F_i}$ in order to extract extra factors $u_{\wv}$ and $u_{\bv}$ as required to cancel the reciprocals of these two respective quantities. We can in fact easily extend the coefficients to all of $\IR^d$ such that they remain bounded and Lipschitz, for instance, by replacing each $u_F$ by $0\vee(1\wedge u_F)$. From \cite[Theorem~1.6, Chapter~8]{EK86}, it follows that the closure of $\cA$ generates a strongly continuous contraction semigroup on $C([0,1]^d)$ with core given by all infinitely differentiable functions on $\IR^d$ with compact support. We are therefore in a position to apply \cite[Corollary~8.4, Chapter~4]{EK86}, to compare the action of $\cA_n$ to that of $\cA$ on subgraph densities, which are dense in $C([0,1]^d)$. Let $h$ be a smooth function on $\IR^d$ with compact support. By means of \cite[(8.12), Chapter~4]{EK86}, we only need to show that
\bes{
  \mel\IE\babs{\bclr{\cA_n (h\circ\bt\circ \cP)}(G^n_t) - (\cA h)\bclr{\bt(\cP(G^n_t))}} \\
  &= \IE\babs{\rho\,\bclr{s_{\mathrm{c},0}{\Delta_{\mathrm{c},0,h}}(G^n_t) + s_{\mathrm{d},0}{\Delta_{\mathrm{d},0,h}}(G^n_t) + s_{\mathrm{c},1}{\Delta_{\mathrm{c},1,h}}(G^n_t) + s_{\mathrm{d},1}{\Delta_{\mathrm{d},1,h}}(G^n_t)} \\
  &\qquad\qquad\qquad\qquad\qquad
  + \eta\,\bclr{\Delta_{\wv,\wv}(G^n_t) + \Delta_{\wv,\bv}(G^n_t) + \Delta_{\bv,\bv}}(G^n_t)}
 \to 0, \qquad 
}
as $n\to\infty$ for all $t\geq 0$. Note, for instance, that
\be{
\abs{\Delta_{\mathrm{c},0,h}(G^n_t)}
\leq \abs{h}_1\sum_{i=1}^d\sum_{\{r,s\}\in E(\bar{F}_i)}
\sum_{H\in \cC(\bar{F}_i)}\bclr{\abs{\Delta_{H(r\not\sim s)}(G^n_t)}+\abs{\Delta_{H}(G^n_t)}}.
}
Since all sums are finite, we only need to estimate quantities of the form $\IE\abs{\Delta_H(G^n_t)}$. However, \eqref{58} of Lemma~\ref{lem15} immediately yields that 
\be{
  \IE\abs{\Delta_H(G^n_t)}\to 0 \qquad\text{as $n\to\infty$,}
}
uniformly over $t\in[\delta,T]$ for each $0<\delta<T$. This implies convergence of the finite-dimensional distributions for all $t>0$, which is enough for convergence in the path topology of convergence in measure by $(c)$ of Proposition~\ref{prop1}.
\end{proof}

\begin{proof}[Proof of Theorem~\ref{thm2}] We use Remark~\ref{rem1} with $X^n=G^n$, $Y^n=\cP(G^n)$ its projection onto $\cW_0$, and $Z=(\kappa_t,c_t)_{t\geq 0}$ the limiting process. Lemma~\ref{lem10} implies that $\displaystyle \lim_{n \rightarrow \infty}\IE[\dMZ(X^n,Y^n)] = 0$, while Lemma~\ref{lem16} implies that $Y^n$ converges weakly to~$Z$. The convergence of the finite dimensional distributions follows from the argument at the end of Lemma~\ref{lem16}, \eqref{58} from Lemma~\ref{lem15}, and Slutsky's Theorem.
\end{proof}


\appendix



\section{Sketch of the derivation of the evolution of subgraph densities}
\label{sec17}

In this appendix we illustrate the types of calculations that are needed to derive the evolution of subgraph densities in \eqref{24} from \eqref{29}. Detailed notes can be obtained from the authors upon request. We only consider the drift term arising from vertex dynamics, which is the simplest object. We may restrict ourselves to the case $d=1$. Let $F$ be a coloured graph with $k$ vertices. Recalling the notation \eqref{25}--\eqref{28}, we have
\bes{
  \sum_{1\leq u\leq n}r^{\mathrm{v}}_u(G) \D_{u}t^{\inj}_{F}(G) 
  & = \frac{1}{(n)_k}\sum_{u}\sum_{v:v\neq u}\^c_u e_{uv}c_v\sump_{u_1,\dots,u_k} \chi^{F,1}_{u_1,\dots,u_k}(u)\\
  &\qquad + \frac{1}{(n)_k}\sum_{u}\sum_{v:v\neq u}c_u  e_{uv} \^c_v\sump_{u_1,\dots,u_k} \chi^{F,1}_{u_1,\dots,u_k}(u)\\
  & = R_{1}+R_{2}.
}
Observe that
{\small\ba{
  R_{1} &=\frac{1}{(n)_k}\sum_{u}\sum_{v:v\neq u}\^c_u e_{uv}c_v\sump_{u_1,\dots,u_k} \begin{cases}
       & \text{if $e_{ij}=1\hence  e_{u_iu_j}=1$ $\forall 1\leq i<j\leq k$,} \\
   +1  & \text{if  $\exists i$ s.t.\  $u_i=u$ with $c_{u_i}=1-c_i$, and}\\
       & \text{if $c_i=c_{u_i}$ $\forall 1\leq i\leq k$ s.t.\  $u_i\neq u$;}\\[2ex]
       & \text{if $e_{ij}=1\hence  e_{u_iu_j}=1$ $\forall 1\leq i<j\leq k$,} \\
    -1   & \text{if $\exists i$ s.t.\  $u_i=u$, and} \\
       & \text{if $c_i=c_{u_i} $ $\forall 1\leq i\leq k$;}\\[2ex]
     0 & \text{else.}
   \end{cases}\\
  &=\frac{1}{(n)_k}\sum_{u}\sum_{v:v\neq u}\^c_ue_{uv}c_v\sump_{u_1,\dots,u_k} \begin{cases}
       & \text{if $e_{ij}=1\hence  e_{u_iu_j}=1$ $\forall 1\leq i<j\leq k$,} \\
    + 1& \text{if $\exists i$ s.t.\  $u_i=u$ and $c_i=1$, and}\\
       & \text{if $c_i=c_{u_i}$ $\forall 1\leq i\leq k$ s.t.\  $u_i\neq u$;}\\[2ex]
       & \text{if $e_{ij}=1\hence  e_{u_iu_j}=1$ $\forall 1\leq i<j\leq k$,} \\
    -1   & \text{if  $\exists i$ s.t.\  $u_i=u$ and $c_i=0$, and} \\
       & \text{if $c_i=c_{u_i}$ $\forall 1\leq i\leq k$ s.t.\  $u_i\neq u$;}\\[2ex]
     0 & \text{else.}
   \end{cases}\\
  &=\frac{1}{(n)_k}\sum_{u}\sum_{v:v\neq u}\^c_ue_{uv}c_v\sump_{u_1,\dots,u_k} \begin{cases}
      & \text{if $e_{ij}=1\hence  e_{u_iu_j}=1$ $\forall 1\leq i<j\leq k$,} \\
     1  & \text{if $\exists i$ s.t.\  $u_i=u$ and $c_i=1$, and}\\
       & \text{if $c_i=c_{u_i}$ $\forall 1\leq i\leq k$ s.t.\  $u_i\neq u$;}\\[2ex]
     0 & \text{else.}
   \end{cases}\\[2ex]
  & \quad - \frac{1}{(n)_k}\sum_{u}\sum_{v:v\neq u}\^c_ue_{uv}c_v\sump_{u_1,\dots,u_k} \begin{cases}
        & \text{if $e_{ij}=1\hence  e_{u_iu_j}=1$ $\forall 1\leq i<j\leq k$,} \\
       1  & \text{if $\exists i$ s.t.\  $u_i=u$ and $c_i=0$, and} \\
         & \text{if $c_i=c_{u_i}$ $\forall 1\leq i\leq k$ s.t.\  $u_i\neq u$;}\\[2ex]
       0 & \text{else.}
     \end{cases}\\
     &= R_{11}-R_{12}.
}}
Since $e_{uu}=0$ for all $u$, we may replace $\sum_{v:v\neq u}$ by $\sum_{v}$. Thus,
{\small\ba{
R_{11}&=\frac{1}{(n)_k}\sum_{u}\sum_{v}\^c_ue_{uv}c_v\sump_{u_1,\dots,u_k} \begin{cases}
      & \text{if $e^F_{ij}=1\hence  e^F_{u_iu_j}=1$ $\forall 1\leq i<j\leq k$,} \\
     1  & \text{if $\exists i$ s.t.\  $u_i=u$ and $c_i=1$, and}\\
       & \text{if $c_i=c_{u_i}$ $\forall 1\leq i\leq k$ s.t.\  $u_i\neq u$;}\\[2ex]
     0 & \text{else.}
   \end{cases}\\
& =\frac{1}{(n)_k}\sump_{u_1,\dots,u_k}\sum_{u}\sum_{v}\^c_ue_{uv}c_v \begin{cases}
      & \text{if $e^F_{ij}=1\hence  e^F_{u_iu_j}=1$ $\forall 1\leq i<j\leq k$,} \\
     1  & \text{if $\exists i$ s.t.\  $u_i=u$ and $c_i=1$, and}\\
       & \text{if $c_i=c_{u_i}$ $\forall 1\leq i\leq k$ s.t.\  $u_i\neq u$;}\\[2ex]
     0 & \text{else.}
   \end{cases}\\
&= \sum_{l=1}^k\frac{1}{(n)_k}\sump_{u_1,\dots,u_k}\sum_{v}\^c_{u_l}e_{u_lv}c_v \begin{cases}
      & \text{if $e^F_{ij}=1\hence  e^F_{u_iu_j}=1$ $\forall 1\leq i<j\leq k$,} \\
     1  & \text{if $c^F_l=1$, and}\\
       & \text{if $c^F_i=c_{u_i}$ $\forall 1\leq i\leq k$ s.t.\  $u_i\neq u_l$;}\\[2ex]
     0 & \text{else.}
   \end{cases}\\
&= \sum_{p:c_p^F=1}\frac{1}{(n)_k}\sump_{u_1,\dots,u_k}\sum_{v}e_{u_pv}c_v \begin{cases}
      & \text{if $e^F_{ij}=1\hence  e^F_{u_iu_j}=1$ $\forall 1\leq i<j\leq k$,} \\
     1  & \text{if $c_{u_p}=0$, and} \\
       & \text{if $c^F_i=c_{u_i}$ $\forall 1\leq i\leq k$ s.t.\  $u_i\neq u_p$;}\\[2ex]
     0 & \text{else.}
   \end{cases}\\
&= \sum_{p:c_p^F=1}\frac{1}{(n)_k}\sump_{u_1,\dots,u_{k+1}}\begin{cases}
      & \text{if $e^F_{ij}=1\hence  e^F_{u_iu_j}=1$ $\forall 1\leq i<j\leq k$,} \\
     1  & \text{if $e_{u_pu_{k+1}}=1$, if $c_{u_p}=0$, if $c_{u_{k+1}}=1$, and}\\
       & \text{if $c_{u_i}=c^F_i$ $\forall 1\leq i\leq k$ s.t.\  $u_i\neq u_p$;}\\[2ex]
     0 & \text{else.}
   \end{cases}\\
&\quad+ \sum_{p:c_p^F=1}\sum_{q:q\neq p}\frac{1}{(n)_k}\sump_{u_1,\dots,u_k} \begin{cases}
       & \text{if $e^F_{ij}=1\hence  e^F_{u_iu_j}=1$ $\forall 1\leq i<j\leq k$,} \\
     1  & \text{if $e_{u_pu_q}=1$, if $c_{u_p}=0$, if $c_{u_q}=1$, and}\\
       & \text{if $c_{u_i}=c^F_i$ $\forall 1\leq i\leq k$ s.t.\  $u_i\neq u_p$;}\\[2ex]
     0 & \text{else.}
   \end{cases}\\
&= \sum_{\substack{p\in V(F):\\c_p^F=1}}(n-k)t^{\inj}\bclr{F(c_p\gets0,+ (k+1),c_{k+1}\gets1,p\sim k+1);G}\\
&\kern10em+ \sum_{\substack{p\in V(F):\\c_p^F=1}}\sum_{\substack{q\in V(F):\\q\neq p, c_q^F=1}}t^{\inj}\bclr{F(c_p\gets0,p\sim q);G}.
}}
Similarly,
{\small\bes{
R_{12}&=\sum_{\substack{p\in V(F):\\c_p^F=0}}(n-k)t^{\inj}\bclr{F(+ k+1,c_{k+1}\gets1,p\sim k+1);G}\\
&\kern10em+  \sum_{\substack{p\in V(F):\\c_p^F=0}}\sum_{\substack{q\in V(F):\\q\neq p, c_q^F=1}}t^{\inj}\bclr{F(p\sim q);G}.
}}
Collecting the terms, we arrive at
{\small\bes{
  R_{1} 
  & = (n-k)\sum_{\substack{p\in V(F):\\c_p^F=1}}t^{\inj}\bclr{F(c_p\gets0,+ (k+1),c_{k+1}\gets1,p\sim k+1);G}\\
  & \kern8em+ \sum_{\substack{p\in V(F):\\c_p^F=1}}\sum_{\substack{q\in V(F):\\q\neq p, c_q^F=1}}t^{\inj}\bclr{F(c_p\gets0,p\sim q);G} \\
  & \kern8em - (n-k)\sum_{\substack{p\in V(F):\\c_p^F=0}}t^{\inj}\bclr{F(+(k+1),c_{k+1}\gets1,p\sim k+1);G}\\
  & \kern8em - \sum_{\substack{p\in V(F):\\c_p^F=0}}\sum_{\substack{q\in V(F):\\q\neq p, c_q^F=1}}t^{\inj}\bclr{F(p\sim q);G}.
}}
In the same manner, we have
{\small\bes{
  R_{2}  
  & = (n-k)\sum_{\substack{p\in V(F):\\c_p^F=0}}t^{\inj}\bclr{F(c_p\gets1, +(k+1),c_{k+1}\gets0,p\sim k+1);G}\\
  &\kern8em+ \sum_{\substack{p\in V(F):\\c_p^F=0}}\sum_{\substack{q\in V(F):\\q\neq p, c_q^F=0}}t^{\inj}\bclr{F(c_p\gets1,p\sim q);G} \\
  & \kern8em - (n-k)\sum_{\substack{p\in V(F):\\c_p^F=1}}t^{\inj}\bclr{F(+(k+1),c_{k+1}\gets0,p\sim k+1);G}\\
  &\kern8em- \sum_{\substack{p\in V(F):\\c_p^F=1}}\sum_{\substack{q\in V(F):\\q\neq p, c_q^F=0}}t^{\inj}\bclr{F(p\sim q);G}.
}}
Combining the above results, we obtain
{\small\bes{
  R_1 & = R_{11}+R_{12} \\
  & = (n-k)\bbbbclr{\,\sum_{\substack{p\in V(F)}}t^{\inj}\bclr{F(+(k+1),c_{k+1}\gets c_p, p\sim k+1, c_p\gets1-c_p);G}\\
  & \qquad\qquad\qquad - \sum_{\substack{p\in V(F)}}t^{\inj}\bclr{F(+(k+1),c_{k+1}\gets1-c_p,p\sim k+1);G}}\\
  & \quad + \sum_{\substack{p\in V(F)}}\sum_{\substack{q\in V(F):\\q\neq p, c_q^F=c_p^F}}t^{\inj}\bclr{F(c_p\gets1-c_p,p\sim q);G} \\
  & \qquad\qquad\qquad - \sum_{\substack{p\in V(F)}}\sum_{\substack{q\in V(F):\\q\neq p, c_q^F\neq c_p^F}}t^{\inj}\bclr{F(p\sim q);G}.
}}


\section*{Acknowledgements} 
SA was supported through a Knowledge Exchange Grant, under project no.\ RTI4001 of the Department of Atomic Energy, Government of India. FdH was supported through NWO Gravitation Grant NETWORKS-024.002.003. AR was supported through Singapore Ministry of Education Academic Research Fund Tier 2 grant MOE2018-T2-2-076. The authors thank the International Centre for Theoretical Sciences (ICTS) for hospitality during the ICTS-NETWORKS workshop ``Challenges in Networks'' in January 2024, and during a research visit in September 2024. We sincerely thank Svante Janson and Fiona Skerman for their careful reading of Lemma 2.2 and their detailed and insightful feedback, which significantly improved and corrected the exposition in that section.  AR thanks Rongfeng Sun for helpful discussions.




\Addresses

\end{document}